\documentclass[11pt]{amsart}
\usepackage{graphicx} 
\usepackage{tikz} 

\usepackage{amssymb, amsfonts, amsmath, amsthm, mathabx, bbm}

\usepackage[colorlinks=true,linkcolor=blue,citecolor=blue]{hyperref}

\usepackage[colorinlistoftodos,prependcaption,textsize=small]{todonotes}

\title{Second order free cumulants: product, commutator, and anti-commutator}

\author{Daniel Munoz George}

\address{Daniel Munoz George: Department of Mathematics, The University of Hong Kong, PokFuLam, Hong Kong}

\email{dmunozgeorge@gmail.com}

\author{Daniel Perales}

\address{Daniel Perales: Department of Mathematics, University of Notre Dame, IN, USA}

\email{dperale2@nd.edu}

\date{\today}

\newtheorem{theorem}{Theorem}[section]
\newtheorem{corollary}[theorem]{Corollary}
\newtheorem{lemma}[theorem]{Lemma}
\newtheorem{proposition}[theorem]{Proposition}
\theoremstyle{definition}
\newtheorem{definition}[theorem]{Definition}
\newtheorem{remark}[theorem]{Remark}
\newtheorem{notation}[theorem]{Notation}
\newtheorem{example}[theorem]{Example}

\usepackage{enumitem}

\newcommand{\gammanm}{\gamma}
\newcommand{\Inm}{I}

\newcommand{\cP}{\mathcal{P}}
\newcommand{\NC}{\mathcal{NC}}
\newcommand{\cU}{\mathcal{U}}
\newcommand{\cV}{\mathcal{V}}
\newcommand{\PS}{\mathcal{PS}}
\newcommand{\ab}{\allowbreak}

\newcommand{\cc}{\mathbb{C}}

\newcommand{\nn}{\mathbb{N}}

\renewcommand{\AA}{\mathcal{A}}

\newcommand{\CC}{\mathcal{C}}

\newcommand{\EE}{\mathcal{E}}
\newcommand{\FF}{\mathcal{F}}
\newcommand{\GG}{\mathcal{G}}

\newcommand{\JJ}{\mathcal{J}}
\newcommand{\KK}{\mathcal{K}}

\newcommand{\NN}{\mathcal{N}}
\newcommand{\OO}{\mathcal{O}}
\newcommand{\PP}{\mathcal{P}}

\newcommand{\XX}{\mathcal{X}}

\newcommand{\NCac}{\XX }

\newcommand{\Ssemicircle}{\mathcal{NC}_2^{spoke}}

\newcommand{\freec}[2]{\kappa_{#1}^{#2}}
\newcommand{\Part}[1]{P_{#1}}

\usepackage{calc}
\newcounter{PartitionDepth}
\newcounter{PartitionLength}

\usepackage{forest}
\forestset{
  decor/.style = {
    label/.expanded = {[inner sep = 0.2ex, font=\unexpanded{\tiny}]right:{$#1$}}
  },
  root/.style = {minimum size = 0.1ex},
  decorated/.style = {
    for tree = {
      circle, fill, inner sep = 0.3ex, minimum size = 1.ex,
      grow' = south, l = 0, l sep = 1.2ex, s sep = 0.7em,
      fit = tight, parent anchor = center, child anchor = center,
      delay = {decor/.option = content, content =}
    }
  },
  default preamble = {decorated, root},
}

\usepackage{fullpage}

\begin{document}

\begin{abstract}
Given two second order free random variables $a$ and $b$, we study the second order free cumulants of their product $ab$, their commutator $ab-ba$, and their anti-commutator $ab+ba$.  Let $(\kappa_n^a)_{n\geq 1}$ and $(\kappa_{n,m}^a)_{n,m\geq 1}$ denote the sequence of free cumulants of first and second order, respectively, of a random variable $a$ in a second order non-commutative probability space $(\mathcal{A},\varphi,\varphi^2)$. Given $a$ and $b$ two second order freely independent random variables, we provide formulas to compute each of the cumulants $(\freec{n,m}{ab})_{n,m\geq 1}$, $(\freec{n,m}{ab-ba})_{n,m\geq 1}$, and $(\freec{n,m}{ab+ba})_{n,m\geq 1}$ in terms of the individual cumulants $(\freec{n}{a})_{n\geq 1}$, $(\freec{n,m}{a})_{n,m\geq 1}$, $(\freec{n}{b})_{n\geq 1}$, and $(\freec{n,m}{b})_{n,m\geq 1}$. For $n=m=1$ our formulas read:
\begin{align*}
\freec{1,1}{ab} &= \freec{2}{a}\freec{2}{b} +\freec{1,1}{a}(\freec{1}{b})^2+\freec{1,1}{b}(\freec{1}{a})^2,\\
\freec{1,1}{ab-ba} &= 2\freec{2}{a}\freec{2}{b},\\
\freec{1,1}{ab+ba} &= 2\freec{2}{a}\freec{2}{b} +4\freec{1,1}{a}(\freec{1}{b})^2+4\freec{1,1}{b}(\freec{1}{a})^2.
\end{align*}

In general, our formulas express the cumulants $\freec{n,m}{ab}$, $\freec{n,m}{ab-ba}$, and $\freec{n,m}{ab+ba}$ as sums indexed by special subsets of non-crossing partitioned permutations. The formulas for the commutator and anti-commutator where not studied before, while the formula for the product was only known in the case the where the individual second order free cumulants vanish. As an application, we compute explicitly the cumulants of the anti-commutator and product of two second order free semicircular variables. 
\end{abstract}

\maketitle

\tableofcontents

\section{Introduction}

$\ $

Free probability is a useful tool to study large random matrices. By now there is a extensive list of results that confirm that independent random matrices tend to free random variables when the size of the matrix is large \cite{voiculescu1991limit, voiculescu1992free, voiculescu1998strengthened}. Second order freeness, initiated in \cite{mingo2007secondPart1} helps to get a more detailed study, by extending the relation between random matrices and free probability theory from the level of expectations to the level of
fluctuations. Since then, different works have shown that some classical ensembles of independent random matrices become second order free when the size of the matrix tends to infinity. In \cite{mingo2007secondPart1}, it was shown that orthogonal families of Gaussian and Wishart random matrices are asymptotically free of second order. In the same direction, it was shown in \cite{mingo2007second} that Haar unitary and independent random matrices with a second order distribution are second order free. Asymptotic second order freeness is not generally satisfied for real ensembles of random matrices, this motivated the introduction of \textit{real second order freeness} in \cite{redelmeier2014}. Later, it was shown in \cite{mingopopa2013} that independent and Haar orthogonal random matrices are asymptotically real second order free.

An important question is how one can understand the distribution of polynomial expression $P(A,B)$ of two independent random matrices $A$ and $B$, in terms of the distributions of $A$ and $B$ itself. Even in the limiting case, that is usually simpler, this is not an easy question. In the limit the problem amounts to study the distribution of $P(a,b)$ in terms of the distributions of $a$ and $b$, two free random variables.
The basic cases of addition $a+b$, and multiplication $ab$, are very well understood. However, other simple polynomials, such as the commutator $i(ab-ba)$, or the anti-commutator $ab+ba$ are considerably harder.

From a combinatorial point of view, the problem reduces to computing the free cumulants of the polynomial expression $\left(\freec{n}{{}_{P(a,b)}}\right)_{n\geq 1}$ in terms of the free cumulants of each of the variables, $(\freec{n}{a})_{n\geq 1}$ and $(\freec{n}{b})_{n\geq 1}$. The free cumulants were introduced by Speicher \cite{speicher1994multiplicative} as functionals that linearize the addition: $\freec{n}{a+b}=\freec{n}{a}+\freec{n}{b}$ for all $n$. Since then, free cumulants have been also used to study the multiplication \cite{nica1996multiplication}, the commutator \cite{nica1998commutators}, and the anti-commutator \cite{perales2021anti}. In the case of the anti-commutator, the formula requires the study of graphs that are associated to each partition.

\begin{definition}[{\cite[Definition 1.1]{perales2021anti}}]
\label{defi.Graph.pi}
Given a set partition $\pi$ of $[2n]:=\{1,2,\dots,2n\}$ one can associate a graph $\GG_\pi$, where the vertices are blocks of $\pi$, and for $k=1,2,\dots,n$ we draw an undirected edge between the block containing element $2k-1$ and the block containing element $2k$. We allow for loops and multiple edges, thus $\GG_\pi$ has exactly $n$ edges.  

Letting $\NN\CC(2n)$ be the set of non-crossing partitions of $[2n]$, we denote
\[
\NCac_{2n}:=\{\pi\in \NN\CC(2n) : \GG_\pi\mbox{ is connected and bipartite}\}.
\]
A partition $\pi \in \NCac_{2n}$ has a natural \emph{bipartite decomposition} $\pi = \pi ' \sqcup \pi ''$.  Denoting by $V_1$ the block  of $\pi$ which contains the number $1$, we have that $\pi '$ consists of the blocks of $\pi$ which are at even distance from $V_1$ in the graph $\mathcal{G}_{\pi}$, while $\pi''$ consists of the blocks of $\pi$ which are at odd distance from $V_1$ in that graph.
\end{definition}

With this notation in hand, \cite[Theorem1.4]{perales2021anti} expresses the cumulants of the anti-commutator $ab+ba$ in terms of the cumulants of each $a$ and $b$:

\begin{equation}
\label{formula.anticommutator.main}
\freec{n}{ab+ba} = \sum_{\substack{\pi\in \NCac_{2n} \\ \pi=\pi'\sqcup\pi''}} \left(  \freec{\pi'}{a}  \freec{\pi''}{b} +  \freec{\pi'}{b}  \freec{\pi''}{a} \right), \qquad \text{for } n\geq 1,
\end{equation}
where for a variable $c$ and partition $\sigma$ we use the notation 
\begin{equation}
\label{eq:product.partition}
\freec{\sigma}{c}:=\prod_{V\in \sigma} \freec{|V|}{c}.
\end{equation}

Understanding the second order distribution of polynomial expressions $P(a,b)$ of two second order free random variables $a$ and $b$ is even more challenging. In the last of a series of papers concluding with \cite{collins2007second}, the authors introduced the concept of the second (and even higher) order cumulants. In the second order framework, the question is how to compute the second order free cumulants of the polynomial $\left(\freec{n,m}{{}_{P(a,b)}}\right)_{n,m\geq 1}$ in terms of the individual free cumulants sequences $(\freec{n}{a})_{n\geq 1}$, $(\freec{n,m}{a})_{n,m\geq 1}$, $(\freec{n}{b})_{n\geq 1}$ and $(\freec{n,m}{b})_{n,m\geq 1}$. Notice that the problem is already more complicated, as one may need to involve the free cumulants of first order. Some of the few results in this direction include a formula for the second order cumulants of products as arguments \cite{mingo2009second}, and a computation of the second order cumulants of even and $R$-diagonal variables \cite{arizmendi2023second}.\\

The goal of the present work, is to generalize \eqref{formula.anticommutator.main} to the second order. Our study parallels the one done in \cite{perales2021anti} for the first order version. Instead of set partitions, the main combinatorial object in second order freeness is the set of non-crossing annular partitioned permutations $PS_{\NC}(n,m):=S_{\NC}(n,m)\cup S_{\NC}^\prime(n,m)$, see Definition \ref{Definition:Non-crossing-pp}. It is important to keep in mind that every permutation $\pi\in S_n$ can be seen as a partition $\Part{\pi}$ by considering the cycles (of its cycle decomposition) as the blocks of the partition.

\begin{definition}
\label{defi.Graph.pi.2}
\label{main.defi.2}
Let $S_{2n}$ and be the set of permutations of $[2n]:=\{1,2,\dots,2n\}$. 
\begin{enumerate}
\item[1.] Given a permutation $\pi\in S_{2n}$, we denote by $G_{\pi}:=G_{\Part{\pi}}$ the graph from Definition \ref{defi.Graph.pi} associated to the partition $\Part{\pi}$, whose blocks are the cycles of $\pi$.

\item[2.] Let $\gammanm:=(1,\dots,2n)(2n+1,\dots,2n+2m)\in S_{2n+2m}$. We denote
\[
\JJ_{2n,2m}:=\{\pi\in S_{\NC}(2n,2m): \GG_\pi\mbox{ is bipartite and }\pi^{-1}\gammanm\mbox{ separates even}\}.
\]
A permutation $\pi \in \JJ_{2n,2m}$ has a \emph{bipartite decomposition} $\pi = \pi ' \sqcup \pi ''$, as in Definition \ref{defi.Graph.pi}.

\item[3.] We denote
\[
\NCac_{2n,2m}:=\{(\cU,\pi)\in S_{\NC}^\prime(2n,2m): \GG_\cU\mbox{ is bipartite and connected}\}.
\]
A partitioned permutation $(\cU,\pi)\in \NCac_{2n,2m}$ has a natural \emph{decomposition} $\pi = \pi ' \sqcup \pi '' \sqcup A\sqcup B$. If $\pi=\pi_1\times \pi_2$ then $A\in\pi_1$ and $B\in\pi_2$ are the only two cycles of $\pi$ that are merged together into a single block $U$ of $\cU$. On the other hand, $\pi'$ consists of the cycles of $\pi\setminus\{A,B\}$ which are at even distance from $U$ in the graph $\GG_\cU$ while $\pi''$ consist of the cycles which are at odd distance from $U$ in the graph $\GG_\cU$.
\end{enumerate}
\end{definition}

As advertised, our main result is to provide a general formula to compute the second order free cumulants of the anti-commutator $ab+ba$ in terms of the first and second order free cumulants of $a$ and $b$.

\begin{theorem}
\label{Thm.anticommutator.main.2}
Consider two second order free random variables $a$ and $b$, and let $(\freec{n}{a})_{n\geq 1}$, $(\freec{n,m}{a})_{n,m\geq 1}$, $(\freec{n}{b})_{n\geq 1}$, $(\freec{n,m}{b})_{n,m\geq 1}$ and $(\freec{n,m}{ab+ba})_{n,m\geq 1}$ be the sequence of first and second order free cumulants of $a$, $b$ and $ab+ba$, respectively. Then, for every $n,m\geq 1$ one has
\begin{align}
\label{formula.anticommutator.main.2.intro}
\freec{n,m}{ab+ba} &= \sum_{\substack{\pi \in \JJ_{2n,2m}\\ \pi=\pi' \sqcup \pi''}}\left( \freec{\pi'}{a} \, \freec{\pi''}{b}  + \freec{\pi'}{b} \, \freec{\pi''}{a}  \right) + \sum_{\substack{(\cU,\pi)\in \NCac_{2n,2m}\\ \pi =\pi'\sqcup \pi'' \sqcup A\sqcup B }} \left(\freec{|A|,|B|}{a}\, \freec{\pi'}{a} \, \freec{\pi''}{b} + \freec{|A|,|B|}{b}\, \freec{\pi'}{b} \, \freec{\pi''}{a}  \right),
\end{align}
where we use the notation \eqref{eq:product.partition}.
\end{theorem}

The first two formulas look as follows:
\begin{align*}
\freec{1,1}{ab+ba} &= 2\freec{2}{a}\freec{2}{b} +4\freec{1,1}{a}\freec{1}{b}\freec{1}{b}+4\freec{1,1}{b}\freec{1}{a}\freec{1}{a},\\
\freec{1,2}{ab+ba}=\freec{2,1}{ab+ba} &=\, 4\freec{3}{a}\freec{3}{b} + 12\freec{1}{a}\freec{2}{a}\freec{3}{b}+12\freec{1}{b}\freec{2}{b}\freec{3}{a}
+4\freec{2,1}{a}\freec{2}{b}\freec{1}{b}+4\freec{2,1}{b}\freec{2}{a}\freec{1}{a}\\
&+8\freec{2,1}{a}(\freec{1}{b})^3+8\freec{2,1}{b}(\freec{1}{a})^3
+16\freec{1,1}{a}\freec{1}{a}\freec{2}{b}\freec{1}{b}+16\freec{1,1}{b}\freec{1}{b}\freec{2}{a}\freec{1}{a}.   
\end{align*}

The proof of Theorem \ref{Thm.anticommutator.main.2} can be found in Section \ref{Section: proof}. The approach relies on first performing standard computations using linearity of the cumulants and products as arguments formula, to get a sum indexed by partitions $\pi$, this is the content of Proposition \ref{prop.basic.anticomm.2}. Then, in Proposition \ref{prop.few.nonvanishing.epsilons.2} we use the graph $\GG_\pi$ to detect several permutations $\pi$ that do not really contribute to the sum. Finally, one can rewrite the second sum of Equation \eqref{formula.anticommutator.main.2.intro} as a double sum over $\pi=\pi_1\times\pi_2$ and $(C_2,C_2)\in \pi_1\times \pi_2$ where both $\GG_{\pi_1}$ and $\GG_{\pi_2}$ are connected and bipartite. This allows us to restate the formula in terms of the sets $\NCac_{2n}$ and $\NCac_{2m}$ from Definition \ref{defi.Graph.pi}, this is done in Theorem \ref{Thm.anticommutator.main.2.2}.\\

The ideas used in the proof of Theorem \ref{Thm.anticommutator.main.2} can be adapted to study the cumulants of the free commutator $ab-ba$. The only difference is that now the terms in the sum may have negative signs depending on permutation $\pi$, as we may get the same term with opposite signs, this means that there might be several cancellations. 

\begin{definition}
\label{def:admissible}
Let $I:=(1,2)(2,3)\dots (2n+2m-1,2n+2m)\in S_{2n+2m}$. A permutation $\pi\in\JJ_{2n,2m}$ is said to be admissible if $I\pi$ separates $2i-1$ from $2i$ for all $1\leq i \leq n+m$. In other words, there is no cycle of $I\pi$ containing both $2i-1$ and $2i$. We denote the set of admissible permutations by $\AA_{2n,2m}$.
\end{definition}

\begin{theorem}
\label{Thm.commutator.main.2.intro}
Consider two second order free random variables $a$ and $b$, and let $(\freec{n}{a})_{n\geq 1}$, $(\freec{n,m}{a})_{n,m\geq 1}$, $(\freec{n}{b})_{n\geq 1}$, $(\freec{n,m}{b})_{n,m\geq 1}$ and $(\freec{n,m}{ab-ba})_{n,m\geq 1}$ be the sequence of first and second order free cumulants of $a$, $b$ and $ab-ba$, respectively. Then, for every $n,m\geq 1$ one has
\begin{align*}
\freec{n,m}{ab-ba} =& \sum_{\substack{\pi \in \AA_{2n,2m}\\ \pi=\pi' \sqcup \pi''}}s(\pi) \left(\freec{\pi'}{a} \, \freec{\pi''}{b}  + (-1)^{m+n} \freec{\pi'}{b} \, \freec{\pi''}{a}  \right) \\ & + \sum_{\substack{(\cU,\pi)\in \NCac\EE_{2n,2m}\\ \pi =\pi'\sqcup \pi'' \sqcup A\sqcup B }} \left(\freec{|A|,|B|}{a}\, \freec{\pi'}{a} \, \freec{\pi''}{b} +  \freec{|A|,|B|}{b}\, \freec{\pi'}{b} \, \freec{\pi''}{a}  \right),
\end{align*}
where $\NCac\EE_{2n,2m}$ is the set of permutations $\pi\in\NCac_{2n,2m}$ such that every cycle of $\pi$ has even size. And $s(\pi)=\pm 1$ is a sign depending on $\pi$ that will be precisely defined in \eqref{not:sign.pi} at the beginning of Section \ref{sec:commutator}.
\end{theorem}

The first three formulas look as follows:
\begin{align*}
\freec{1,1}{ab-ba} &= 2\freec{2}{a}\freec{2}{b},\\
\freec{2,1}{ab-ba} = \freec{1,2}{ab-ba}& =0,\\
\freec{2,2}{ab-ba} &= 4\freec{4}{a}\freec{4}{b}+12(\freec{2}{a}\freec{2}{b})^2 +12\freec{4}{a}(\freec{2}{b})^2+12\freec{4}{b}(\freec{2}{a})^2+4\freec{2,2}{a}(\freec{2}{b})^2+4\freec{2,2}{b}(\freec{2}{a})^2, 
\end{align*}

\begin{remark}
One can check that the second sum in Theorem \ref{Thm.commutator.main.2.intro} vanishes completely, unless $n$ and $m$ are both even. Notice also that the second sum is cancellation free, as there are only positive signs. This property is basically inherited from the first order case. On the contrary, the first sum in Theorem \ref{Thm.commutator.main.2.intro} is not cancellation-free, in Section \ref{ssec:examples.commutator} we will see how in the computation of $\freec{2,1}{ab-ba}$, one needs to cancel some terms associated to admissible permutations in $\AA_{4,2}$.

It is not hard to check that the set $\AA_{2n,2m}$ of admissible permutations, which indexes the first sum, does not contain any permutations with fixed points (cycles of size 1). Since such terms do not appear in the second sum either, then this means that the right hand side does not contain any terms of the form $\freec{1}{a}$ or $\freec{1}{b}$ at all. Therefore, the commutator $ab-ba$ does not depends on the expected values of $a$ and $b$, meaning that one can always consider centered variables.

However, the set $\AA_{2n,2m}$ of admissible permutations may contain odd cumulants $\freec{2n+1}{a}$ or $\freec{2n+1}{b}$ with $n\geq 1$. In fact, in Remark \ref{rem:odd.cumulants.appear} we will see that the formula for $\freec{4,2}{ab-ba}$ already contains a term of the form $8(\freec{3}{a})^2(\freec{3}{b})^2$ that does not cancel. To put things into perspective, recall that  in the first order case the formula for the commutator only involves even cumulants. Thus, the dependence of the commutator on odd cumulants is a new phenomenon appearing only in the second order case. Ultimately, this is one the reasons why finding an indexing set in the first sum that does not have cancellations is a challenging task. 
\end{remark}

The same ideas used in the proof of Theorem \ref{Thm.anticommutator.main.2} can be adapted to study the second order free cumulants of the product of two second order free variables. Recall that the formula in the first order was derived in \cite{nica1996multiplication}:
\begin{equation}\label{Equation: free cumulants of product of free variables.intro}
\kappa_n(ab)=\sum_{\pi\in \NC(n)}\freec{\pi}{a}\, \freec{Kr_n(\pi)}{b},
\end{equation}
where $Kr_n(\pi)$ is Kreweras complement of a non-crossing partition in the $n$-disk. Regarding a partition as a permutation the Kreweras complement is defined as the permutations $Kr_n(\pi):=\pi^{-1}\gamma_n$ where $\gamma_n:=(1,\dots,n)\in S_n$. Similarly, the Kreweras complement of a non-crossing permutation in the $(n,m)$-annulus $\pi\in S_{\NC}(n,m)$ is defined as $Kr_{n,m}(\pi):=\pi^{-1}\gamma_{n,m}$, where $\gamma_{n,m}:=(1,\dots,n)(n+1,\dots,n+m)\in S_{n+m}$. For a detailed discussion on the Kreweras complement we refer to Sections \ref{ssec:partitions} and \ref{ssec:partitioned.permutations}. Our next result, provides a second order analogue of Equation \eqref{Equation: free cumulants of product of free variables.intro}, that preserves its nature.

\begin{theorem}\label{Thm: Product of second order free variables}
Let $a,b$ be two second order free random variables, then for any $n,m\geq 1$
\begin{eqnarray}\label{Equation: Second order free cumulants of product of free variables}
\kappa_{n,m}(ab)&=&\sum_{\pi\in S_{\NC}(n,m)}\freec{\pi}{a}\, \freec{Kr_{n,m}(\pi)}{b} \\
&+& \sum_{\pi=\pi_1\times \pi_2}\sum_{\substack{U\in\pi_1 \\ V\in \pi_2}}\kappa_{|U|,|V|}^a\, \freec{\pi\setminus\{U,V\}}{a}\, \freec{Kr(\pi)}{b} \nonumber \\
&+& \sum_{\pi=\pi_1\times \pi_2}\sum_{\substack{U\in Kr_n(\pi_1) \\ V\in 
 Kr_m(\pi_2)}}\kappa_{|U|,|V|}^b\, \freec{\pi}{a}\, \freec{ Kr(\pi) \setminus \{U,V\}}{b} \nonumber
\end{eqnarray}
where the second and third sums are over $\pi=\pi_1\times\pi_2\in \NC(n)\times \NC(m)$ and $Kr(\pi)=Kr_n(\pi_1)\cup Kr_m(\pi_2)$ is the union of the Kreweras complements of $\pi_1$ and $\pi_2$.
\end{theorem}

\begin{remark}\label{Remark: Similar result to cumulants of free variables}
A result in the same direction of Theorem \ref{Thm: Product of second order free variables} was derived in \cite[Theorem 7.3]{arizmendi2023second}, where the authors compute the second order cumulants of the product of second order free random variables. However they assume the extra hypothesis that the second order cumulants of the individual variables are all $0$, which ultimately leads to a cancellation of the second and third sums in the right hand side of \eqref{Equation: Second order free cumulants of product of free variables}. In this sense, our result generalizes \cite[Theorem 7.3]{arizmendi2023second}. A parallel version for higher order cumulants of the product of free variables was derived in \cite[Theorem 7.9]{collins2007second}. They index the sum over partitioned permutations though, while our version is indexed by non-crossing permutations and their Kreweras complement.
\end{remark}

Finally, Theorem \ref{Thm: Product of second order free variables} has two alternative versions where the sums are rather indexed by non-crossing pairings (Lemma \ref{Lemma: Second order cumulants of free product.v1}) and graphs that are either trees or have a single cycle (Lemma \ref{Lemma: Second order cumulants of free product.v2}). Our approach using graphs proves useful to easily compute the second order free cumulants of the product of centered second order variables, see Corollary \ref{Corollary: Product of centered second order free variables}.\\

As an application, we use our results in the case where the variables are second order free semicircular variables, which are characterized by having all cumulants equal to zero, except for $\kappa_2$ and $\kappa_{2,2}$. These variables emerged in the large $N$-limit of $N\times N$ Wigner random matrices, see for instance \cite{male2022joint, mingo2024asymptotic}.

\begin{proposition}\label{Proposition: example of semicircles.intro}
Let $(\AA,\varphi,\varphi^{2})$ be a second order non-commutative probability space and let $a,b\in \AA$ be two second order free semicircular variables such that $\freec{2}{a}=\freec{2,2}{a}=\freec{2}{b}=\freec{2,2}{b}=1$. Then the second order cumulants of their anti-commutator are given by
$$
\freec{n,m}{ab+ba}= \left\{ \begin{array}{lc} 2\frac{(n+m-1)!}{(n-1)!(m-1)!}  & \text{if } n\text{ and }m\text{ are odd, }\\
2\frac{(n+m-1)!}{(n-1)!(m-1)!}+2nm &  \text{if }n\text{ and }m\text{ are even,}\\
0 &  \text{otherwise.} \end{array} \right.
$$
\end{proposition}

The term $\frac{(n+m-1)!}{(n-1)!(m-1)!}$ counts non-crossing pairings with a very specific structure, we leave the precise description of this set to Section \ref{secc:examples}. We can also deal with the general case, where the cumulants $\freec{2}{a},\freec{2,2}{a},\freec{2}{b}$, and $\freec{2,2}{b}$ have arbitrary values, see Remark \ref{rem.anticommutator.semicircular.general}.

The formula for the commutator is almost the same, except for a sign.

\begin{proposition}\label{prop.commutator.semicircular.intro}
Let $(\AA,\varphi,\varphi^{2})$ be a second order non-commutative probability space and let $a,b\in \AA$ be two second order free semicircular variables such that $\freec{2}{a}=\freec{2,2}{a}=\freec{2}{b}=\freec{2,2}{b}=1$. Then the second order cumulants of their commutator are given by
$$
\freec{n,m}{ab-ba}= \left\{ \begin{array}{lc} -2\frac{(n+m-1)!}{(n-1)!(m-1)!}  & \text{if } n\text{ and }m\text{ are odd, }\\
2\frac{(n+m-1)!}{(n-1)!(m-1)!}+2nm(-1)^{\frac{m+n}{2}} &  \text{if }n\text{ and }m\text{ are even,}\\
0 &  \text{otherwise.} \end{array} \right.
$$
\end{proposition}
Same as with the anti-commutator, we can also deal with the general case.  Finally, the formula for the product, is the following. 

\begin{proposition}\label{Proposition: example of semicircles.intro.v2.product}
Let $(\AA,\varphi,\varphi^{2})$ be a second order non-commutative probability space and let $a,b\in \AA$ be two second order free semicircular variables. 
Then the second order cumulants of their product are given by
$$
\freec{n,m}{ab}= \left\{ \begin{array}{lc} n(\freec{2}{a}\freec{2}{b})^n  & \text{if } n=m,\\
0 &  \text{otherwise.} \end{array} \right.
$$
\end{proposition}

Notice that there are no individual cumulants of second order appearing in this formula.\\

\textbf{Organization of the paper.}
Besides this introduction, the paper contains six more parts organized as follows. In Section \ref{sec:prelim} we set the basic notation on partitioned permutations and free cumulants. The proof of our main formula for the anti-commutator, Theorem \ref{Thm.anticommutator.main.2}, is given in Section \ref{Section: proof}. Then in Section \ref{Section: detail study of the sets} we provide three different approaches to better understand the indexing set $\JJ_{2n,2m}$. In Section \ref{sec:commutator} we compute the commutator of two free random variables advertised in Theorem \ref{Thm.commutator.main.2.intro}. The proof of our main formula for the product, Theorem \ref{Thm: Product of second order free variables}, is given in Section \ref{secc:product}. Finally, in Section \ref{secc:examples}, we apply our methods to study semicircular variables, showing Propositions \ref{Proposition: example of semicircles.intro}, \ref{prop.commutator.semicircular.intro}, and \ref{Proposition: example of semicircles.intro.v2.product}.

\section{Preliminaries}
\label{sec:prelim}

In this section we discuss the preliminaries on non-crossing objects, namely, non-crossing partitions, permutations and partitioned permutations that concern our results. We also recall the concepts of freeness and cumulants.

\subsection{Partitions}
\label{ssec:partitions}

\noindent 
We begin with a brief introduction to the definitions and notations used in this paper regarding partitions of a set. For detailed discussion of all of these concepts the standard reference is \cite{NS}.  A \textit{partition $\pi$ of a finite set $S$} is a set of the form $\pi=\{V_1,\dots,V_k\}$ where $V_1,\dots, V_k\subset S$ are pairwise disjoint non-empty subsets of $S$ such that $V_1\cup \dots \cup V_k=S$. The subsets $V_1,\dots, V_k$ are called \textit{blocks} of $\pi$, and we write $\#(\pi)$ for number of blocks (in this case $k$).  We denote by $\PP(S)$ the set of all partitions of $S$, and we further write $\PP(n)$ in the special case where $S=[n]:=\{1,\dots,n\}$. 

Given a partition $\pi\in\PP(n)$, we say that $V\in\pi$ is an \emph{interval block} if it is of the form $V=\{i,i+1,\dots,i+j\}$ for some integers $1\leq i<i+j \leq n$. We further say that $\pi\in\PP(n)$ is an \textit{interval partition} if all the blocks $V\in \pi$ are intervals.

\begin{definition} 
We say that $\pi\in\PP(n)$ is a \textit{non-crossing partition} if for every $1\leq i < j < k < l \leq n$ such that $i, k$ belong to the same block $V$ of $\pi$ and $j,l$ belong to the same block $W$ of $\pi$, then it necessarily follows that all $i,j,k,l$ are in the same block, namely $V=W$. We will denote by $\NN\CC(n)$ the set of all non-crossing partitions of $[n]$.
\end{definition}

We can equip $\PP(n)$ with a lattice structure using the reverse refinement order $\leq$. For $\pi,\sigma\in \PP(n)$, we write ``$\pi \leq \sigma$'' if every block of $\sigma$ is a union of blocks of $\pi$. The maximal element of $\PP(n)$ with this order is $1_n:=\{\{1,\dots,n\}\}$ (the partition of $[n]$ with only one block), and the minimal element is $0_n:=\{\{1\},\{2\},\dots,\{n\}\}$ (the partition of $[n]$ with $n$ blocks). With this order, $(\NN\CC(n),\leq)$ is a sub-poset of $(\PP(n),\leq)$. 

\begin{definition}
\label{defi.joins}
Given two partitions $\pi,\sigma\in \PP(n)$, the \emph{join of $\pi$ and $\sigma$ in the lattice of all partitions}, denoted as $\pi\lor \sigma$, is defined as the smallest partition that is bigger than $\pi$ and $\sigma$ in the inverse refinement order. Given two partitions $\pi,\sigma\in \NN\CC(n)$, the \emph{join of $\pi$ and $\sigma$ in the lattice of non-crossing partitions}, denoted as $\pi\lor_{\NC} \sigma$, is defined as the smallest non-crossing partition that is bigger than $\pi$ and $\sigma$ in the reverse refinement order. 
\end{definition}

Notice that the definition of the join depends on which lattice we are working with. We are only concerned with the lattice $\NN\CC(n)$, but we will work with a special case where both joins coincide, and in this case we can make use of a explicit description of $\pi\lor\sigma$. In this paper, we are particularly interested in partitions $\pi$ such that $\pi\lor I_{2n}=1$ for $I_{2n}:=\{\{1,2\},\{3,4\},\dots,\{2n-1,2n\}\}$. Since $I_{2n}$ is an interval partition, this means that $\pi\lor \sigma=\pi\lor_{\NC} I_{2n}$ (see Exercise 9.43 of \cite{NS}). One can rephrase this relation in terms of the graph $\GG_\pi$ from Definition \ref{defi.Graph.pi}. Actually, the reason to construct the edges of $\GG_\pi$ by joining the blocks containing $2i-1$ and $2i$ is specifically to store information regarding $\pi\lor I_{2n}$.

\begin{lemma}[{\cite[Lemma 3.3]{perales2021anti}}]
\label{Lemma.connected}
Fix a partition $\pi\in \PP(2n)$. Then $\pi\lor I_{2n}=1_{2n}$ if and only if $\GG_\pi$ is connected.
\end{lemma}

\subsection{Permutations}
\label{ssec:permutations}

We now briefly discuss permutations and its relation with partitions. This will motivate the introduction of non-crossing permutations which we explore in detail in Subsection \ref{ssec:partitioned.permutations}. We can always regard a permutation $\pi\in S_n$ as a partition by considering the cycles (of its cycle decomposition) as the blocks of the partition. When necessary for a permutation $\pi \in S_n$ we will denote by $\Part{\pi}$ the corresponding partition. Biane \cite{biane1997some} showed that every permutation $\pi\in S_n$ satisfies the following inequality 
$$\#(\pi)+\#(\pi^{-1}\gamma_n)\leq n+1,$$
where $\gamma_n=(1,\dots,n)\in S_n$ and $\#(\cdot)$ denotes the number of cycles of the permutation in its cycle decomposition. This inequality can also be expressed in terms of the length function $|\pi| =: n-\#(\pi)$ as
\begin{equation}\label{Aux: Triangle inequality}
|\pi|+|\pi^{-1}\gamma_n|\geq |\gamma_n|.
\end{equation}
This means that the length function $|\cdot|$ satisfies the triangle inequality. Moreover, the equality holds only if $\Part{\pi}\in \NC(n)$. The converse is however not that simple. Given a partition, there might be more than one permutation depending on the cyclic order of the blocks. Thankfully, given a non-crossing partition $\pi$, the only permutation that attains equality in \eqref{Aux: Triangle inequality} is the one whose cycles are the blocks of $\pi$ with the elements arranged in increasing order. With this in mind, throughout this paper we will think of non-crossing partition as a non-crossing permutation and vice-versa with the convention that the cycles of the permutation are described by the blocks of the partition with the elements arranged in increased order. This interpretation is very useful in many contexts and provides a natural definition to the notation of annular non-crossing permutations. More explicitly, the inequality \eqref{Aux: Triangle inequality} is actually true for every pair of permutations $\pi,\gamma\in S_n$. Consequently, a natural question is for a fixed $\gamma$ which permutations $\pi$  attain the equality in \eqref{Aux: Triangle inequality}. This question has been extensively explored in the context offree probability. The case where the fixed permutation $\gamma$ has one cycle was studied in \cite{biane1997some}, when $\gamma=(12\cdots n)$ the permutations $\pi$ satisfying equality are referred as disk non-crossing permutation. The case where $\gamma$ with two cycles was studied in \cite{mingo2004annular}, the $\pi$ that attain equality are called annular non-crossing permutations. For the general case where $\gamma$ has several cycles, we refer to \cite[Section 1.8]{mingo2017free}. In this paper we are specially interested in two cycles and we will introduce it in detail in Subsection \ref{ssec:partitioned.permutations}. 

To finish this subsection let us introduce the Kreweras complement of a non-crossing partition. This concept emerges naturally in our formulas for the cumulants of the product of two free variables.

\begin{definition}[Kreweras complement]
\label{defi.Kreweras.1}
Given $\pi\in \NC(n)$, its \textit{Kreweras comeplemt}, $Kr_n(\pi)\in \NC(n)$, is the non-crossing permutation $\pi^{-1}\gamma_n$, where $\gamma_n=(1,\dots,n)\in S_n$. 
\end{definition}

The map $Kr_n:\NC(n)\to\NC(n)$ is actually a lattice anti-isomorphism of $\NC(n)$. Another important property of this map is that $\#(Kr_n(\pi))=n+1-\#(\pi)$ for all $\pi\in\NC(n)$, thus we actually have that $Kr_n(1_n)=0_n$ and $Kr_n(0_n)=1_n$.

In terms of the topology of the non-crossing partitions, the Kreweras complement map has a deeper interpretation, see (\cite[Exercise 18.25]{NS}). Given $\pi\in \NC(n)$, consider additional numbers $\{1',\dots,n'\}$. Then, the \emph{Kreweras complement of  $\pi$}, is the partition $Kr_n(\pi)\in \{1',\dots,n'\}\cong \NC(n)$ that is the largest element among those $\sigma\in\NC(\{1',\dots,n'\})$ which have the property that $\pi\cup \sigma\in \NC(\{1,1',\dots,n,n'\})$. See, for instance, a graphical representation of this in Figure \ref{Fig:Kreweras_first_order}.

\begin{figure}[h!]
    \centering
    \includegraphics[width=0.5\linewidth]{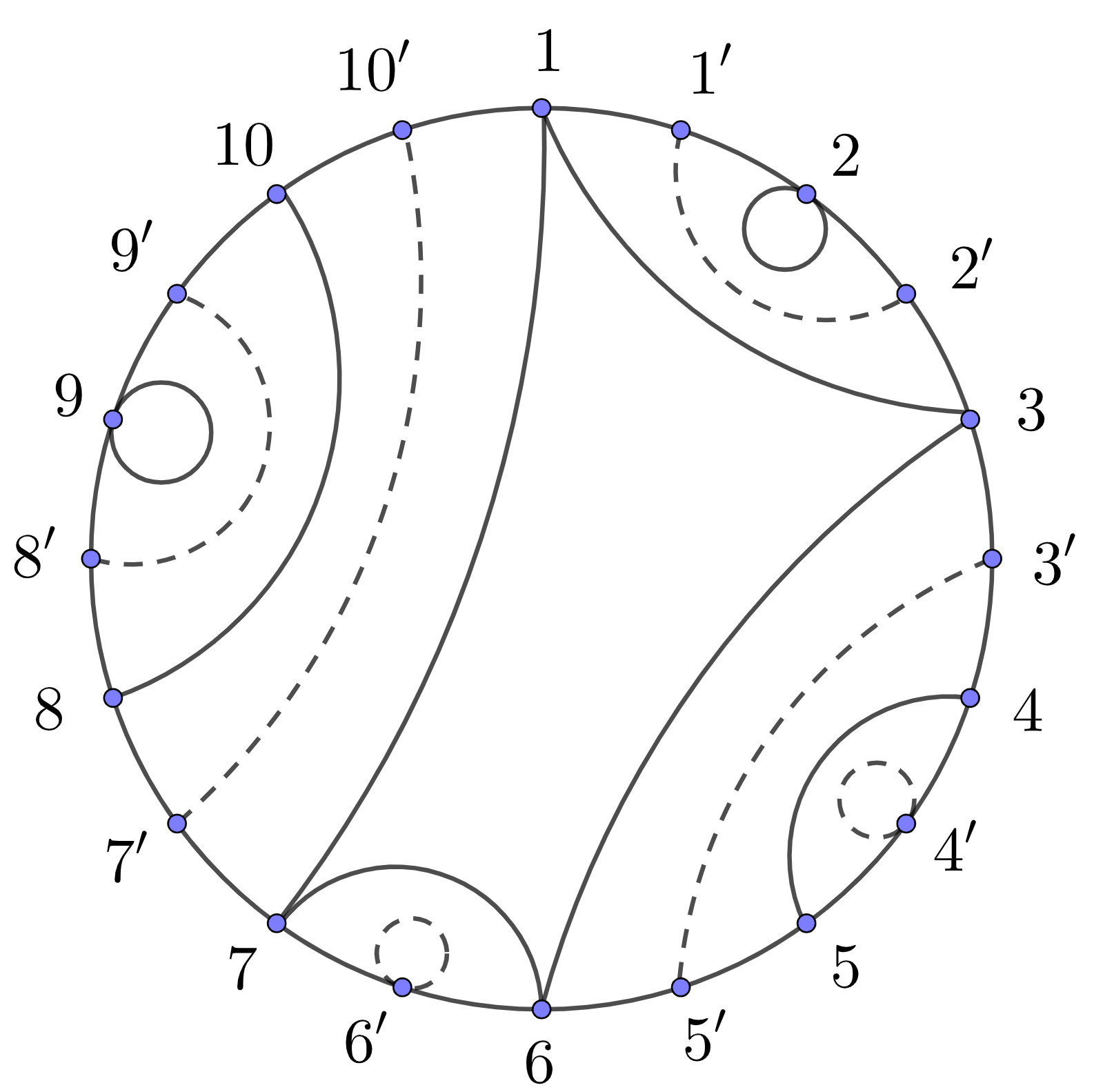}
    \caption{The non-crossing partition $\pi=(1,3,6,7)(4,5)(8,10)(9)$ in solid lines and its Kreweras complement $\pi^{-1}\gamma_{10}=(1,2)(3,5)(4)(6)(7,10)(8,9)$ in dot lines.}
    \label{Fig:Kreweras_first_order}
\end{figure}

This interpretation of the Kreweras complement is the standard definition one finds in the literature concerning non-crossing partitions. Our definition, however, allows us to extend in a more natural way the concept of Kreweras complement of a disk non-crossing partition to the concept of Kreweras complement of an annular non-crossing permutation. We introduce such a concept at the end of Subsection \ref{ssec:partitioned.permutations}.

\subsection{Partitioned permutations}
\label{ssec:partitioned.permutations}

In free probability in order to define the notion of free cumulants and freeness we make use of the set of non-crossing partitions. In second order, to define the notion of second order cumulants and second order freeness we rather look at non-crossing permutations. Let us give a brief introduction to the definitions and notations used in this paper regarding permutations. For a further explanation we refer to \cite{mingo2017free}. Given $n,m\in \mathbb{N}$ we will extensively use the following special permutation
$$\gamma_{n,m}:=(1,\dots, n)(n+1,\dots n+m)\in S_{m+n}.$$
Given $\pi\in S_{m+n}$ we say that $\pi$ is a \textit{non-crossing annular permutation} if 
\begin{enumerate}
    \item $\pi \lor \gamma_{n,m}=1_{m+n}$, and
    \item $\#(\pi)+\#(\pi^{-1}\gamma_{n,m})=m+n$.
\end{enumerate}
Here $\lor$ is defined as in Subsection \ref{ssec:partitions} and the permutation is regarded as a partition by letting the cycles of the permutation to be the blocks of the partition. We denote the set of non-crossing annular permutations as $S_{\NC}(n,m)$. In the spirit of Definition \ref{defi.Kreweras.1} we introduce the Kreweras complement in the annular case.

\begin{definition}
Given $\pi\in S_{\NC}(n,m)$, its \textit{Kreweras complement} is defined as the permutation $Kr_{n,m}(\pi):=\pi^{-1}\gamma_{n,m}$. 
\end{definition}

Notice that $Kr_{n,m}(\pi)$ is again an annular non-crossing permutation. Topologically, it has the same interpretation as in the first order, if we include $\pi$ and $Kr_{n,m}(\pi)$ in the same annulus alternating the numbers (see Figure \ref{Fig:Kreweras_second_order}) then we still get a annular non-crossing permutation. Moreover, given $\pi$, then $Kr_{n,m}(\pi)$ is the largest annular non-crossing permutation with this property. 

\begin{figure}[h!]
    \centering
    \includegraphics[width=0.5\linewidth]{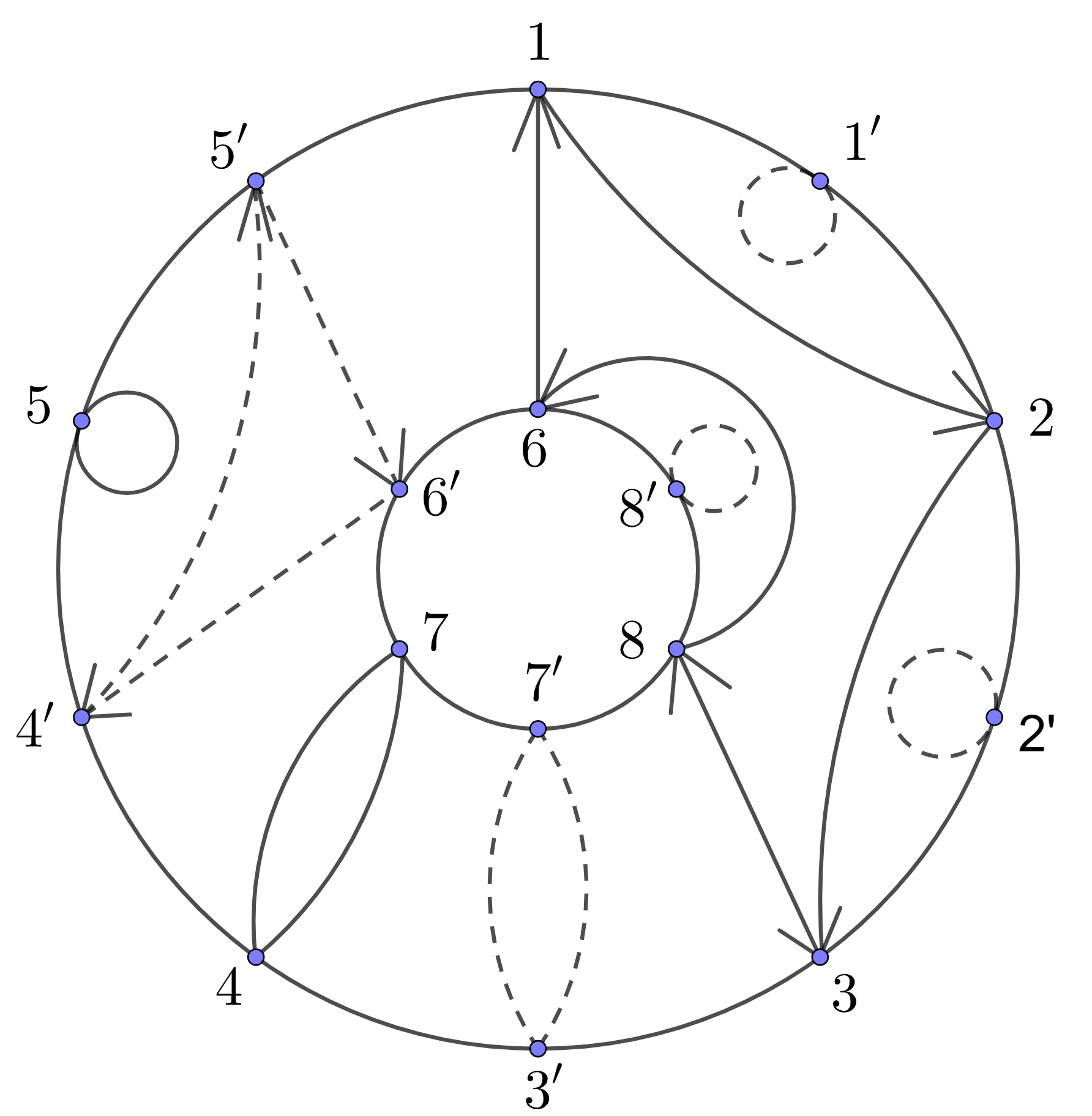}
    \caption{The non-crossing permutation $\pi=(2,3,8,6,1)(4,7)(5)\in S_{\NC}(5,3)$ in solid lines and its Kreweras complement $\pi^{-1}\gamma_{5,3}=(1)(2)(3,7)(4,5,6)(8)$ in dot lines.}
    \label{Fig:Kreweras_second_order}
\end{figure}

A \textit{partitioned permutation} is a pair $(\cU,\pi)\in \PP(n+m)\times S_{n+m}$ where any cycle of $\pi$ is contained in a block of $\cU$. We will focus on the following subset:

\begin{definition}\label{Definition:Non-crossing-pp}
The pair $(\cU,\pi)\in \PS(n+m)$ is a \emph{non-crossing annular partitioned permutation} if it is of one of the two following types.
\begin{enumerate}
    \item \textbf{Type 1:} $\pi\in S_{\NC}(n,m)$ and $\cU=\Part{\pi}.$ We will abuse notation and denote this set simply by $S_{\NC}(n,m)$, associating $(\Part{\pi},\pi)$ with $\pi\in S_{\NC}(n,m)$.
    \item \textbf{Type 2:} $\pi= \pi_1\times \pi_2\in \NC(n)\times \NC(m)$ and every cycle of $\pi$ is a block of $\cU$ except for one block of $\cU$ which is the union of two cycles $A,B\in\pi$ with $A\in\pi_1$ and $B\in\pi_2$. We denote this set by $S_{\NC}^\prime(n,m)$.
\end{enumerate}
We denote by $PS_{\NC}(n,m):=S_{\NC}(n,m)\cup S_{\NC}^\prime(n,m)$ the whole set of non-crossing annular partitioned permutations.
\end{definition}

We should mention that the original definition of non-crossing annular partitioned permutation has a more algebraic motivation, coming from an inequality similar to \eqref{Aux: Triangle inequality}. One first need to define a length function on the set of partitioned permutations that satisfies the triangle inequality. Then a partitioned permutation $(\cU,\pi)$ is non-crossing if we have equality in the triangle inequality and $\cU \lor \Part{\pi^{-1}\gamma_{n,m}}=1_{n,m}$. For more details we refer the reader to \cite[Section 4-5]{collins2007second}. The description of non-crossing annular partitioned permutation that we use here as a Definition \ref{Definition:Non-crossing-pp} was proved in \cite[Proposition 5.11]{collins2007second}.

To finish this subsection let us introduce the following result concerning non-crossing permutations. Given a permutation $\pi\in S_n$ and a set $A\subset [n]$ we let $\pi|_A$ to be the permutation of the elements of $A$ whose cycles are the cycles $C\cap A$ where $C$ is a cycle of $\pi$ and we respect the cyclic order of the elements in $C$. Given two permuations $\pi,\sigma\in S_n$ we say that $\sigma \leq \pi$ if every cycle of $\sigma$ is contained in a cycle of $\pi$ and for each cycle $C$ of $\pi$ the permutation $\sigma|_C$ is non-crossing with respect to $C$, that is
$$\#(\sigma|_C)+\#(\sigma|_C^{-1}C)=|C|+1.$$

\begin{lemma}\label{Lemma: equality if and only if less or equal than} 
Let $\pi,\sigma\in S_{n}$. Then $\pi\leq \sigma$ if and only if $|\pi|+|\pi^{-1}\sigma|=|\sigma|$.
\end{lemma}
\begin{proof}
Suppose $\pi\leq \sigma$, let $C_1\cdots C_w$ be the cycle decomposition of $\sigma$ and let $\pi_i$ be the restriction of $\pi$ to each $C_i$. Then,
$$\#(\pi_i)+\#(\pi_i^{-1}C_i)=|C_i|+1.$$
Summing over $i$ gives,
$$\#(\pi)+\#(\pi^{-1}\sigma)=n+\#(\sigma),$$
hence $|\pi|+|\pi^{-1}\sigma|=|\sigma|$. The converse follows directly from \cite[Lemma 8]{mingo2009second}.
\end{proof}

\subsection{Free cumulants}
\label{ssec:cumulants}

\noindent 
A \textit{non-commutative probability space} is a pair $(\AA,\varphi)$, where $\AA$ is a unital algebra over $\mathbb{C}$ and $\varphi: \AA \to \cc$ is a linear functional, such that $\varphi(1_\AA) =1$. The \textit{$n$-th multivariate moment} is the multilinear functional $\varphi_n: \AA^n \to \mathbb{C}$, such that $\varphi_n(a_1, \ldots, a_n):=\varphi(a_1\cdots a_n) \in \mathbb{C}$, for elements $a_1, \ldots, a_n \in \AA$. In this framework, Voiculescu's original definition of freeness for random variables explains how to compute the mixed moments in terms of the moments of each variable (see \cite{voiculescu1992free}). Since this definition is out of the scope of this paper, we will restrict our presentation to the characterization of freeness using free cumulants.

\begin{notation}\label{notat:multiplicative}
Given a family of multilinear functionals $\{f_m :\AA^{m}\to \mathbb{C}\}_{m\geq1}$ and a partition $\pi\in\PP(n)$, we define $f_\pi :\AA^{n}\to \mathbb{C}$ to be the map
\[
f_\pi(a_1, \ldots, a_n):=\prod_{V \in \pi}f_{|V|}(a_V), \qquad \forall a_1,\dots,a_n\in\AA,
\] 
where for every block $V:=\{{v_1}, \cdots, {v_k}\}$ of $\pi$ (such that ${v_1} < \cdots < {v_k}$ are in natural order) we use the notation $f_{|V|}(a_V):=f_{k}(a_{v_1},\ldots,a_{v_k})$.
\end{notation}

With this notation in hand we can define the free cumulants.

\begin{definition}[Free cumulants]
Let $(\AA,\varphi)$ be a non-commutative probability space. The \textit{free cumulants} are the family of multilinear functionals $\{\kappa_n : \AA^{n}\to \cc\}_{n\geq 1}$ recursively defined by the following formula:
\begin{equation}
\label{FreeMomCum}
	\varphi_n(a_1, \dots, a_n) =\sum_{\pi \in \NN\CC(n)} \kappa_\pi(a_1, \dots, a_n).
\end{equation} 
If we are just working with one variable, and all the arguments are the same $a_1=a_2= \dots= a_n=a$, then we adopt the simpler notation $\freec{n}{a}:=\kappa_n(a, \dots, a)$ and $\freec{\pi}{a}:=\kappa_\pi(a, \dots, a)$
\end{definition}

\begin{remark}
Cumulants are well defined since the right-hand side of the equation contains only one $\kappa_n$ term and the other terms are monomials of cumulants of smaller sizes. Thus we can recursively define $\kappa_n$ in terms of $\varphi_n$ and $\kappa_{n-1},\kappa_{n-2},\dots, \kappa_1$.
\end{remark}

A \textit{second order non-commutative probability space} is a triple $(\AA,\varphi,\varphi^{2})$, where $(\AA, \varphi)$ is a non-commutative probability space and $\varphi^{2}: \AA\times \AA\rightarrow \mathbb{C}$ is a bilinear functional which is tracial in both arguments and which satisfies
$$\varphi^2(a,1)=0=\varphi^2(1,b) \text{ for all } a,b\in \AA.$$
Similarly as in the first order case, the \textit{$n,m$-multivariate moment} is the multilinear functional $\varphi_{n,m}: \AA^n \times \AA^m \rightarrow \mathbb{C}$, such that $\varphi_{n,m}(a_1,\dots,a_{n+m}) := \varphi^2(a_1\cdots a_n ,a_{n+1}\cdots a_{n+m})\in \mathbb{C}$, for elements $a_1,\dots,a_{n+m}\in \AA$. In this setting, one defines second order freeness as a rule to compute the mixed moments in terms of the first and second order moments of the individual variables (see \cite[Definition 2.5]{mingo2007second}). For our work we will rely on the equivalent definition of second order freeness as the vanishing of the mixed cumulants which we later introduce.

\begin{notation}
Given a family of multilinear functionals $\{f_m : \AA^m \to \mathbb{C}\}_{m\geq 1}$ and $\{f_{n,m}: \AA^n \times \AA^m \to \mathbb{C}\}_{n,m\geq 1}$ and a non-crossing annular partitioned permutation $(\cU,\pi)\in \PS(n+m)$, we define $f_{(\cU,\pi)}: \AA^n \times \AA^m \to \mathbb{C}$ to be the map such that for all $a_1\dots,a_{n+m}\in \AA$
\begin{align*}
f_{(\cU,\pi)}(a_1,\dots, a_{n+m}) &:=
 \prod_{\substack{V\in \pi}}f_{|V|}(a_V), & \text{if } (\cU,\pi)\in S_{\NC}(n,m),\text{ or}  \\
f_{(\cU,\pi)}(a_1,\dots, a_{n+m}) &:= f_{|A|,|B|}(a_{A\cup B})\prod_{\substack{V\in \pi\setminus\{A,B\}}}f_{|V|}(a_V),  & \text{if }  (\cU,\pi)\in S_{NC}^\prime(n,m). 
\end{align*}

Here $A=(v_1,\dots,v_{j_1})$ and $B=(v_{{j_1}+1},\dots, v_{j_1+j_2})$ are the two cycles of $\pi$ contained in the same block of $\cU$ as in definition \ref{Definition:Non-crossing-pp} and we use the notation
$$f_{|A|,|B|}(a_{A\cup B})  := f_{j_1, j_2}(a_{v_1},\dots,a_{v_{j_1+j_2}}),$$
\end{notation}

We are now in place to the define the second order free cumulants in a similar way to the free cumulants.

\begin{definition}[Second order free cumulants]
Let $(\AA, \varphi, \varphi^2)$ be a second order non-commutatve probability space. The \textit{second order free cumulants} are the family of multilinear functionals $\{\kappa_{n,m}: \AA^n \times \AA^m \to \mathbb{C}\}_{n,m\geq 1}$ recursively defined by the following formula:
\begin{eqnarray*}
\varphi_{n,m}(a_1,\dots, a_{n+m}) &=& \sum_{(\cU,\pi)\in \PS_{NC}(n,m)}\kappa_{(\cU,\pi)}(a_1,\dots, a_{n+m}) \\
&=& \sum_{(\cU,\pi)\in S_{\NC}(n,m)}\kappa_{(\cU,\pi)}(a_1,\dots, a_{n+m})+\sum_{(\cU,\pi)\in S_{\NC}^\prime(n,m)}\kappa_{(\cU,\pi)}(a_1,\dots, a_{n+m}).
\end{eqnarray*}
\end{definition}

Notice that in the first term of the last line, the sum runs over pairs $(\Part{\pi},\pi)$ with $\pi\in S_{\NC}(n,m)$, thus 
$$\kappa_{(\Part{\pi},\pi)}(a_1,\dots, a_{n+m})= \kappa_{\pi}(a_1,\dots,a_{n+m})$$ recovers only first order free cumulants which we simply call free cumulants. On the other hand, the second term runs over partitions that have only one block of $\cU$ which is a union of two cycles of $\pi$ and therefore $\kappa_{(\cU,\pi)}(a_1,\dots, \ab a_{n+m})$ is the product of a single second order free cumulant and free cumulants. Further, the unique pair $(1_{n+m},\gamma_{n,m})\in S_{\NC}^{\prime}(n,m)$ contributes the second order free cumulant $\kappa_{n,m}(a_1,\dots,a_{n+m})$ while all the rest of pairs contribute cumulants of smaller sizes. Hence, as in the first order case, the second order cumulants are well defined.

\begin{theorem}[Vanishing of mixed cumulants, see {\cite[Lecture 11]{NS}}]
Given a non-commutative probability space $(\AA,\varphi)$ and $a,b\in\AA$ two random variables. Then the following two statements are equivalent: 
\begin{enumerate}
    \item[1.] $a$ and $b$ are free.
    \item[2.] Every mixed cumulant vanishes. Namely, for every $n\geq 2$ and $a_1,\dots,a_n\in\{a,b\}$ which are not all equal, we have that $\kappa_n(a_1,\dots,a_n)=0$.
\end{enumerate}
\end{theorem}

\begin{theorem}[Vanishing of mixed second order cumulants, see {\cite[Section 7.3]{collins2007second}}]
Given a second order non-commutative probability space $(\AA, \varphi, \varphi^2)$ and $a,b\in \AA$ two random variables. Then the following statements are equivalent:
\begin{enumerate}
    \item[1. ] $a$ and $b$ are second order free.
    \item[2. ] Every mixed cumulant vanishes. Namely, for every $n\geq 2$ and $a_1,\dots,a_n\in\{a,b\}$ which are not all equal, we have that $\kappa_n(a_1,\dots,a_n)=0$ and for every $n,m\geq 1$ and $a_1,\dots,a_{n+m}\in\{a,b\}$ which are not all equal, we have that $\kappa_{n,m}(a_1,\dots,a_{n+m})=0$.
\end{enumerate}
\end{theorem}

Finally, when working with cumulants whose entries have products of the underlying algebra $\AA$, there is an efficient formula that allows us to write this cumulant as a sum over cumulants with more entries, where the products are now separated into different entries. The general formula was found in \cite{krawczyk2000combinatorics} and is known as the products as arguments formula. Here we will just use a particular case.

\begin{theorem}[Products as arguments formula]
\label{thm.products.arguments}
Let $(\AA,\varphi)$ be a non-commutative probability space and fix $n\in\nn$. Let $a_1,a_2,\dots ,a_{2n}\in\AA$ be random variables, and consider $I_{2n}:=\{\{1,2\},\{3,4\},\dots,\{2n-1,2n\}\}$ the unique interval pair partition. Then we have that
\begin{equation}
\label{eq.products.arguments}
    \kappa_{n}(a_1a_2,a_3a_4,\dots,a_{2n-1}a_{2n})=\sum_{\substack{\pi\in\NN\CC(2n)\\ \pi \lor I_{2n}=1_{2n} }} \kappa_{\pi}(a_1,a_2,a_3,a_4,\dots,a_{2n-1},a_{2n}).
\end{equation}
\end{theorem}

For the second order free cumulants, a formula for the cumulants of products was also derived in \cite{mingo2009second}. Here we will use a particular case.

\begin{theorem}[Product as argument for second order free cumulants]
Let $(\AA, \varphi, \varphi^2)$ be a second order non-commutative probability space and fix $n,m\in \mathbb{N}$. Let $a_1,\dots,a_{2n+2m}\in \AA$ be random variables and consider the permutation $\gammanm:=(1,\dots, 2n)(2n+1,\dots, 2n+2m)\in S_{2n+2m}$. Then we have that
\begin{equation}
\label{eq.products.arguments.2}
\kappa_{n,m}(a_1a_2,\dots,a_{2n+2m-1}a_{2n+2m}) = \sum_{(\cU,\pi)\in \PS_{NC}(2n,2m)}\kappa_{(\cU,\pi)}(a_1,\dots,a_{2n+2m}),\end{equation}
where the sum is over all pairs $(\cU,\pi)\in \PS_{NC}(2n,2m)$ such that the cycles of $\pi^{-1}\gammanm$ separates even numbers, that is, no cycle of $\pi^{-1}\gammanm$ has two even numbers.
\end{theorem}

One may notice that the conditions in \eqref{eq.products.arguments} and \eqref{eq.products.arguments.2} look different. However, let us recall that any non-crossing partition can be regarded as a permutation. In this setting, it turns out that the condition $\pi \lor I_{2n}=1_{2n}$ is equivalent to $\pi^{-1}\gamma_{2n}$ separates even numbers where $\gamma_{2n}=(1,\dots,2n)\in S_n$, see \cite[Lemma 14]{mingo2009second}. This explains how both conditions are related.

\section{Anti-commutator}
\label{Section: proof}

The goal of this section is to prove Theorem \ref{Thm.anticommutator.main.2}. The approach is to adapt the proof of formula \eqref{formula.anticommutator.main} to the second order setting. Thus, throughout this section it is important to keep in mind that the proof of \eqref{formula.anticommutator.main} required the use of products as arguments formula and the multilinearity of the cumulants, and also relies on two crucial properties

\begin{enumerate}[label=(P\arabic*)]
\item $\pi\lor I_{2n}=1_{2n}$ if and only if $\GG_\pi$ is connected.\label{P1}
\item If there exist $\varepsilon\in \{1,*\}^n$ such that $\{A(\varepsilon),B(\varepsilon)\}\geq \pi$, then $\GG_\pi$ is bipartite. \label{P2}
\end{enumerate}

Through this section we are going to fix two free random variables $a,b$, and natural numbers $n,m\in\nn$. Our ultimate goal is to describe the second order $(n,m)$-cumulant of the anti-commutator, $\freec{n,m}{ab+ba}$, in terms of the cumulants of $a$ and $b$. By multilinearity of the cumulants this amounts to study $(n,m)$-cumulants with entries given by $ab$ or $ba$. To keep track of this kind of expressions we use the notation introduced in \cite[Notation 3.1]{perales2021anti}.
\begin{notation}
\label{nota.epsilons}
Given non-commutative variables $a, b$, we use the notation $(ab)^1:=ab$ and $(ab)^{*}:=ba$. Given a $n$-tuple $\varepsilon \in \{1,*\}^n$, we denote by $(ab)^\varepsilon:=((ab)^{\varepsilon(1)},(ab)^{\varepsilon(2)},\dots,(ab)^{\varepsilon(n)})$ the $n$-tuple with entries $ab$ or $ba$ dictated by the entries of $\varepsilon$.  Furthermore, we will denote by $(a,b)^{\varepsilon}$ the $2n$-tuple obtained from separating the $a$'s from the $b$'s in the $n$-tuple $(ab)^\varepsilon$.  To keep track of the entries in $(a,b)^{\varepsilon}$ that contain an $a$ we use the notation
\[
A(\varepsilon):=\{ 2i-1 |1\leq i\leq n,\ \varepsilon(i)=1\}\cup \{ 2i |1\leq i\leq n,\ \varepsilon(i)=*\}.
\]
Then the entries in $(a,b)^{\varepsilon}$ that contain a $b$ are given by
$B(\varepsilon):=[2n]\backslash A(\varepsilon)$.

For example, if we consider $\varepsilon=(1,*,*,1,*,1) \in \{1,*\}^6$, then $(ab)^\varepsilon=(ab,ab,ba,ab,ba,ab)$ and if we split each entry we get 
\[
(a,b)^{\varepsilon}=(a,b,a,b,b,a,a,b,b,a,a,b).
\]
This means that
\[
A(\varepsilon)=\{1,3,6,7,10,11\},\qquad\text{and}\qquad B(\varepsilon)=\{2,4,5,8,9,12\}.
\]
\end{notation}

Using the products as entries formula together with the vanishing of mixed cumulants, we can easily rewrite the $(n,m)$-cumulant of the anti-commutator.

\begin{proposition}
\label{prop.basic.anticomm.2}
The second order free cumulants of the anti-commutator $ab+ba$ of two second order free random variables $a,b$ satisfy the following formula for all $n,m\in \nn$:
\begin{equation}
\label{eq.rephrase.2}
\freec{n,m}{ab+ba} =\sum_{\substack{(\cU,\pi)\in \PS_{NC}(2n,2m),\\ \pi^{-1}\gammanm \text{ sep. even}}} \sum_{\substack{\varepsilon\in \{1,*\}^{n+m}\\ \{A(\varepsilon),B(\varepsilon)\}\geq \cU}} \kappa_{(\cU,\pi)}((a,b)^{\varepsilon}). 
\end{equation}
Here, we use the notation $\gammanm:=(1,\dots, 2n)(2n+1,\dots, 2n+2m)\in S_{2n+2m}$.
\end{proposition}

\begin{proof}
For a fixed $\varepsilon\in \{1,*\}^{m+n}$, products as arguments formula for second order cumulants \eqref{eq.products.arguments.2}  asserts that
\[
\kappa_n((ab)^{\varepsilon})=\sum_{\substack{(\cU,\pi)\in \PS_{NC}(2n,2m),\\ \pi^{-1}\gammanm \text{ sep. even}}} \kappa_{(\cU,\pi)}((a,b)^{\varepsilon}),
\]
where $\pi^{-1}\gammanm$ sep. even means that there are no two even numbers in a cycle of $\pi^{-1}\gammanm$.

Therefore, if we sum over all possible $\varepsilon\in \{1,*\}^{n}$ we get that 

\begin{align*}
\kappa_{n,m}(ab+ba,\dots,ab+ba) &=\sum_{\varepsilon\in \{1,*\}^{n+m}} \sum_{\substack{(\cU,\pi)\in \PS_{NC}(2n,2m),\\ \pi^{-1}\gammanm \text{ sep. even}}} \kappa_{(\cU,\pi)}((a,b)^{\varepsilon}) \nonumber \\
&=\sum_{\substack{(\cU,\pi)\in \PS_{NC}(2n,2m),\\ \pi^{-1}\gammanm \text{ sep. even}}} \sum_{\varepsilon\in \{1,*\}^{n+m}}  \kappa_{(\cU,\pi)}((a,b)^{\varepsilon}),
\end{align*}

where in the second equality we just changed the order of the sums. Finally, since $a$ and $b$ are second order free, every mixed free cumulant  and second order free cumulant will vanish. Thus we require that $\{A(\varepsilon),B(\varepsilon)\}\geq \cU$ in order to get $\kappa_{(\cU,\pi)}((a,b)^{\varepsilon})\neq 0$. 
\end{proof}

A key idea is that the right-hand side of \eqref{eq.rephrase.2} can be greatly simplified by observing that there are very few $\varepsilon$ satisfying $\{A(\varepsilon),B(\varepsilon)\}\geq\cU$ for a fixed partitioned permutation $(\cU,\pi)$. 

\begin{remark}
An analogous version to Proposition \ref{prop.basic.anticomm.2} was presented in \cite[Proposition 3.2]{perales2021anti}. In the first order case, the condition that $\pi^{-1}\gamma_{2n}$ separates even is equivalent to $\pi \lor I_{2n}=1_{2n}$ and thus equivalent to the graph $\GG_\pi$ being connected, this is property \ref{P1} mentioned at the beginning of the section. Naturally, one would expect to have the same property in second order case. However, the set of permutations $\pi\in S_{\NC}(2n,2m)$ satisfying that $\pi^{-1}\gammanm$ separates even is a subset of the permutations $\pi$ such that the graph $\GG_\pi$ is connected. This is, not every permutation $\pi$ with connected graph $\GG_\pi$ satisfies that $\pi^{-1}\gammanm$ separates even. On the other hand, the case when $\pi\in \NC(2n)\times \NC(2m)$ will be easily checked with the help of the results derived for first order case in \cite{perales2021anti}. In this sense, the property \ref{P1} is more involved in second order case. In Section \ref{Section: detail study of the sets}  we explore in more detail the subset of connected graphs $\GG_\pi$ that satisfy the condition $\pi^{-1}\gammanm$ separates even for $\pi\in S_{\NC}(2n,2m)$.
\end{remark}

Unless otherwise stated, in the remaining of the paper we will use the notation
\begin{align}
 I&:=\{\{1,2\},\dots, \{2n+2m-1,2n+2m\}\}\in\cP(2n+2m) \label{eq:not.I} \\ 
 \boldsymbol{1}&:=\{\{1,2,\dots,2n+2m-1,2n+2m\}\}\in\cP(2n+2m)  \label{eq:not.1} \\
\gamma &:= (1,\dots,2n)(2n+1,\dots,2n+2m)\in S_{2n+2m}. \label{eq:not.gamma}
\end{align}
 Notice that $\boldsymbol{1}$ is the maximum partition in $\cP(2n+2m)$. We also let 
 $$[2n]=: \{1,\dots,2n\} \qquad  \text{and} \qquad [2m]=: \{2n+1,\dots,2n+2m\}.$$

\begin{proposition}\label{Proposition: separates implies connected}
If $\pi\in S_{\NC}(2n,2m)$ is such that $\pi^{-1}\gammanm$ separates even elements, then $\Part{\pi}\vee I=\boldsymbol{1}$. Moreover, Lemma \ref{Lemma.connected} implies that $\GG_\pi$ is connected.
\end{proposition}

\begin{proof}
Suppose $\Part{\pi}\vee I<\boldsymbol{1}$, then there is a partition $\cV$ with two blocks $A$ and $B$ such that $\Part{\pi}\vee I\leq \cV$. Observe that there must be one block $A$ or $B$ that intersect both sets $[2n]$ and $[2m]$, otherwise $A=[2n]$ and $B=[2m]$ or the other way around. Since $\Part{\pi} \leq \Part{\pi}\vee I\leq \cV$ in either we would have each cycle of $\pi$ being contained in either $[2n]$ or $[2m]$ which is a contradiction as $\pi \vee \gammanm=\boldsymbol{1}$. So let us assume $B\cap [2n]\neq \emptyset$ and $B\cap [2m]\neq \emptyset$.

Let $\gamma_A:=\gammanm|_A$ and $\gamma_B:=\gammanm|_B$. Note that $\gamma_A$ must have either $1$ or $2$ cycles, if it has $1$ cycle then by \cite[Lemma 5]{mingo2009second} it follows $\pi|_A \in \NC(\gamma_A)$. If it has two cycles, $C_1,C_2$, by \cite[Lemma 5]{mingo2009second} it follows $\pi|_A \in \left(\NC(C_1)\times \NC(C_2)\right) \sqcup S_{\NC}(C_1,C_2)$. In the first case it implies
$$\#(\pi|_A)+\#(\pi|_A^{-1}\gamma_A) = |A|+1,$$
in the second case it implies
$$\#(\pi|_A)+\#(\pi|_A^{-1}\gamma_A) = |A|+2,$$
provided $\pi|_A \in \NC(C_1)\times \NC(C_2)$ while
$$\#(\pi|_A)+\#(\pi|_A^{-1}\gamma_A) = |A|,$$
provided $\pi|_A \in S_{\NC}(C_1,C_2)$. On the other hand since 
$B\cap [2n]\neq \emptyset$ and $B\cap [2m]\neq \emptyset$ it follows by the same argument that
$$\#(\pi|_B)+\#(\pi|_B^{-1}\gamma_B) = |B|.$$
So we conclude that 
$$\#(\pi)+\#(\pi^{-1}\gamma_A\gamma_B)\in \{2n+2m+1,2n+2m+2,2n+2m\}.$$

On the other hand, since each block $A$ and $B$ is the union of blocks of the form $\{2j-1,2j\}$ for $j=1,2,\dots,n+m$ then it follows that $$\gamma_A\gamma_B=\gammanm(u_1,u_2)(u_2,u_3)\cdots (u_{r}u_{r+1})(v_1,v_2)(v_2,v_3)\cdots (v_{s},v_{s+1}),$$
where $u_1 < u_2 <\cdots < u_{r+1}$ are all even numbers contained in $[2n]$ and $v_1 < v_2 <\cdots < v_{s+1}$ are all even numbers contained in $[2m]$. This is because when we multiply by the transpositions $(u_i,u_{i+1})$ we split the cycle $(1,\dots,2n)$ into two cycles, one that contains the elements of $A$ and one that contains the elements of $B$ and both respect the order of the cycle $(1,\dots,2n)$. The cycle that contains the elements of $A$ is a cycle of $\gamma_A$ while the cycle that contains the elements of $B$ is a cycle of $\gamma_B$ (which might be empty in which case we take no transpositions). Similarly when multiplying by the transpositions $(v_i,v_{i+1})$ we obtain the other cycles of $\gamma_A$ and $\gamma_B$. Therefore $$\pi^{-1}\gamma_A\gamma_B=\pi^{-1}\gammanm(u_1,u_2)(u_2,u_3)\cdots (u_{r}u_{r+1})(v_1,v_2)(v_2,v_3)\cdots (v_{s},v_{s+1}),$$
but $\pi^{-1}\gammanm$ separates even numbers and then
\begin{align*} 
\#(\pi^{-1}\gamma_A\gamma_B) &=
\#(\pi^{-1}\gammanm(u_1,u_2)(u_2,u_3)\cdots (u_{r},u_{r+1})(v_1,v_2)(v_2,v_3)\cdots (v_{s},v_{s+1})) \\
&=\#(\pi^{-1}\gammanm)-r-s.
\end{align*}
Using that $\pi\in S_{\NC}(2n,2m)$ we get,
$$\#(\pi)+\#(\pi^{-1}\gamma_A\gamma_B)=\#(\pi)+\#(\pi^{-1}\gammanm)-r-s=2n+2m-r-s. $$
By the observed before it means the unique possible solution is $r=s=0$ which means $A=[2n]$ and $B=[2m]$, or the other way around. But this impossible as mentioned at the beginning of the proof.
\end{proof}

Now we can use $\GG_\pi$ to test when $\{A(\varepsilon),B(\varepsilon)\}\geq \pi$ for a fixed $\pi\in S_{\NC}(2n,2m)$ and a fixed $\varepsilon\in \{1,*\}^{m+n}$. Turns out that $\GG_\pi$ must be bipartite.

\begin{remark}[Bipartite graphs]
\label{rem.bipartite.graph}
Let $\GG=(V_\GG,E_\GG)$ be an undirected graph (where loops and multiedges are allowed), we say that $\GG$ is \emph{bipartite} if there exist a partition $\{V'_\GG,V''_\GG\}$ of the set of vertices $V_\GG$ such that there are no edges $E_\GG$ connecting two vertices in $V'_\GG$ or connecting two vertices in $V''_\GG$. In other words, all vertices connect a vertex in $V'_\GG$ with a vertex in $V''_\GG$. 

Moreover, if $\GG$ is connected and bipartite, then the bipartition $\{V'_\GG,V''_\GG\}$ is unique (up to the permutation $\{V''_\GG,V'_\GG\}$). This is because once we identify a vertex $v\in V'_\GG$, then the set of any other vertex $u\in  V_\GG$ is determined by its distance to $v$. If the distance is even, then $u\in V'_\GG$, and if the distance is odd, then $u\in V''_\GG$.
\end{remark}

\begin{proposition}
\label{prop.few.nonvanishing.epsilons.2}
Let $\pi\in S_{\NC}(2n,2m)$ be such that $ \pi^{-1}\gammanm$ separates even elements. Assume that there exists an $\varepsilon\in\{1,*\}^{n+m}$ such that $\{A(\varepsilon),B(\varepsilon)\}\geq \pi$, then $\pi\in \JJ_{2n,2m}$. Moreover if $\pi\in \JJ_{2n,2m}$ there are exactly two tuples $\varepsilon\in\{1,*\}^{n+m}$ satisfying $\{A(\varepsilon),B(\varepsilon)\}\geq \pi$. Furthermore, these tuples are completely opposite, that is, they do not coincide in any entry.
\end{proposition}

\begin{proof}
The existence of an $\varepsilon\in\{1,*\}^n$ such that $\{A(\varepsilon),B(\varepsilon)\}\geq \pi$ implies that we can write $\pi=\{A_1,\dots,A_r,B_1,\dots, B_s\}$ such that $A_i\subset A(\varepsilon)$ for $i=1,\dots, r$ and $B_j\subset B(\varepsilon)$ for $j=1,\dots,s$. If we consider $V'=\{A_1,\dots,A_r\}$ and $V''=\{B_1,\dots, B_s\}$, then $\{V',V''\}$ is a bipartition of $\GG_\pi$. Indeed, by construction of $\GG_\pi$, if $e\in E_{\GG_\pi}$, then $e$ connects the blocks containing elements $2k-1$ and $2k$ for some $k=1,\dots,n$. However, by definition of $A(\varepsilon)$ and $B(\varepsilon)$, one must contain $2k-1$ while the other contains $2k$. Thus $e$ connects vertices from different sets and $\{V',V''\}$ is actually a bipartition of $\GG_\pi$. Since $ \pi^{-1}\gammanm$ separates even by assumption, we conclude that $\pi\in \JJ_{2n,2m}$.

Given $\pi \in \JJ_{2n,2m}$, from Proposition \ref{Proposition: separates implies connected} we obtain that $\GG_\pi$ is connected. By definition of $\JJ_{2n,2m}$ we also know that $\GG_\pi$ is bipartite. Thus Remark \ref{rem.bipartite.graph} asserts that there exist a unique bipartition $\{\pi',\pi''\}$ of the vertices of $\GG_\pi$ (blocks of $\pi$), say $\pi'=\{V_1,\dots,V_r\}$ and $\pi''=\{W_1,\dots, W_s\}$, where $V_1$ contains the element 1. Then there are only two options, $A(\varepsilon)=V_1\cup \dots \cup V_r$ or $A(\varepsilon)=W_1\cup \dots \cup W_s$. The set $A(\varepsilon)$ clearly determines $\varepsilon$, since $\varepsilon(i)=1$ if and only if $2i-1\in A(\varepsilon)$. So there are two possible $\varepsilon$. Furthermore, since $(V_1\cup \dots \cup V_r)\cap (W_1\cup \dots \cup W_s)=\emptyset$, then the two possible $\varepsilon$ do not coincide in any entry. 
\end{proof}

\begin{notation}\label{Notation: Deiniftion of two epsilons}
In light of the previous result, given a $\pi\in S_{\NC}(2n,2m)$ such that $\GG_\pi$ is connected and bipartite, we will denote by $\varepsilon_\pi$ the (unique) tuple such that $\{A(\varepsilon_\pi),B(\varepsilon_\pi)\}\geq \pi$ and $1\in A(\varepsilon_\pi)$. And we denote by $\varepsilon'_\pi$ the other possible tuple, which actually satisfies that $A(\varepsilon'_\pi)=B(\varepsilon_\pi)$ and $B(\varepsilon'_\pi)=A(\varepsilon_\pi)$.
\end{notation}

\begin{remark}
\label{rem.pi.prime.2}
Recall from Definition \ref{main.defi.2} that every $\pi\in \JJ_{2n,2m}$ is naturally decomposed as $\pi:=\pi'\sqcup \pi''$ where $(\pi',\pi'')$ is the bipartition of $\GG_\pi$. From the previous proof we observe that $\pi'=\pi|A(\varepsilon_\pi)$ and $\pi''=\pi|B(\varepsilon_\pi)$.  
\end{remark}

We are ready to prove our main result, Theorem \ref{Thm.anticommutator.main.2}. We state it in a slightly different form, where we expand the sums further. The equivalence of both results is further discussed below in Remark \ref{Remark: Equivalence of main theorems}.

\begin{theorem}
\label{Thm.anticommutator.main.2.2}
Consider two second order free random variables $a$ and $b$, and let $(\freec{n}{a})_{n\geq 1},(\freec{n,m}{a})_{n,m\geq 1}$, $(\freec{n}{b})_{n\geq 1},(\freec{n,m}{b})_{n,m\geq 1}$ and $(\freec{n,m}{ab+ba})_{n,m\geq 1}$ be the sequence of first and second order free cumulants of $a$, $b$ and $ab+ba$, respectively. Then, for every $n,m\geq 1$ one has
\begin{equation}
\label{formula.anticommutator.main.2}
\freec{n,m}{ab+ba} = \sum_{\substack{\pi \in \JJ_{2n,2m}\\ \pi=\pi' \sqcup \pi''}} \,\left( \freec{\pi'}{a}  \, \freec{\pi''}{b} + \freec{\pi'}{b} \freec{ \pi''}{a} \right)+ \sum_{\substack{\pi_1\times \pi_2\in \NCac_{2n}\times \NCac_{2m} \\ \pi_1 =\pi_1'\sqcup \pi_1'' \\ \pi_2 =\pi_2'\sqcup \pi_2''}}\sum_{(C_1,C_2) \in \pi_1 \times \pi_2} f(C_1,C_2)
\end{equation}

where
$$f(C_1,C_2)= \left\{ \begin{array}{lcc} \mathcal{A} & \text{if} & (C_1,C_2)\in \pi_1'\times \pi_2' \\ \\ \mathcal{B} & \text{if} & (C_1,C_2)\in \pi_1'\times \pi_2'' \\ \\ \mathcal{C} & \text{if} & (C_1,C_2)\in \pi_1''\times \pi_2' \\ \\ \mathcal{D} & \text{if} & (C_1,C_2)\in \pi_1''\times \pi_2'' \end{array} \right.$$
and
$$
\mathcal{A} =: \freec{|C_1|,|C_2|}{a}\, \freec{\pi_1'\setminus C_1}{a}\,  \freec{\pi_1''}{b}\, \freec{\pi_2'\setminus C_2}{a}\, 
  \freec{\pi_2''}{b}  + \freec{|C_1|,|C_2|} {b}
\, \freec{\pi_1'\setminus C_1}{b}\,  \freec{\pi_1''}{a}\, \freec{\pi_2'\setminus C_2}{b}\, 
  \freec{\pi_2''}{a},
$$
$$
\mathcal{B} =: \freec{|C_1|,|C_2|}{a} \, 
\freec{\pi_1'\setminus C_1}{a} \,  
\freec{\pi_1''}{b} \, 
\freec{\pi_2'}{a} \,
\freec{\pi_2''\setminus C_2}{b} 
+ \freec{|C_1|,|C_2|} {b} \, 
\freec{\pi_1'\setminus C_1}{b} \,  
\freec{\pi_1''}{a} \, 
\freec{\pi_2'}{b}\, 
\freec{\pi_2''\setminus C_2}{a},
$$
$$
\mathcal{C} =: \freec{|C_1|,|C_2|}{a}\, \freec{\pi_1'}{a}\,  \freec{\pi_1''\setminus C_1}{b}\, \freec{\pi_2'\setminus C_2}{a} \, 
 \freec{\pi_2''}{b}  + \freec{|C_1|,|C_2|} {b}
\, \freec{\pi_1'}{b}\,  \freec{\pi_1''\setminus C_1}{a}\, \freec{\pi_2'\setminus C_2}{b}\, 
  \freec{\pi_2''}{a},
$$
$$
\mathcal{D} =: \freec{|C_1|,|C_2|}{a}\, \freec{\pi_1'}{a}\,  \freec{\pi_1''\setminus C_1}{b}\, \freec{\pi_2'}{a}\, 
  \freec{\pi_2''\setminus C_2}{b}  + \freec{|C_1|,|C_2|} {b}
\, \freec{\pi_1'}{b}\,  \freec{\pi_1''\setminus C_1}{a}\, \freec{\pi_2'}{b} \, 
 \freec{\pi_2''\setminus C_2}{a}.
$$

\end{theorem}

\begin{remark}\label{Remark: Equivalence of main theorems}
One may regard a partitioned permutation $(\cU,\pi)\in S_{\NC}^\prime(2n,2m)$ as a pair $\pi=\pi_1\times \pi_2$ and $(C_1,C_2)\in\pi_1\times \pi_2$. In this case $(C_1,C_2)$ correspond to the unique cycles of $\pi$ that are merged into a single block of $\cU$. In this setting, it is easy to note that $\GG_\cU$ is the resulting graph after merging the vertices $C_1$ and $C_2$ of the graphs $\GG_{\pi_1}$ and $\GG_{\pi_2}$ respectively. This vertex corresponds to the unique block $U\in \cU$ which is union of two cycles of $\pi$. Further, it is easy to observe that $\GG_\cU$ is connected and bipartite if and only if both $\GG_{\pi_1}$ and $\GG_{\pi_2}$ are connected and bipartite. This allows us to rewrite the second sum of our main result, Equation \eqref{formula.anticommutator.main.2.intro}, as a double sum over $\pi=\pi_1\times\pi_2$ and $(C_1,C_2)\in \pi_1\times \pi_2$ where both $\GG_{\pi_1}$ and $\GG_{\pi_2}$ are connected and bipartite. Therefore, we may restate our main theorem in terms of the sets $\NCac_{2n}$ and $\NCac_{2m}$ introduced in \cite{perales2021anti} consisting of the set of connected and bipartite graphs $\GG_\pi$.
\end{remark}

\begin{proof}
Our starting point is Equation \eqref{eq.rephrase.2} from Proposition \ref{prop.basic.anticomm.2}. We work with the cases $(\cU,\pi)\in S_{\NC}(2n,2m)$ and $(\cU,\pi)\in S_{\NC}^\prime(2n,2m)$ separately, so that $\freec{n,m}{ab+ba}$ is the sum of
\begin{equation*}
\sum_{\substack{\pi\in S_{\NC}(2n,2m) \\ \pi^{-1}\gammanm \text{ sep. even}}} \sum_{\substack{\varepsilon\in \{1,*\}^{n+m}\\ \{A(\varepsilon),B(\varepsilon)\}\geq \pi}} \kappa_{\pi}((a,b)^{\varepsilon})
\end{equation*}
and
\begin{equation*}
\sum_{\substack{(\cU,\pi)\in S_{\NC}^\prime(2n,2m)\\ \pi^{-1}\gammanm \text{ sep. even}}} \sum_{\substack{\varepsilon\in \{1,*\}^{n+m}\\ \{A(\varepsilon),B(\varepsilon)\}\geq \cU}} \kappa_{(\cU,\pi)}((a,b)^{\varepsilon})
\end{equation*}
If $\pi\in S_{\NC}(2n,2m)$ then by Proposition \ref{prop.few.nonvanishing.epsilons.2} the condition $\{A(\varepsilon),B(\varepsilon)\}\geq \pi$ is only true if $\GG_\pi$ is bipartite. Furthermore, in this case $\varepsilon$ can only be one of $\varepsilon_\pi$ or $\varepsilon'_\pi$ as in Notation \ref{Notation: Deiniftion of two epsilons}. Therefore, the first summand can be simplified to  
\[
\sum_{\substack{\pi \in \JJ_{2n,2m}}} 
\Bigg( \prod_{\substack{V\in \pi, \\ V\subset A(\varepsilon_\pi)}} \freec{|V|}{a} \prod_{\substack{W\in \pi, \\ W\subset B(\varepsilon_\pi)}} \freec{|W|}{b} +
 \prod_{\substack{V\in \pi, \\ V\subset A(\varepsilon'_\pi)}} \freec{|V|}{a}  \prod_{\substack{W\in \pi, \\ W\subset B(\varepsilon'_\pi)}} \freec{|W|}{b} \Bigg).
\]
From Remark \ref{rem.pi.prime.2} we have that $\pi'=\{V\in\pi:V\subset A(\varepsilon_\pi)\}=\{W\in \pi: W\subset B(\varepsilon'_\pi)\}$ and similarly $\pi''=\{W\in\pi:W\subset B(\varepsilon_\pi)\}=\{V\in \pi: V\subset A(\varepsilon'_\pi)\}$, so we obtain the desired first summand of the formula \eqref{formula.anticommutator.main.2}. For the second sum we have $\pi=\pi_1\times \pi_2\in \NC(2n)\times \NC(2m)$. The condition $\pi^{-1}\gammanm$ separates even becomes $\pi_1^{-1}\gamma_{2n}$ and $\pi_2^{-1}\gamma_{2m}$ separates even, where $\gamma_{2n}=(1,\dots,2n)$ and $\gamma_{2m}=(2n+1,\dots,2n+2m)$. Further $\{A(\varepsilon),B(\varepsilon)\}\geq \cU\geq \pi$. From this point we can proceed as in the proof of \cite[Theorem 1.4]{perales2021anti} for each of the partitions $\pi_1$ and $\pi_2$ so that $\GG_{\pi_1}$ and $\GG_{\pi_2}$ must be connected and bipartite. Moreover, in this case, for each $\pi_1$, $\varepsilon$ can only be one of $\varepsilon_{\pi_1}$ or $\varepsilon_{\pi_1}^\prime$ and similarly for $\pi_2$. 
Now we are allowed to choose one cycle from each $\pi_1$ and $\pi_2$, say $C_1$ and $C_2$ respectively, so that every cycle of $\pi$ is a block of $\cU$ except $C_1$ and $C_2$ for which $C_1\cup C_2$ is a block of $\cU$. Since we require $\{A(\varepsilon),B(\varepsilon)\}\geq \cU$ then if $C_1 \subset A(\varepsilon)$, it must be $C_2 \subset A(\varepsilon)$. This means that given a choice $\varepsilon_{\pi_1}$ or $\varepsilon_{\pi_1}^\prime$ the choice of $\varepsilon_{\pi_2}$ or $\varepsilon_{\pi_2}^\prime$ is determined. That is, there are only two possible choices for $\varepsilon$, say $\varepsilon_\pi$ and $\varepsilon_\pi^\prime$. To conclude, it is enough to observe that for a given $(\cU,\pi)$ the corresponding contribution $\kappa_{(\cU,\pi)}((a,b)^{\varepsilon_\pi})+\kappa_{(\cU,\pi)}((a,b)^{\varepsilon_\pi^\prime})$ is one of the terms $\mathcal{A},\mathcal{B},\mathcal{C}$ or $\mathcal{D}$ depending on either $(C_1,C_2)\in \pi_1^\prime\times\pi_2^\prime,(C_1,C_2)\in \pi_1^\prime\times\pi_2^{\prime\prime},(C_1,C_2)\in \pi_1^{\prime\prime}\times\pi_2^\prime$ or $(C_1,C_2)\in \pi_1^{\prime\prime}\times\pi_2^{\prime\prime}$ respectively. This gives the desired second summand of the formula \eqref{formula.anticommutator.main.2}.

\end{proof}

\begin{remark}
Notice that in Theorem \ref{Thm.anticommutator.main.2.2}, the first sum preserves exactly the same conditions as the sum that appears in its analogous first order case \eqref{formula.anticommutator.main}. The only two small differences are that we sum over non-crossing annular permutations instead of non-crossing partitions and we sum over permutations $\pi$ such that $\pi^{-1}\gammanm$ separates even, instead of the graph being connected. We will discuss this requirement in Section \ref{Section: detail study of the sets}. On the other hand, the contribution of the second sum is the product of exactly one second order free cumulant, and some first order free cumulants.
\end{remark}


\subsection{A concrete formula for small \texorpdfstring{$n$}{n} and \texorpdfstring{$m$}{m}.}
\label{ssec:concrete}

In order to exemplify how our formula works, we provide a list of the partitions in $\JJ_{2,2}$, $\NCac_{2,2}$, $\JJ_{4,2}$, and $\NCac_{4,2}$. Then we use this to compute the formulas for $\freec{1,1}{ab+ba}$ and $\freec{2,1}{ab+ba}=\freec{1,2}{ab+ba}$.

\subsubsection{Case \texorpdfstring{$n=m=1$}{n=m=1}}

Notice that $\JJ_{2,2}$ contains only one permutation $\pi=(1,3)(2,4)$. Its Kreweras complement is $\sigma=\pi^{-1}\gamma_{2,2}=(1,4)(2,3)$. To put this into perspective, the set $S_{\NC}(2,2)$ that in principle (before using our theorem) indexes the terms of the first sum in the formula for the anti-commutator, has 18 permutations. But only 1 in is actually needed to compute the anti-commutator.

To understand the terms on the second sum, we notice that $(\cU,\pi) \in \NCac_{2,2}$ if $\pi=\pi_1\times \pi_2$ with $\GG_{\pi_1}$ and $\GG_{\pi_2}$ connected and bipartite, thus, $\pi_1=\pi_2=(1)(2)$. Meaning that 
$\pi=(1)(2)(3)(4)$. To form $\cU$ we need to merge a block from $\{(1),(2)\}$ with a block from $\{(3),(4)\}$. Thus there 4 possible partitions $\cU$:
$$\{\{1,3\},\{2\},\{4\}\}, \quad \{\{1,4\},\{2\},\{3\}\}, \quad \{\{1\},\{2,3\},\{4\}\}, \quad \{\{1\},\{2,4\},\{3\}\}.$$

Once we understand $\JJ_{2,2}$ and $\NCac_{2,2}$, our main formula \eqref{formula.anticommutator.main.2} yields that given second order free random variables $a$ and $b$, the $(1,1)$-cumulant of their anti-commutator is
$$\freec{1,1}{ab+ba} = 2\freec{2}{a}\freec{2}{b} +4\freec{1,1}{a}\freec{1}{b}\freec{1}{b}+4\freec{1,1}{b}\freec{1}{a}\freec{1}{a}.$$

\subsubsection{Case \texorpdfstring{$n=2$, $m=1$}{n=2, m=1}}

With some effort, one can also find the partitions $\pi$ in $\JJ_{4,2}$ and their Kreweras complements $\sigma=\pi^{-1}\gamma_{4,2}$. There are 14:
\begin{center}
\begin{tabular}{c|c}
$\pi$ & $\sigma=\pi^{-1}\gamma_{4,2}$ \\ \hline
(1,6,4)(2,3,5)  & (1,5)(2)(4)(3,6) \\
(1,5,4)(2,3,6) & (1,6)(2)(3,5)(4) \\ 
(1,5,4)(2,3)(6) & (1,3,5,6)(2)(4) \\
(1,6,4)(2,3)(5)  & (1,3,6,5)(2)(4) \\
(1,4)(2,3,6)(5) & (1,6,5,3)(2)(4) \\ 
(1,4)(2,3,5)(6)  & (1,5,6,3)(2)(4) \\
(1,6,4)(2)(3,5)  & (1,2,5)(4)(3,6) \\
\end{tabular}
\hspace{.5cm}
\begin{tabular}{c|c}
$\pi$ & $\sigma=\pi^{-1}\gamma_{4,2}$ \\ \hline
(1,5,4)(2,6)(3) & (1,6)(2,3,5)(4) \\
(1)(2,3,5)(4,6) & (1,5,4)(2)(3,6) \\
(1,5)(2,3,6)(4) & (1,6)(2)(3,4,5) \\
(1,3,5)(2)(4,6) & (1,2)(4,5)(3,6) \\ 
(1)(2,6,4)(3,5) & (1,4)(2,5)(3,6) \\ 
(1,5)(2,4,6)(3) & (1,6)(2,3)(4,5) \\ 
(1,5,3)(4)(2,6) & (1,6)(2,5)(3,4) \\
\end{tabular}
\end{center}

The terms on the second sum are indexed by partitions $(\cU,\pi) \in \NCac_{4,2}$. Thus $\pi=\pi_1\times \pi_2$ with 
$\pi_1\in \NCac_4$ and $\pi_2\in \NCac_2$. As mentioned previously, $\NCac_2$ only contains the partition $(1)(2)$. On the other hand $\NCac_4$ has 5 partitions. Thus we get that there are 5 possible $\pi$: 
$$(14)(23)(5)(6), \quad (13)(2)(4)(5)(6), \quad (14)(2)(3)(5)(6),$$
$$\quad (1)(23)(4)(5)(6), \quad (1)(24)(3)(5)(6).$$
To form $\cU$ we need to merge a block from $\pi_1$ with a block from $\pi_2$. So there are 28 possible partitions $\cU$, listed below:

\begin{center}
\begin{tabular}{c}
$\pi=(14)(23)(5)(6)$ \\ \hline
$\{\{1,4,5\},\{2,3\},\{6\}\}$ \\
$\{\{1,4\},\{2,3,5\},\{6\}\}$ \\
$\{\{1,4,6\},\{2,3\},\{5\}\}$ \\
$\{\{1,4\},\{2,3,6\},\{5\}\}$ 
\end{tabular}
\begin{tabular}{c}
$\pi=(13)(2)(4)(5)(6)$ \\ \hline
$\{\{1,3,5\},\{2\},\{4\},\{6\}\}$ \\
$\{\{1,3\},\{2,5\},\{4\},\{6\}\}$ \\
$\{\{1,3\},\{2\},\{4,5\},\{6\}\}$ \\
$\{\{1,3,6\},\{2\},\{4\},\{5\}\}$ \\
$\{\{1,3\},\{2,6\},\{4\},\{5\}\}$ \\
$\{\{1,3\},\{2\},\{4,6\},\{5\}\}$ 
\end{tabular}
\begin{tabular}{c}
$\pi=(14)(2)(3)(5)(6)$ \\ \hline
$\{\{1,4,5\},\{2\},\{3\},\{6\}\}$ \\
$\{\{1,4\},\{2,5\},\{3\},\{6\}\}$ \\
$\{\{1,4\},\{2\},\{3,5\},\{6\}\}$ \\
$\{\{1,4,6\},\{2\},\{3\},\{5\}\}$ \\
$\{\{1,4\},\{2,6\},\{3\},\{5\}\}$ \\
$\{\{1,4\},\{2\},\{3,6\},\{5\}\}$ 
\end{tabular}
\begin{tabular}{c}
$\pi=(1)(23)(4)(5)(6)$ \\ \hline
$\{\{1,5\},\{2,3\},\{4\},\{6\}\}$ \\
$\{\{1\},\{2,3,5\},\{4\},\{6\}\}$ \\
$\{\{1\},\{2,3\},\{4,5\},\{6\}\}$ \\
$\{\{1,6\},\{2,3\},\{4\},\{5\}\}$ \\
$\{\{1\},\{2,3,6\},\{4\},\{5\}\}$ \\
$\{\{1\},\{2,3\},\{4,6\},\{5\}\}$ 
\end{tabular}
\begin{tabular}{c}
$\pi=(1)(24)(3)(5)(6)$ \\ \hline
$\{\{1,5\},\{2,4\},\{3\},\{6\}\}$ \\
$\{\{1\},\{2,4,5\},\{3\},\{6\}\}$ \\
$\{\{1\},\{2,4\},\{3,5\},\{6\}\}$ \\
$\{\{1,6\},\{2,4\},\{3\},\{5\}\}$ \\
$\{\{1\},\{2,4,6\},\{3\},\{5\}\}$ \\
$\{\{1\},\{2,4\},\{3,6\},\{5\}\}$ 
\end{tabular}
\end{center}

Our main formula \eqref{formula.anticommutator.main.2} yields
\begin{align*}
\freec{1,2}{ab+ba}=\freec{2,1}{ab+ba} =&\, 4\freec{3}{a}\freec{3}{b} + 12\freec{1}{a}\freec{2}{a}\freec{3}{b}+12\freec{1}{b}\freec{2}{b}\freec{3}{a}
+4\freec{2,1}{a}\freec{2}{b}\freec{1}{b}+4\freec{2,1}{b}\freec{2}{a}\freec{1}{a}\\
&+8\freec{2,1}{a}(\freec{1}{b})^3+8\freec{2,1}{b}(\freec{1}{a})^3
+16\freec{1,1}{a}\freec{1}{a}\freec{2}{b}\freec{1}{b}+16\freec{1,1}{b}\freec{1}{b}\freec{2}{a}\freec{1}{a}.   
\end{align*}

\section{Study of the indexing set}\label{Section: detail study of the sets}

From the last section, it is clear that even for small values of $n$ and $m$, the indexing set $\JJ_{2n,2m}$ is not simple. The goal of this section is to further understand the permutations in this set. We will do this using three different approaches, first we directly study the partitions $\pi$ in $\JJ_{2n,2m}$, then we study them through their Kreweras complement $\pi^{-1}\gammanm$, and finally we study the permutations $I\pi$ where $I:=(1,2)(3,4)(5,6)\cdots (2n+2m-1,2n+2m)\in S_{NC}(2n,2m)$.

\subsection{ \texorpdfstring{$\JJ_{2n,2m}$}{J(2n,2m)} and connected graphs}
Let us recall that in the first order version of our main result the sum is indexed by connected graphs, or equivalently, by  \ref{P1}, $\pi \lor I_{2n}=1_{2n}$. The latter is also equivalent to the condition that $\pi^{-1}\gamma_{2n}$ separates even, see Lemma \cite[Lemma 14]{mingo2009second}. In the second order case, this is no longer true, in Lemma \ref{Lemma: characterization of connected} below we explicitly find its equivalent statement in the annulus case. By Proposition \ref{Proposition: separates implies connected} we know that if $\pi^{-1}\gammanm$ separates even, this implies that $\GG_\pi$ is connected, but the converse is not true. Naturally, one might wonder how small is this set compared to the set of connected graphs. To answer this question, in this section we prove that every $\pi\in S_{\NC}(2n,2m)$ such that $\GG_\pi$ is connected must be of one of the following two types:
\begin{enumerate}[label=(C\arabic*)]
    \item $\pi^{-1}\gammanm$ separates even, or \label{C1}
    \item $\pi^{-1}\gammanm$ separates even except two even number $r,s$ where $r\in [2n]$ and $s\in [2m]$. \label{C2}
\end{enumerate}
This provides a better understanding of the possible non-crossing pairings for which the graph is connected. Recall that we use the notation $I$, $\boldsymbol{1}$ and $\gamma$ from \eqref{eq:not.I}, \eqref{eq:not.1}, and \eqref{eq:not.gamma}. Further, for any two permutations $\pi,\sigma\in S_{2n+2m}$ we may use the notation $\pi\vee\sigma := \Part{\pi}\vee \Part{\sigma}$, where $\vee$ is the join in the lattice $\cP(2n+2m)$ as in Definition \ref{defi.joins}.

\begin{proposition}\label{Proposition: separates implies connected 2}
Let $\pi\in S_{\NC}(2n,2m)$ such that $\Part{\pi}\vee I=\boldsymbol{1}$. Then there is no $r,s$ even in the same cycle of $\gammanm$ such that both are in the same cycle of $\pi^{-1}\gammanm$.
\end{proposition}
\begin{proof}
Suppose there exist $r,s$ in the same cycle of $\gammanm$ and the same cycle of $\pi^{-1}\gammanm$. By Lemma \ref{Lemma: equality if and only if less or equal than} it follows  that $|(r,s)|+|(r,s)\pi^{-1}\gammanm|=|\pi^{-1}\gammanm|.$ Thus
\begin{eqnarray*}
|\pi|+|\pi^{-1}\gammanm(r,s)| &=& |\pi|+|\pi^{-1}\gammanm|-|(r,s)| \\
&=& 2n+2m-\#(\pi)+2n+2m-\#(\pi^{-1}\gammanm)-1 \\
&=& 2n+2m-1.
\end{eqnarray*}
We conclude that
$$\#(\pi)+\#(\pi^{-1}\gammanm(r,s))=2n+2m+1.$$
On the other hand, from \cite[Equation 2.10]{mingo2004annular} we know that
$$\#(\pi)+\#(\pi^{-1}\gammanm(r,s))+\#(\gammanm(r,s))\leq 2n+2m+2\#(\pi \vee \gammanm(r,s)),$$
so
$$2n+2m+1+\#(\gammanm(r,s)) \leq 2n+2m+2\#(\pi \vee \gammanm(r,s)).$$
Recall that from hypothesis, $r,s$ are in the same cycle of $\gammanm$, so
$$2n+2m+1+3= 2n+2m+1+\#(\gammanm(r,s)) \leq 2n+2m+2\#(\pi \vee \gammanm(r,s)),$$
from where we conclude $\#(\pi \vee \gammanm(r,s)) \geq 2$. Finally, it is clear that $\pi$ must meet at least two clycles of $\gammanm(r,s)$ as otherwise $\pi\vee \gammanm \neq \boldsymbol{1}$. Therefore, it must hold that $\#(\pi \vee \gammanm(r,s))=2$, but $\Inm\leq \gammanm(r,s) \leq \pi \vee \gammanm(r,s)$ implies that $\Inm\vee \pi\leq \pi\vee\gammanm(r,s)$, which contradicts $\Inm\vee \pi =\boldsymbol{1}$.
\end{proof}

\begin{lemma}\label{Lemma: characterization of connected}
Let $\pi\in S_{\NC}(2n,2m)$. Then $\Part{\pi}\vee \Inm=\boldsymbol{1}$ if and only if either $\pi^{-1}\gammanm$ separates even or $\pi^{-1}\gammanm$ separates even number except by two even numbers $r,s$ with $r\in [2n]$ and $s\in [2m]$.
\end{lemma}
\begin{proof}
Assume first $\Part{\pi}\vee \Inm=\boldsymbol{1}$. If $\pi^{-1}\gammanm$ separates even we are done, otherwise there exist $r,s$ even in the same cycle of $\pi^{-1}\gammanm$. By Proposition \ref{Proposition: separates implies connected 2} it must be $r\in [2n]$ and $s\in [2m]$. We will show that $\pi^{-1}\gammanm$ separates even numbers except by $r$ and $s$. Notice that
\begin{eqnarray*}
|\pi|+|\pi^{-1}\gammanm(r,s)| &=& |\pi|+|\pi^{-1}\gammanm|-|(r,s)| \\
&=& 2n+2m-\#(\pi)+2n+2m-\#(\pi^{-1}\gammanm)-1 \\
&=& 2n+2m-1 \\
&=& |\gammanm(r,s)|.
\end{eqnarray*}
Thus $\pi\in \NC(\gammanm(r,s))$ and $\pi\vee \Inm=\boldsymbol{1}=\Part{\gammanm(r,s)}$, by \cite[Lemma 14]{mingo2009second} we conclude $\pi^{-1}\gammanm(r,s)$ separates even numbers as required.
Conversely, if $\pi^{-1}\gammanm$ separates even numbers, Proposition \ref{Proposition: separates implies connected} asserts that $\Part{\pi}\vee \Inm=\boldsymbol{1}$. So we are just left to prove that if $\pi^{-1}\gammanm$ separates even number except by two even numbers $r,s$ with $r\in [2n]$ and $s\in [2m]$, then $\Part{\pi}\vee \Inm=\boldsymbol{1}$. Indeed, note that
$$\#(\pi)+\#(\pi^{-1}\gammanm(r,s))=\#(\pi)+\#(\pi^{-1}\gammanm)-1=2n+2m-1.$$
Thus $\pi\in \NC(\gammanm(r,s))$ and $\pi^{-1}\gammanm(r,s)$ separates even numbers, by \cite[Lemma 14]{mingo2009second} it follows that $\pi\vee \Inm=\Part{\gammanm(r,s)}=\boldsymbol{1}$.
\end{proof}

From Lemma \ref{Lemma: characterization of connected} it follows that $\GG_{\pi}$ is connected if and only if $\pi$ satisfies either \ref{C1} or \ref{C2}.

\subsection{Understanding the Kreweras complement of \texorpdfstring{$\JJ_{2n,2m}$}{J(2n,2m)} }

Another way to understand the indexing set $\JJ_{2n,2m}$ is through its Kreweras complement $Kr_{n,m}(\JJ_{2n,2m})$. Namely, the set
$$\KK_{2n,2m}:=\{\sigma\in S_{\NC}(2n,2m): \GG_{\gammanm\sigma^{-1}}\mbox{ is bipartite and } \sigma \text{ separates even} \}.
$$

For the remaining of the paper, we will denote by 
\begin{equation}
\label{eq:not.even.odd}
\EE:=\{2,4,6,\dots,2n+2m\} \qquad \text{and} \qquad \OO:=\{1,3,5,\dots,2n+2m-1\}
\end{equation}
the sets of even and odd numbers, respectively.

\begin{proposition}
If $\sigma\in\KK_{2n,2m}$ has a cycle $C=(c_1c_2\cdots c_r)$ with only odd numbers $c_1,c_2,\dots, c_r\in \OO$, then $C$ has even size ($r$ is even).
\end{proposition}

\begin{proof}
Let us denote $\pi= \gammanm\sigma^{-1}$. Recall from Definition \ref{defi.Graph.pi.2} that $\pi\in\JJ_{2n,2m}$ has a natural bipartite decomposition $\pi = \pi ' \sqcup \pi ''$ where blocks in the same part do not share an edge of $\GG_\pi$. Let us denote $A'=\cup\pi'$ and $A''=\cup\pi''$. Clearly $A'\cup A''=[2n+2m]$ and $A'\cap A''=\emptyset$. Moreover, for each $i=1,\dots,2n+2m$, we have that $|A'\cap \{2i-1,2i\}|= 1 = |A''\cap \{2i-1,2i\}|$. In other words, one element of $\{2i-1,2i\}$ is in $A'$ and the other is in $A''$. In particular,
$$\gammanm(A'\cap \OO)=A''\cap \EE, \qquad \text{and} \qquad \gammanm(A''\cap \OO)=A'\cap \EE.$$ 

On the other hand, by how $A'$ and $A''$ are constructed, we also know that $\pi(A') =A'$ and $\pi(A'') =A''$. Equivalently
$$\pi^{-1}( A') =A', \qquad \text{and}\qquad \pi^{-1}( A'') =A''.$$ 
Putting these two together, we obtain that $\sigma=\pi^{-1}\gammanm$ satisfies
\begin{equation}
\sigma(A'\cap \OO) \subset A'', \qquad \text{and} \qquad \sigma(A''\cap \OO)\subset A'.
\end{equation}
In particular, if $C=(c_1c_2\dots c_r)$ is a cycle of $\sigma$ with only odd elements, then the elements of $C$ alternate between $A'$ and $A''$. Namely $c_{2j+1}\in A'$ and $c_{2j}\in A''$ or viceversa ($c_{2j+1}\in A''$ and $c_{2j}\in A'$). In any case, since $c_r$ and $c_1=\sigma(c_r)$ belong to different sets, we conclude that $r$ is even. Thus, the cycles of $\sigma$ containing only odd elements are of even size.
\end{proof}

\begin{remark}
The previous result tells us that necessary conditions for $\sigma$ to be in $\KK_{2n,2m}$ are that $\sigma$ separates even numbers, and the cycles of $\sigma$ with only odd numbers are of even size. Despite these two conditions where sufficient in the first order case, see \cite[Proposition 1.2]{perales2021anti}, these conditions are not sufficient in the second order case. For instance,
$\sigma= (13)(2)(4)\in S_{\NC}(2,2) $ separates even numbers, and the cycles with only odd numbers have even size (in this case only $(13)$). However, $\sigma\notin \KK_{2n,2m}$ because the graph associated to
$\pi:=\gamma_{2,2} \sigma^{-1} = (1432)$ consists of one vertex with two loops, and it is not bipartite. 
\end{remark}

One could also determine necessary and sufficient conditions for particular cases of permutations $\sigma$ so that $\sigma\in \KK_{2n,2n}$. A detailed study of the graph associated with $\sigma$ shows that the cycles of this graph are almost determined by the cycles of $\sigma$ with only odd elements. 

\begin{proposition}
Let $\sigma\in S_{\NC}(2n,2m)$ be such that $\sigma$ separates even and it has a single through cycle. Further, this cycle has only odd numbers. By \cite[Remark 3.4]{mingo2004annular} we write this cycle as, $$C=(o_1,\dots,o_l,o_{l+1},\dots,o_{l+s}),$$
where $o_1,\dots,o_l\in [2n]$ and $o_{l+1},\dots,o_{l+s}\in [2m]$. Then $\sigma\in \KK_{2n,2m}$ if and only if the following satisfy
\begin{enumerate}
    \item the cycles of $\sigma$ with only odd numbers have even size
    \item $l$ and $s$ are even.
\end{enumerate}
\end{proposition}

\begin{proof}
Let $u=\gamma\sigma^{-1}(o_{l+1})=\gamma(o_l)$ and $v=\gamma\sigma^{-1}(o_1)=\gamma(o_{l+s})$. Then both $u,v$ are even and $u\in [2n]$ and $v\in [2m]$. Let $\pi:= \gammanm\sigma^{-1}$. We claim $u$ and $v$ are in the same cycle of $\pi$. Indeed, first of all note that $o_1$ and $v$ are in the same cycle of $\pi$ as $\gammanm\sigma^{-1}(o_1)=\gammanm(o_{l+s})=v$. So it is enough to prove $o_1$ and $u$ are in the same cycle of $\pi$. To prove this, note that $u$ and $o_1$ are in the same cycle of $\gamma_{2n}\sigma|_{[2n]}^{-1}$ with $\gamma_{2n}=\gammanm|_{[2n]}=(1,\dots,2n)$. Indeed, $\gamma_{2n}\sigma|_{[2n]}^{-1}(o_1)=\gamma_{2n}(o_l)=\gammanm(o_l)=u$. However, since we have a single through cycle then $\sigma|_{[2n]}^{-1}$ and $\sigma^{-1}$ act the same in the set $[2n]\setminus\{o_1\}$. Hence, the cycles of $\gammanm_{2n}\sigma|_{2n}^{-1}$ and $\gammanm\sigma^{-1}$ that contain $u$ are of the form
$$(u,a_1,\dots,a_s,o_1)$$
and 
$$(u,a_1,\dots,a_s,o_1,b_1,\dots,b_t)$$
respectively. Here $a_i\in [2n]$ and $b_i\in[2m]$ with possibly no $b_i's$. We must note that both cycles are the same from $u$ to $o_1$ and therefore $u$ and $o_1$ are in the same cycle of $\pi$ as required.

Let $\tau=(u,v)$ and $\pi'=\pi\tau$ and $\gamma'=\gamma\tau$. A classical argument on non-crossing permutations shows $\pi'\in \NC(\gamma')$. That is,
$\#(\pi')+\#(\pi'^{-1}\gamma)=2n+2m+1.$
In the last equality we use that $u$ and $v$ are in the same cycle of $\pi$ so that $\#(\pi')=\#(\pi)+1$. Once we are back in the disk case (i.e. the permutation $\gamma'$ has only once cycle) then we are able to use all the results already obtained in \cite{perales2021anti}. In particular, \cite[Proposition 1.2]{perales2021anti} shows that the cycles of $\pi'^{-1}\gamma'$ with only odd numbers determine the cycles of $\GG_{\pi'}$. Note that the last claim follows directly from the fact that $\gamma'$ is a permutation that has only even numbers at its even positions and odd numbers at its odd positions. In addition, the edges that join a number at an odd position $d$ with $\gamma'(d)$ are still the same edges $(1,2),\dots,(2n+2m-1,2n+2m)$. Moreover, observe that the cycles with only odd numbers of $\pi'^{-1}\gamma'$ are the same as the cycles with only odd numbers of $\sigma$. In order to determine if $\GG_{\pi}$ is bipartite, it is enough to prove that all cycles of $\GG_{\pi}$ have even size. The advantage is that the graphs $\GG_{\pi'}$ and $\GG_{\pi}$ are strongly related as they have exactly the same edges and vertices except by two vertices of $\GG_{\pi'}$ which are merged into the same vertex of $\GG_{\pi}$, see for instance figure \ref{Figure: Merging one cycle}. These vertices correspond to the cycles of $\pi'$ that contain $u$ and $v$. The cycle of $\pi'$ that contains $u$ also contains $o_{l+1}$ while the cycle that contains $v$ also contains $o_1$. Part $B$ of the proof of \cite[Proposition 1.2]{perales2021anti} precisely shows that the cycles $(o_1,\dots,o_p)$ with only odd numbers of $\sigma$ correspond to a cycle of $\GG_{\pi'}$ where each $o_1$ belongs to a vertex of $\GG_{\pi'}$ and the vertices containing $o_t$ and $o_{t+1}$ are consecutive vertices in the cycle. From this fact we know that the graph $\GG_{\pi}$ is obtained from merging the vertices that contains $o_1$ and $o_{l+1}$ of $\GG_{\pi'}$ which at the same time belong to a cycle of $l+s$ size. In order to get only cycles of even size in $\GG_{\pi}$ we need two conditions to be satisfied. First, the cycles of $\sigma$ with only odd numbers distinct of $C$ must have even size. Secondly, we need to ask that $o_{1}$ and $o_{l+1}$ are at even distance, that is, $l$ and $s$ are even, so that when merging the vertices of $\GG_{\pi'}$ that belong to one cycle produces two new cycles of even size each in $\GG_\pi$.
\end{proof}

\begin{figure}[h!]
\begin{center}
\includegraphics[width=0.9\linewidth]{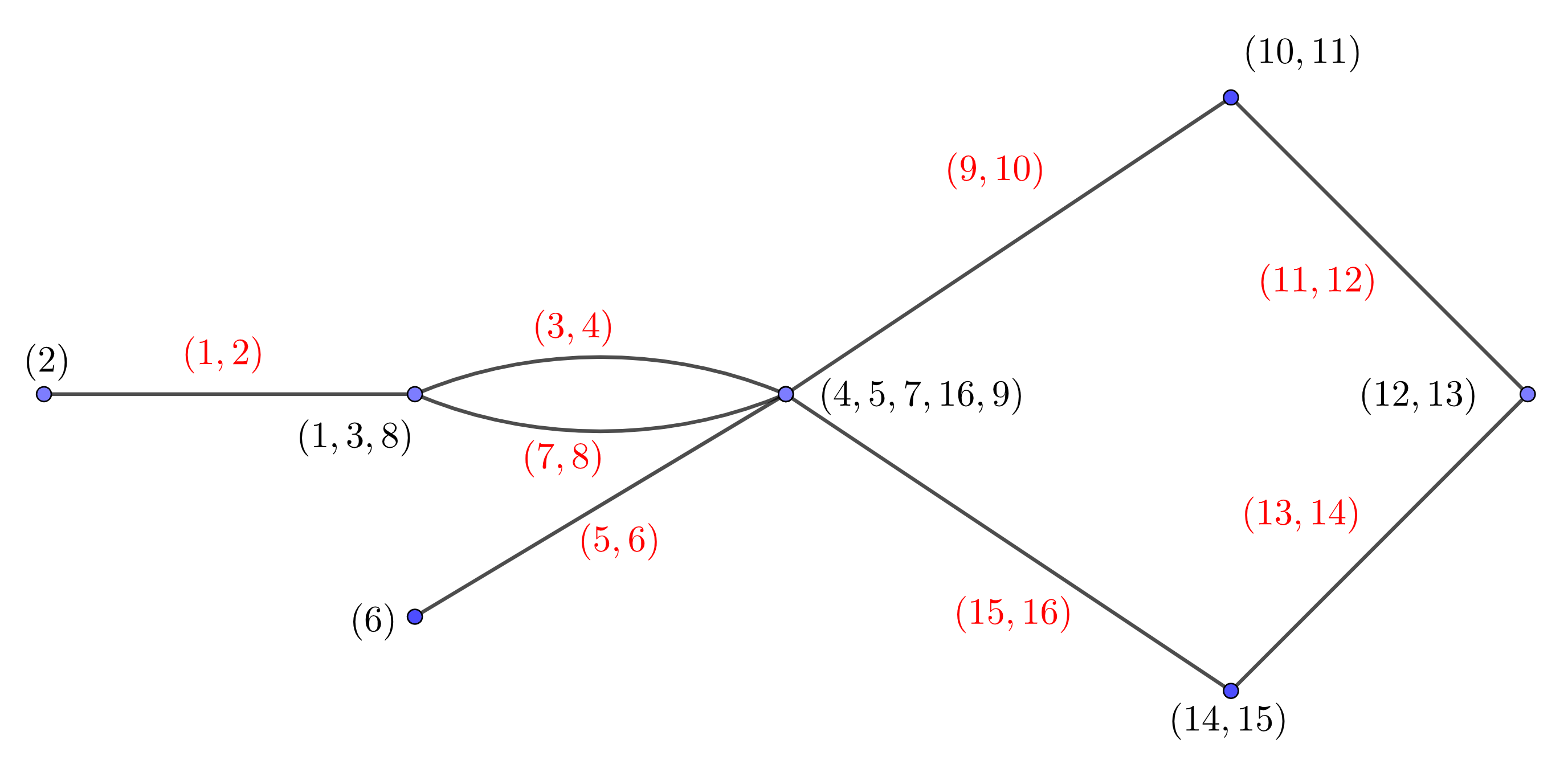}\\
\includegraphics[width=0.9\linewidth]{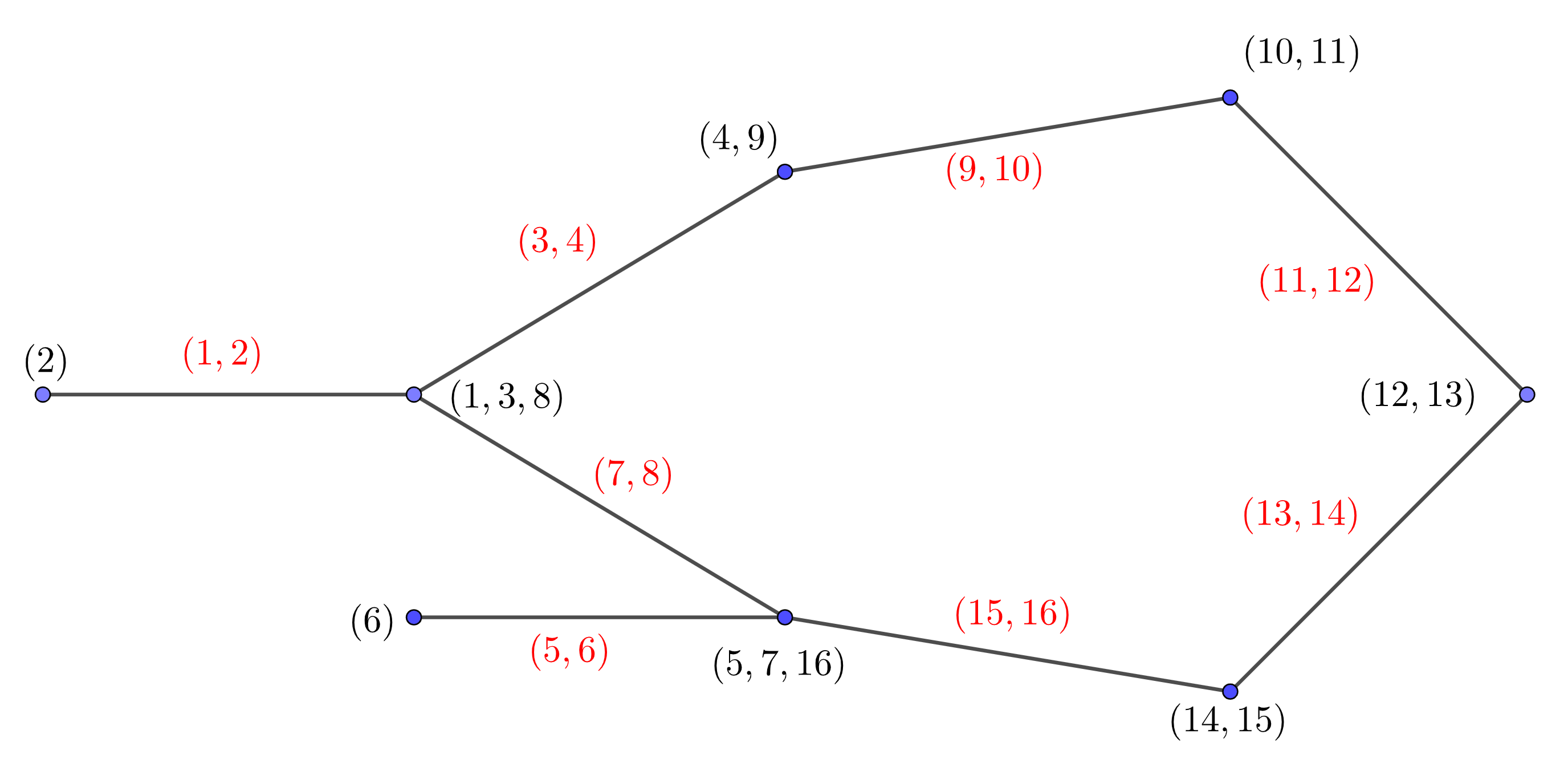}
\caption{The Graphs $\GG_{\pi}$ (up) and $\GG_{\pi'}$ (down) corresponding to $\pi=(1,3,8)(2)(4,5,7,16,9)(6)(10,11)(12,13)(14,15)$ and $\pi'=\pi\tau=\pi(4,16)$. In this case $\sigma=(7,3,9,11,13,15)(1,2)(8)(5,6)(4)(16)(14)(10)(12)$ so that $u=4$ and $v=16$. We label the vertices in black and the edges in red.}
    \label{Figure: Merging one cycle}   
\end{center}
\end{figure}

Before moving on, let us comment that the main difficulty of a general proof is having two even numbers, $u\in[2n]$ and $v\in[2m]$, in the same cycle of $\pi$ where $u=\gammanm\sigma^{-1}(o_{l+1})$ and $v=\gammanm\sigma^{-1}(o_{1})$ and such that $o_1,o_{l+1}$ are in the same cycle of $\sigma$. The latest is not always satisfied. In terms of the graphs $\GG_{\pi}$ and $\GG_{\pi'}$ this means that we merge two vertices of the graph $\GG_{\pi'}$ to obtain the graph $\GG_\pi$. When this condition is not satisfied then we must look for distinct even numbers $u\in [2n]$ and $v\in [2m]$ in the same cycle of $\pi=\gamma\sigma^{-1}$ so that the latter still holds true. In this case however, it might be possible that $o_1=\pi^{-1}(v)=\pi'^{-1}(v)$ and $o_{l+1}=\pi^{-1}(u)=\pi'^{-1}(u)$ are in distinct cycles of $\sigma$. In terms of the graphs this means that now we merge two vertices of $\GG_{\pi'}$ not necessarily in the same cycle and therefore determining if every path that connects $u$ and $v$ has even length is not that clear just from the cycle structure of $\sigma$.

\begin{remark}
One could also come up with other necessary conditions for $\sigma$ to be in $\KK_{2n,2m}$. An example of such condition is the following.

If $\sigma\in\KK_{2n,2m}$ and $\sigma$ has only one thru cycle $C$, then the cycle contains at least two elements of $[2n]$ and at least two elements of $[2m]$. 

However this condition, together with the previous ones, is not enough to guarantee that a permutation is in $\KK_{2n,2m}$. For instance, one can notice that $(18)(27)(356)(4)\in S_{\NC}(4,4)$ satisfy this and all the previous conditions, but the graph associated to $\pi:=\gamma_{4,4} \sigma^{-1} = (154)(28)(37)(6)$  is not bipartite. 
\end{remark}

An interesting question for future work, is to give a description of $\KK_{2n,2m}$ purely in terms of permutations. In other words, to express the condition that the graph associated to $\pi:=\gammanm \sigma^{-1}$ is bipartite, using only the permutations, without the need to draw the graph $\GG_\pi$.

\subsection{The action of \texorpdfstring{$I$}{I} on \texorpdfstring{$\JJ_{2n,2m}$}{J(2n,2m)}}
\label{ssec:I.pi}

To finish this section, we give another approach to better understand the graph $\GG_\pi$ using permutations. This approach will be very useful in Section \ref{secc:examples} to give a concise proof of Proposition \ref{Proposition: example of semicircles.intro} and to prove Theorem \ref{Thm.commutator.main.2.intro}.
The idea is to study the permutation $I\pi$ where $I$ is the unique permutation that pairs consecutive elements 
$$
I:=(12)(34)(56)\cdots (2n+2m-1,2n+2m)\in S_{NC}(2n,2m).
$$
Two simple but useful facts are that $I^2=id$ is the identity permutation, and $I$ changes the parity of the numbers, namely, $I(\EE)=\OO$ and $I(\OO)=\EE$.

\begin{proposition}
\label{prop:I.pi}
 Let $\pi\in\JJ_{2n,2m}$. Then the permutation $I\pi$ has cycle decomposition
 $$I\pi=C^{\text{out}}C^{\text{inn}}O_1 \dots O_k,$$
where:
 \begin{enumerate}
     \item Every cycle has even length.
     \item $C^{\text{out}}$ is a cycle containing all the elements $2,4,\dots, 2n$ appearing in increasing order (with possibly odd numbers in between).
     \item $C^{\text{inn}}$ is a cycle containing all the elements $2n+2,2n+4,\dots, 2n+2m$ in increasing order (with possibly odd numbers in between). 
     \item The cycles $O_1, \dots, O_k$ contain only odd elements, and coincide precisely with the cycles of $\gamma^{-1}\pi$ that only contain odd elements. 
 \end{enumerate}
\end{proposition}

\begin{proof}
To prove part 1, recall from Definition \ref{defi.Graph.pi.2} that $\pi\in\JJ_{2n,2m}$ has a natural bipartite decomposition $\pi = \pi ' \sqcup \pi ''$  where blocks in the same part do not share an edge of $\GG_\pi$. Let us denote $A'=\cup\pi'$ and $A''=\cup\pi''$. Clearly $A'\cup A''=[2n+2m]$ and $A'\cap A''=\emptyset$. Moreover, for each $i=1,\dots,2n+2m$, we have that $|A'\cap \{2i-1,2i\}|= 1 = |A''\cap \{2i-1,2i\}|$. In other words, one element of $\{2i-1,2i\}$ is in $A'$ and the other is in $A''$. In particular,
$$I(A')=A'', \qquad \text{and} \qquad I(A'')=A'.$$ 
On the other hand, by how $A'$ and $A''$ are constructed, we also know that 
$$\pi( A') =A', \qquad \text{and}\qquad \pi( A'') =A''.$$ 
Putting these two together, we obtain that $I \pi$ satisfies
\begin{equation}
I \pi(A') \subset A'', \qquad \text{and} \qquad I \pi(A'')\subset A'.
\end{equation}
In particular, if $C=(c_1c_2\dots c_r)$ is a cycle of $I \pi$, then the elements of $C$ alternate between $A'$ and $A''$. Namely $c_{2j+1}\in A'$ and $c_{2j}\in A''$ or viceversa ($c_{2j+1}\in A''$ and $c_{2j}\in A'$). In any case, since $c_r$ and $c_1=I\pi(c_r)$ belong to different parts ($A'$ or $A''$), we conclude that $r$ is even. Thus, the cycles of $I \pi$ are of even size.

To prove the remaining parts, recall that $\pi\in \JJ_{2n,2m}$ implies that $\pi^{-1}\gammanm$ separates even elements, thus its inverse $\gammanm^{-1}\pi$ also separates even, meaning that it can be expressed as 
$$\gammanm^{-1}\pi=E_2E_4\cdots E_{2n+2m} O_1 \dots O_k$$
where cycles $O_i$ contains only odd elements for $i=1,2,\dots,k$ and for $j=2,4,\dots,2n+2m$, the cycles containing an even element are of the form $E_{j}=(j,o_{j,1},o_{j,2},\dots, o_{j,l(j)})$ with $l(j)\geq 0$, and $o_{j,i}$ odd for all $i=1,\dots,l(j)$.

On the other hand, one can notice that
$$I\gammanm=(1)(3)\dots(2m+2n-1)(2,4\dots,2n)(2n+2,\dots, 2m+2n),$$ 
fixes all odd numbers.  Therefore, 
$$I\pi=(I\gammanm)(\gammanm^{-1}\pi)=C^{\text{out}}C^{\text{inn}}O_1 \dots O_k$$
where $O_1,\dots O_k$ are the same cycles of $\gammanm^{-1}\pi$ (containing only odd elements), this proves part 4. Finally, by looking at the orbit of $2$ one can check that cycles $E_2,E_4,\dots, E_{2n}$ are glued together. Namely, 
$$C^{\text{out}}:=(2,o_{2,1},\dots, o_{2,l(2)},4,o_{4,1},\dots, o_{4,l(4)}, 6, \dots\dots,2n,o_{2n,1},\dots, o_{2n,l(2n)})$$
is a cycle of $I\pi$. And similarly, the cycles $E_{2n+2},\dots, E_{2n+2m}$ are glued together into the cycle
$$C^{\text{inn}}:=(2n+2,o_{2n+2,1},\dots, o_{2n+2,l(2n+2)},2n+4,\dots\dots, 2n+2m,\dots)$$
of $I\pi$, as desired. 
\end{proof}

We now explain how the permutation $I\pi$ is helpful to understand the graph of $\pi$. The intuition is very clear, the cycles of $I$ are precisely the edges of the graph $\GG_\pi$, while the cycles of $\pi$ are the vertices of the graph. Thus one can think of the permutation $I\pi$ as a way to travel through $\GG_\pi$. This can be formally stated as follows.
\begin{lemma}
For every $j\in[2n+2m]$, the cycle of $\pi$ containing $j$ and the cycle of $\pi$ containing $I\pi(j)$ are connected by an edge in $\GG_\pi$. 
\end{lemma}

\begin{proof}
Notice that $\pi(j)$ and $j$ are in the same cycle of $\pi$. We conclude by noticing that the cycles of $\pi$ containing $\pi(j)$ and $I(\pi(j))$ are connected by definition.
\end{proof}

\begin{remark}
Observe that the previous lemma implies that the cycles $C^{\text{out}},C^{\text{inn}}O_1, \dots, O_k$ of $I\pi$ correspond to cycles of the graph $\GG_\pi$. Moreover, since $C^{\text{out}}, C^{\text{inn}}$ together contain all the even numbers $\EE$, then the two corresponding cycles in $\GG_\pi$ contain all the edges (and thus all the vertices) in $\GG_\pi$.
\end{remark}

Another implication is that if we construct a graph $ \vec{\GG}_\pi$ with the cycles of $\pi$ as vertices and for every $j\in[2n+2m]$ we draw an directed edge from the cycle containing $j$ and the cycle containing $I\pi(j)$. Then, by pairing directed edges with the same endpoints (but opposite directions) into one undirected edge, we end up with $\GG_\pi$. 

We will explain this process to draw $\GG_\pi$ in more detail. Let us begin by separating the cycles $C\in\pi$ in 3 types:
 \begin{itemize}
        \item $C$ is an \emph{outer} block if $C\subset [2n]:=\{1,\dots,\dots, 2n\}$.
        \item $C$ is an \emph{inner} block if $C\subset [2m]:=\{2n+1,\dots,2n+2m\}$
        \item  $C$ is a \emph{intermediate} block if it contain at least one element of each $[2n]$ and $[2m]$.
    \end{itemize}
To draw the graph $\GG_\pi$ using $I\pi$ the procedure is as follows:
\begin{itemize}
\item \textbf{Step 1.} We first draw the cycle of $G_\pi$ corresponding to $C^{\text{out}}$ in clockwise direction. Namely we put the cycles (vertices) containing the elements of  $C^{\text{out}}$ (say starting in 2) and we join consecutive vertices with a directed edge. One should notice that there is a simple cycle that contains all the inner and middle blocks (with possibly some outer blocks), we will call this the \emph{exterior core cycle}. To help with the intuition, we should draw the remaining (outer) blocks in the ``exterior'' of this core cycle.
Notice that while drawing $C^{out}$, we may go back to the same vertex several times (forming simple cycles in the way).
Notice also that we may draw the same edge twice (but at most twice, each in a different direction). These edges correspond to pairs $(2i-1,2i)$ with $i\leq n$ where both elements appear in $C^{\text{out}}$. We call these \emph{flexible exterior edges}. 
\item \textbf{Step 2.} We now do a similar process with $C^{\text{inn}}$, but we will only draw in the interior of the ``exterior core cycle" from step 1. Notice that each inner and middle block appearing in the exterior core cycle will also contain elements in $C^{\text{inn}}$, so we should make these coincide. In order to do this we will draw the vertices with elements in $C^{\text{inn}}$ in the order they appear, but now in counter-clockwise direction. Same as before we join consecutive vertices with a directed edge. After this procedure we should also be able to identify an \emph{interior core cycle} that is simple and contains all the outer and middle blocks (with possibly some inner blocks). Notice that some edges from the  interior core coincide with the edges of the exterior core but have opposite direction. 
Finally we may have drawn the same edge twice (not in the core cycle), corresponding to pairs $(2i-1,2i)$ with $i\geq n+1$ where both elements appear in $C^{\text{out}}$, we call these \emph{flexible interior edges}.
  \item At this stage we already drew all the vertices and edges in $\GG_\pi$. So, one can simply convert all the (directed) edges into undirected edges to obtain $\GG_\pi$. However, in order to complete the directed graph $\vec{\GG}_\pi$, one should repeat the previous procedure with the remaining blocks of $I\pi$, namely $O_1, \dots, O_k$. By drawing the vertices and directed edges in counter-clockwise direction we will notice that each cycle $O_i$ yields a cycle of $\GG_\pi$. Moreover these cycles are formed precisely by edges that where only drawn in one direction (after Step 1 and Step 2), and we are now drawing the edges in the opposite direction. Notice also that these cycles might be in the exterior (if $O_i\subset [2n]$), in the interior (if $O_i\subset[2m]$) or embedded in between the interior core cycle and the exterior core cycle. 
\end{itemize}

Check Figure \ref{Figure:Ipi} for an example on how the graph of $$\pi=(1,3,8)(2)(4,5,7,16,9)(6)(10,11)(12,13)(14,15)$$ is constructed.

\begin{figure}[h!]
\begin{center}
\includegraphics[width=0.9\linewidth]{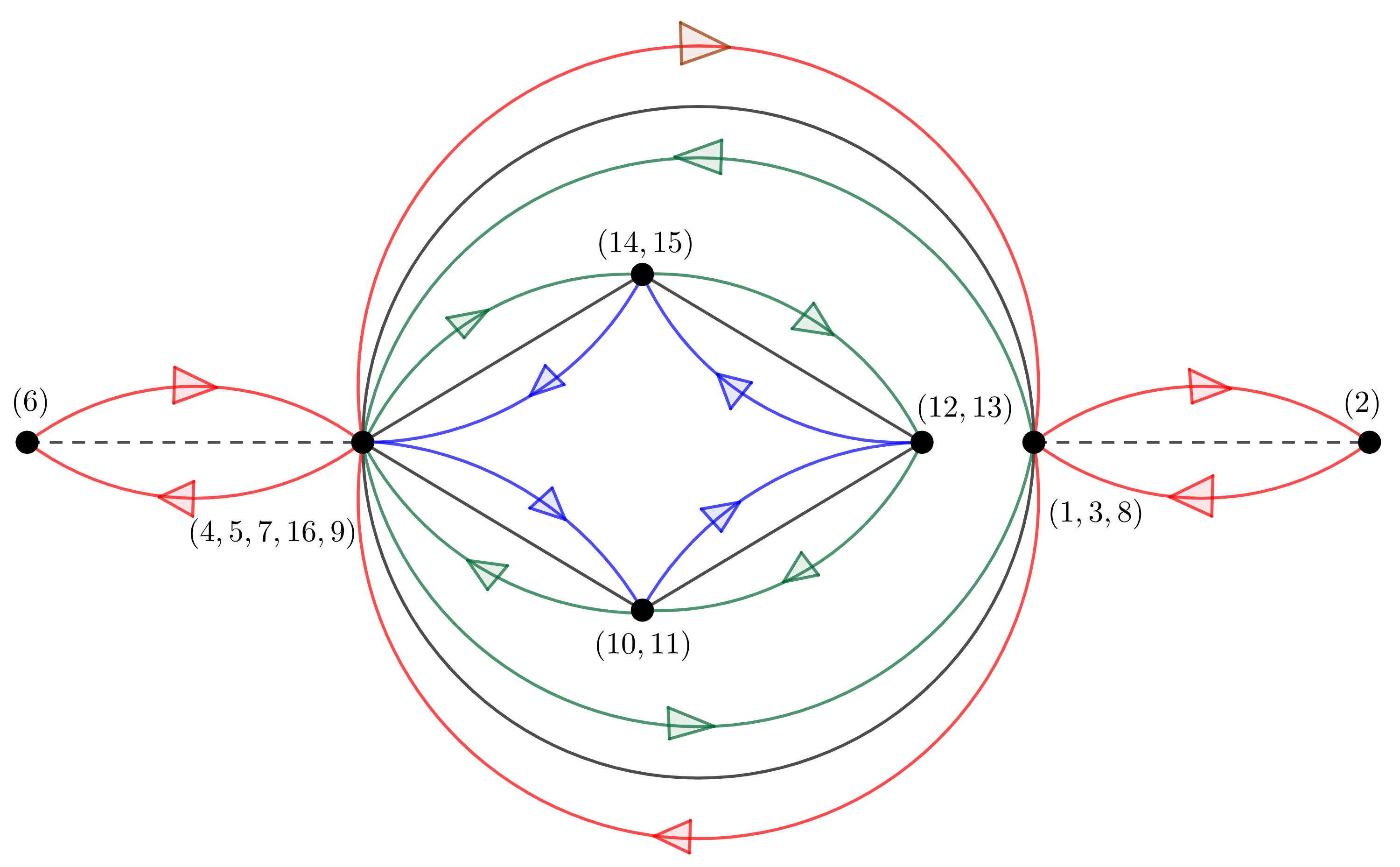}
\caption{The Graph $\GG_{\pi}$ of Figure \ref{Figure: Merging one cycle} corresponding to $\pi=(1,3,8)(2)(4,5,7,16,9)(6)(10,11)(12,13)(14,15)$. In this case $I\pi=C^{\text{out}}C^{\text{inn}}O_1=(2,1,4,6,5,8)(10,12,14,16)(3,7,15,13,11,9)$. In red, blue and green the cycles $C^{\text{out}}$ and $C^{\text{inn}}$ and $O_1$ respectively. The two dotted edges correspond to flexible (exterior) edges.}
    \label{Figure:Ipi}   
\end{center}
\end{figure}

From this construction one should be clear that the set of graphs $\GG_\pi$ that can be obtained from an arbitrary $\pi\in \JJ_{2n,2m}$ is very particular. First of all, the graph must be planar. And it contains two distinguished simple cycles (the interior and exterior core cycles) that coincide in some vertices. These core cycles divide the plane in two regions (interior and exterior) where the remaining vertices belong. Regarding the connection to $I\pi$ one can draw the graph in such a way that one can go around all the exterior vertices by following $C^{\text{out}}$ and one can go around all the interior edges by following $C^{\text{inn}}$. Finally, let us point out a fact that will play a fundamental role when studying the commutator in Section \ref{sec:commutator}.   

\begin{remark}
\label{rem:cutting.edges}
 The edges that appear twice in a cycle of $I\pi$, namely, those pairs of elements $2i-1,2i$ that are both in $C^{\text{out}}$ or both in $C^{\text{inn}}$, correspond to the flexible edges (interior or exterior) of the graph. Notice that the flexible edges are precisely the \emph{cutting edges} of $\GG_\pi$. Recall that a connected graph $G$ has a cutting edge $e$ if the graph $G\setminus\{e\}$, obtained by removing $e$ from $G$, is no longer connected. Alternatively a connected graph has a cutting edge, if there is an edge that does not belong to any simple cycle. This observation means that $\GG_\pi$ has no cutting edges if and only if $\pi$ is admissible, as in Definition \ref{def:admissible}. 
\end{remark}


\section{Commutator}
\label{sec:commutator}

The commutator $ab-ba$ of a pair of free random variables can be seen as a signed counterpart of the anti-commutator. When studying the cumulants, one can parallel the approach used to study the anti-commutator to express it as sum of products of cumulants of $a$ and $b$. The difference here, is that the terms in the sum may have positive or negative signs, which may lead to cancellation of terms in the sum.

\subsection{Adapting the formulas for the anti-commutator}
The first step is just a modification of Proposition \ref{prop.basic.anticomm.2}, and its proof is analogous, so we omit it. Recall from Notation \ref{nota.epsilons} that given a $n$-tuple $\varepsilon \in \{1,*\}^n$, we denote by $(ab)^\varepsilon:=((ab)^{\varepsilon(1)},(ab)^{\varepsilon(2)},\dots,(ab)^{\varepsilon(n)})$ the $n$-tuple with entries $ab$ or $ba$ dictated by the entries of $\varepsilon$, and $(a,b)^{\varepsilon}$ is the $2n$-tuple obtained from separating the $a$'s from the $b$'s in $(ab)^\varepsilon$.  To study the commutator we also need to keep track of the sign associated to each $n$-tuple, so we multiply by $-1$ for each $*$ in $\varepsilon$:
$$s(\varepsilon):=(-1)^{|\{ 1\leq i\leq n:\, \varepsilon(i)=*\}|}.$$

With this notation in hand, Proposition \ref{prop.basic.anticomm.2} can be readily adapted to the commutator.

\begin{proposition}
\label{prop.basic.comm.2}
The second order free cumulants of the commutator $ab-ba$ of two second order free random variables $a,b$ satisfy the following formula for all $n,m\in \nn$:
\begin{equation}
\freec{n,m}{ab-ba} =\sum_{\substack{(\cU,\pi)\in \PS_{NC}(2n,2m),\\ \pi^{-1}\gammanm \text{ sep. even}}} \sum_{\substack{\varepsilon\in \{1,*\}^{n+m}\\ \{A(\varepsilon),B(\varepsilon)\}\geq \cU}} s(\varepsilon) \kappa_{(\cU,\pi)}((a,b)^{\varepsilon}). 
\end{equation}
Here, we use the notation $\gammanm:=(1,\dots, 2n)(2n+1,\dots, 2n+2m)\in S_{2n+2m}$.
\end{proposition}

Same as before, one can check that only permutations $\pi\in S_{\NC}(2n,2m)$ that actually contribute to the sum are those such that $\GG_\pi$ is connected and bipartite. Namely, only partitions $\pi\in \JJ_{2n,2m}$ contribute to the sum. Moreover, for each $\pi$ that contributes there are exactly two $(m+n)$-tuples that contribute, $\varepsilon_\pi$ and $\varepsilon'_\pi$. 
Notice that since these two tuples do not coincide in any entry, then $s(\varepsilon'_\pi)=(-1)^{m+n} s(\varepsilon_\pi)$. Given a permutation $\pi \in \JJ_{2n,2m}$ we use the notation $s(\pi)=s(\varepsilon_\pi)$. Recall that one of the parts in the bipartition is $\pi':=\pi|A(\varepsilon_\pi)$ where
\[
A(\varepsilon_\pi):=\{ 2i-1 |1\leq i\leq n,\ \varepsilon_\pi(i)=1\}\cup \{ 2i |1\leq i\leq n,\ \varepsilon_\pi(i)=*\}.
\]
Thus we have the alternative description
\begin{equation}
\label{not:sign.pi}
s(\pi):=(-1)^{|\{ k\in A(\varepsilon_\pi):\, k \text{ is even}\}|}
\end{equation}

Then Theorem \ref{Thm.anticommutator.main.2} can be adapted to the commutator case as follows.

\begin{theorem}
\label{Thm.commutator.main}
Consider two second order free random variables $a$ and $b$, and let $(\freec{n}{a})_{n\geq 1}$, $(\freec{n,m}{a})_{n,m\geq 1}$, $(\freec{n}{b})_{n\geq 1}$, $(\freec{n,m}{b})_{n,m\geq 1}$ and $(\freec{n,m}{ab-ba})_{n,m\geq 1}$ be the sequence of first and second order free cumulants of $a$, $b$ and $ab-ba$, respectively. Then, for every $n,m\geq 1$ one has
\begin{align*}
\freec{n,m}{ab-ba} =& \sum_{\substack{\pi \in \JJ_{2n,2m}\\ \pi=\pi' \sqcup \pi''}}s(\pi) \left(\freec{\pi'}{a} \, \freec{\pi''}{b}  + (-1)^{m+n} \freec{\pi'}{b} \, \freec{\pi''}{a}  \right) \\ & + \sum_{\substack{(\cU,\pi)\in \NCac_{2n,2m}\\ \pi =\pi'\sqcup \pi'' \sqcup A\sqcup B }} s(\pi) \left(\freec{|A|,|B|}{a}\, \freec{\pi'}{a} \, \freec{\pi''}{b} + (-1)^{m+n} \freec{|A|,|B|}{b}\, \freec{\pi'}{b} \, \freec{\pi''}{a}  \right).
\end{align*}
\end{theorem}

Contrary to the anti-commutator case where all the terms where positive, now each partition $\pi$ has a sign $s(\pi)$ that can potentially cancel with other term. We can actually cancel several terms to simplify the sum and obtain Theorem \ref{Thm.commutator.main.2.intro} advertised in the introduction. To achieve this, we need to cancel several terms in both sums. We will firs explain how to simplify the second sum by reducing to the first order case and using the results of Nica and Speicher \cite{nica1998commutators}. For the first sum we can still cancel terms following the same philosophy used in \cite{nica1998commutators}, but finding such canceling pairs is not a simple task and becomes technical. Since the underlying idea for the cancellation coming from the graph $\GG_\pi$ is easier to grasp, we will first give an intuitive explanation of the cancellation phenomenon, followed by rigorous proof that relies on the permutation $I\pi$ studied in Section \ref{ssec:I.pi}.

\subsection{Intuition behind the cancellation in first sum.}  

Simplifying the second sum in Theorem \ref{Thm.commutator.main} follows easily from reducing to the first order case and using the results of Nica and Speicher \cite{nica1998commutators}. Simplifying the first sum requires more effort, but we can still cancel terms following the same philosophy used in \cite{nica1998commutators}. More specifically, we will find pairs of partitions $\pi_1, \pi_2 \in \JJ_{2n,2m}$ such that
\begin{enumerate}[label=(C\arabic*)]
    \item $\freec{\pi'_1}{a} \, \freec{\pi''_1}{b} =\freec{\pi'_2}{a} \, \freec{\pi''_2}{b}$, and \label{condition.1}
    \item $s(\pi_1)=-s(\pi_2)$. \label{condition.2}
\end{enumerate}
Notice that these two conditions together imply that the terms corresponding to $\pi_1$ and $\pi_2$ cancel each other:
\begin{equation}
\label{eq:first.sum.cancel}
s(\pi_1) \left(\freec{\pi'_1}{a} \, \freec{\pi''_1}{b}  + (-1)^n \freec{\pi'_1}{b} \, \freec{\pi''_1}{a}  \right)+s(\pi_2) \left(\freec{\pi'_2}{a} \, \freec{\pi''_2}{b}  + (-1)^n \freec{\pi'_2}{b} \, \freec{\pi''_2}{a}  \right)=0.
\end{equation}

In what follows, we will identify a subset of $\JJ_{2n,2m}$ that contains pairs of permutations that satisfy both \ref{condition.1} and \ref{condition.2} and thus cancel each other. After canceling them we end up with a sum indexed by a smaller set of \emph{admissible} permutations. Despite that our formula is simpler, it is not cancellation free, so it might be possible to simplify it further. 

In order to identify partitions satisfying \ref{condition.1}, the key observation is that the product $\freec{\pi'}{a} \, \freec{\pi''}{b}$ only depends on the graph $\GG_\pi$, this is because the bipartition of $\pi=(\pi',\pi'')$ is determined by the graph $\GG_\pi$. Moreover, the sizes of the blocks of $\pi$ (that ultimately govern the product) are precisely the sizes of degrees of the vertices of the graph $\GG_\pi$. This means that the permutations having the same associated graph will have the same product. Namely:
$$\pi_1, \pi_2\in \JJ_{2n,2m} \text{ and } \GG_{\pi_1}=\GG_{\pi_2} \qquad \Rightarrow \qquad   \ref{condition.1}.$$
The second key observation is that to obtain a pair of partitions that cancel each other we need to focus on the graphs $\GG_\pi$ with a \emph{cutting edge}, or equivalently in permutations $\pi$ such that there is pair $(2i-1,2i)$ where both elements are in the same cycle of $I\pi$ (either in 
$C^{\text{out}}$ or  $C^{\text{inn}}$), see Remark \ref{rem:cutting.edges}.

Our approach is to find a partition $\pi_1$ such that $\GG_{\pi_1}$ has a cutting edge, and construct another partition $\pi_2$ such that $\GG_{\pi_2}=\GG_{\pi_1}$ and $s(\pi_2)=-s(\pi_1)$, thus providing a pair of partitions satisfying \ref{condition.1} and \ref{condition.2}.

Constructing the bijection $(\pi_1,\pi_2)$ mentioned above is the main technical part. Let us give a final piece of intuition on how this is done. Given a partition $\pi_1$ such that $\GG_{\pi_1}$ has a cutting edge, we can separate the vertices of the graphs in two, by grouping the vertices in each connected component obtained after removing the edge. Notice that one component has the all the vertices in core cycles, we call this the \emph{fixed} component and the other the \emph{moving} component. For instance, if the flexible edge is exterior, then the vertices in the moving component will have only consecutive elements in $[2n]$ (assuming 1 is next to $2n$). To construct a partition $\pi_2$ we fix all the cycles of $\pi_1$ that belong to the fixed component, and then modify the cycles in $\pi_1$ in the moving component (shifting by 2 most of the elements) in such a way that we get the same graph, the new partition $\pi_2$ is still not crossing, and such that $s(\pi_1)=-s(\pi_2)$.

\subsection{Rigorous explanation of cancellation in the first sum.} Recall from Definition \ref{def:admissible} that a permutation $\pi\in\JJ_{2n,2m}$ is said to be admissible if $I\pi$ separates $2i-1$ from $2i$ for all $1\leq i \leq n+m$. The permutations that are not admissible will be the ones that we can cancel.

\begin{definition}
Given a permutation $\pi\in\JJ_{2n,2m}$ for $1\leq i \leq n+m$ if the pair $(2i-1,2i)$ belong to $C^{\text{out}}$ or both belong to $C^{\text{inn}}$ then we say that the pair is flexible. If $\pi$ has at least one flexible pair, then we say that $\pi$ is cancelable, and denote the set of all cancelable permutations by $\FF_{2n,2m}$.
\end{definition}
Notice that $\JJ_{2n,2m}=\FF_{2n,2m}\sqcup \AA_{2n,2m}$ is the disjoint union of cancelable and admissible permutations. To construct our bijection, for a $\pi\in\FF_{2n,2m}$ we want to have a distinguished flexible pair. Thus we will use the following notation.
\begin{notation}
\label{not:smallest.pair}
Given $\pi\in\FF_{2n,2m}$ we order the elements of $C^{\text{out}}\cup C^{\text{inn}}$ by listing in order all the elements of $C^{\text{out}}$ (starting with 2) followed by the elements of $C^{\text{inn}}$ (starting with $2n+2$). Notice that list looks as follows:
$$L_\pi= 2,\dots, 4,\dots,6, \dots,\dots, 2n+2m,\dots.$$
where the empty spaces may contain some odd elements. We denote by $f_\pi$ the first element in the list that belongs to a flexible pair. Letting $i$ be such that $f_\pi\in \{2i-1,2i\}$ then we say that $(2i-1,2i)$ is \emph{smallest flexible pair}. If $f_\pi$ is even we say the $\pi$ is an even cancelable permutation, and denote the set of such permutations by $\FF^E_{2n,2m}$. Otherwise, if $f_\pi$ is odd, we say that $\pi$ is an odd cancelable permutations, and denote the set by $\FF^O_{2n,2m}$.
\end{notation}
Clearly, every cancelable partition is either even or odd, thus $\FF_{2n,2m}=\FF^E_{2n,2m}\sqcup \FF^O_{2n,2m}$. We will construct a bijection from $\FF^E_{2n,2m}$ to $\FF^O_{2n,2m}$.

\begin{definition}
\label{def:bijection}
Given $\pi\in\FF^E_{2n,2m}$ with $f_\pi=2i$, let $j\geq i$ be such that $2j$ is the largest even number appearing before $2i-1$ in the list $L_\pi$. Namely, the list looks as follows:
$$L_\pi= \dots, 2i,\dots,2j,\dots, 2i-1,\dots, 2j+2, \dots.$$
We construct the auxiliary permutation $\tau\in S_{2n+2m}$ by letting
\begin{align*}
\tau(2i-1)&=2j, \\ 
\tau(2i)&=2j-1,\\
\tau(t)&=t-2, \qquad \text{whenever }t\in \{2i+1,\dots,2j\} \\
\tau(t)&=t, \qquad \qquad  \text{whenever }t\in \{1,2,\dots, 2n+2m\} \setminus \{2i-1,\dots,2j\}.
\end{align*}

Then, we define the permutation
$$T_\pi:= \tau \pi \tau^{-1}.$$
\end{definition}

\begin{example}
Let us give an example of the map $T_\pi$. Let 
$$\pi=(1,3,8)(2)(4,5,7,16,9)(6)(10,11)(12,13)(14,15)\in S_{\NC}(8,8),$$
whose graph $\GG_\pi$ can be seen in Figure \ref{Figure:Ipi}. In this case $f_\pi=2=2i$ and $j=i$ so that $\tau=(1,2)\in S_{2n+2m}$. Thus,
$$T_\pi=\tau\pi\tau^{-1}=(2,3,8)(1)(4,5,7,16,9)(6)(10,11)(12,13)(14,15)\in S_{\NC}(8,8).$$
\end{example}

\begin{proposition}
\label{prop:properties.T}
For every $\pi\in\FF^E_{2n,2m}$ it holds that $T_\pi\in \FF^O_{2n,2m}$. Moreover, $\pi$ and $T_\pi$ satisfy \ref{condition.1} and \ref{condition.2}.
\end{proposition}

\begin{proof}
There are several little facts that we need to check. Since our approach relies on understanding $IT_\pi$, the order in which we proof these claims might no be the most intuitive. Let us make a list of the facts in the order they are proved.

\begin{enumerate}[label=(F\arabic*)]
\item $T_\pi$ is an odd cancelable permutation. \label{claim1}
\item $Kr_{2n,2m}(T_\pi)$ separates even. \label{claim3}
\item $\#(T_\pi)+\#(Kr_{2n,2m}(T_\pi))=2m+2n.$ \label{claim4}
\item Exists $x\in[2n]$ such that $T_\pi(x)\in[2m]$. \label{claim5}
\item $\GG_{\pi}$ is isomorphic to $\GG_{T_\pi}$. \label{claim6} 
\item $s(T_\pi)=-s(\pi)$. \label{claim7}
\end{enumerate}

First we explain why the proposition follows from these facts. Notice that \ref{claim5} implies that $T_\pi \lor \gamma=1_{2m+2n}$, so using \ref{claim4} we conclude that $T_\pi$ is a non-crossing annular permutation. On the other hand, since $\pi\in \JJ_{2n,2m}$, then $\GG_{\pi}$ is bipartite. Thus \ref{claim6} guarantees that $\GG_{T_\pi}$ is bipartite too. The previous claims together with \ref{claim3} yield that $T_\pi\in \JJ_{2n,2m}$. By \ref{claim1} we conclude that $T_\pi\in\FF^O_{2n,2m}$. Finally, \ref{claim6} and \ref{claim7} imply that $\pi$ and $T_\pi$ satisfy \ref{condition.1} and \ref{condition.2}.

Now, we proceed to prove these 7 facts. We begin by understanding the permutation $IT_\pi$. First notice that $\tau I=I\tau$. Indeed, this follows form a direct computation
    \begin{align*}
        \tau I(2i)&=2j=I\tau(2i),  \\
        \tau I(2i-1)&=2j-1=I\tau(2i-1),\\
        \tau I(2c)&=2c-3=I\tau(2c), && \text{for }2c\in \{2i+1,\dots,2j\},\\
        \tau I(2c-1)&=2c-2=I\tau(2c-1), && \text{for }2c-1\in \{2i+1,\dots,2j\},\\
        \tau I(c)&=I(c)=I\tau(c), && \text{for }c\in \{1,2,\dots, 2n+2m\} \setminus \{2i-1,\dots,2j\}.
        \end{align*}
As a consequence, we obtain that 
\begin{equation}
\label{eq:conjugate}
IT_\pi=I \tau \pi \tau^{-1}=\tau I\pi \tau^{-1}
\end{equation} is the conjugate of $I\pi$ by $\tau$.
    
Recall from Proposition \ref{prop:I.pi} that 
\begin{equation}
\label{eq:IP.form}
I\pi=C^{\text{out}}C^{\text{inn}}O_1 \dots O_k,
\end{equation}
where the $O_r$ contain only odd elements and coincide with the blocks of $\sigma:=(Kr_{2n,2m}(\pi))^{-1}=\gamma^{-1}\pi$, while the $C^{\text{out}}$ and $C^{\text{inn}}$ are union of blocks of $\sigma$ that contain even numbers.
In particular, since $2i-1$ sits between $2j$ and $2j+2$ in the list $L_\pi$, then we know that $2i-1$ and $2j$ are in the same cycle of $\sigma$. This means that every other cycle $C$ of $\sigma$ satisfies that either $C\subset \{2i,2i+1,\dots,2j-1\}$ or $C\subset \{1,2,\dots, 2n+2m\} \setminus \{2i-1,\dots,2j\}$.  Since $\tau$ is the identity in $\{1,2,\dots, 2n+2m\} \setminus \{2i-1,\dots,2j\}$, in the latter case we obtain that $\tau(C)=C$. Furthermore, from the topological interpretation of $Kr_{2n,2m}(\pi)$, see Figure \ref{Fig:Kreweras_second_order},we also know that $\pi(\{2i,\dots,2j\}) =\{2i,\dots,2j\} \subset [2n]$.

From equations \eqref{eq:conjugate} and \eqref{eq:IP.form} we obtain that 
$$IT_\pi=D^{\text{out}}D^{\text{inn}}O'_1 \dots O'_k,$$
where $D^{\text{out}}:=\tau (C^{\text{out}} )$, $D^{\text{inn}}:=\tau (C^{\text{inn}})$ and $O'_r:=\tau (O_r)$ for $r=1,\dots,k$. Now we analyze each of these cycles using our knowledge of $I\pi$. Let us assume without loss of generality that $f_\pi\in C^{\text{out}}$ (the case $f_\pi\in C^{\text{inn}}$ is analogous).

First, since $C^{\text{inn}}$ is the union of cycles of $\sigma$ containing even numbers of the form $2k>2n$, then we must have $C^{\text{inn}}\subset \{1,2,\dots, 2n+2m\} \setminus \{2i-1,\dots,2j\}$ and we obtain
$D^{\text{inn}}=C^{\text{inn}}$.

On the other hand, for the cycles $O_r$ with only odd elements we have two cases: 
\begin{itemize}
\item Either $O_r\subset [2n+2m]\setminus \{2i-1,\dots,2j\}$, and then $O'_r =O_r$,
\item or $O_r\subset \{2i+1,\dots,2j\}$, and then $ O'_r =O_r-2$, still has the same size and only odd elements.
\end{itemize}

Finally, the cycle $C^{\text{out}}$ of $I\pi$ is of the form
$$C^{\text{out}}:=(\mathbf{a}_1,2i-2, \mathbf{o}_1,2i, \mathbf{o}_{2},2i+2,\mathbf{a}_2,2j,\mathbf{o}_3,2i-1,\mathbf{o}_4,2j+2,\mathbf{a}_3),$$
where $\mathbf{o}_1,\mathbf{o}_2,\mathbf{o}_3,\mathbf{o}_4$ are strings of odd numbers, while $\mathbf{a}_1,\mathbf{a}_2,\mathbf{a}_3$ are strings of numbers. Therefore
$$D^{\text{out}}=(\mathbf{a}_1,2i-2, \mathbf{o}_1,2j-1, \mathbf{o}_{2}-2,2i,\mathbf{a}_2-2,2j-2,\mathbf{o}_3-2,2j,\mathbf{o}_4,2j+2,\mathbf{a}_3),$$
where for a string of numbers $\mathbf{b}:=b_1,b_2,\dots,b_l$, we denote $\mathbf{b}-2:=b_1-2,b_2-2,\dots,b_l-2$. Notice that the structure of the even numbers is preserved here, in the sense that $2,4,\dots, 2n$ still appear in order in $D^{\text{out}}$. Moreover since since the pairs $(2r-1,2r)$ are preserved by $\tau$, the smallest flexible pair of $T_\pi$ is now $(2j-1,2j)$ with $f_{T_\pi}=2j-1$. Thus $T_\pi$ is an odd cancelable permutation, proving \ref{claim1}. 

Overall, we conclude that 
$$IT_\pi=D^{\text{out}}D^{\text{inn}}O'_1 \dots O'_k,$$
where the $O_1,\dots, O_k$ contain only odd elements while $D^{\text{out}}$ contains $2,4,\dots, 2n$ in that order and $D^{\text{inn}}$ contains $2n+2,2n+4,\dots, 2n+2m$ in that order. This formula has several implications.

First, by reverse engineering Proposition \ref{prop:I.pi}, we get that
$$(Kr_{2n,2m}(T_\pi))^{-1}=\gamma^{-1}T_\pi=(\gamma^{-1}I)(IT_\pi)$$
separates even, and thus $Kr_{2n,2m}(T_\pi)$ separates even. This proves \ref{claim3}.

Secondly, by the form of $IT_\pi$ we can check that
$$\#(Kr_{2n,2m}(T_\pi))=n+m+k=\#(Kr_{2n,2m}(\pi)).$$
On the other hand, we also know that $\#(T_\pi)=\#(\tau \pi \tau^{-1})=\#(\pi)$. Thus, \ref{claim4} follows from the fact that $\pi$ is a non-crossing annular permutation:
$$\#(T_\pi)+\#(Kr_{2n,2m}(T_\pi))=\#(\pi)+\#(Kr_{2n,2m}(\pi))=2m+2n.$$

To check that there is an $x\in[2n]$ such that $T_\pi(x)\in[2m]$ we also use the fact that $\pi$ is a non-crossing annular permutation, and in particular there exists $x'\in[2n]$ such that $\pi(x')\in[2m]$. We separate in two cases:
\begin{itemize}
    \item If $x'=2i-1$, then we directly have $T_\pi(2j)=\tau(\pi(x'))\in [2m]$. 
    \item If $x'\neq 2i-1$, then $x'\notin \{2i-1,\dots,2j\}$, because $\pi(\{2i,\dots,2j\}) =\{2i,\dots,2j\} \subset [2n]$. We conclude that $T_\pi(x')=\pi(x')\in[2m]$.
\end{itemize}
Thus \ref{claim5} is proved.

Now, to corroborate \ref{claim6} we provide a bijection from $\GG_{\pi}$ to $\GG_{T_\pi}$ that preserves the edges. Let $\Phi$ be a map from the cycles of $\pi$ to the cycles of $T_\pi$ such that the cycle $C=(c_1,\dots,c_r)$ of $\pi$  is mapped to the cycle $\Phi(C):=(\tau(c_1),\dots, \tau(c_n))$ of $T_\pi$.

To check that $\Phi$ preserves edges, recall that two cycles $C$ and $D$ are connected if and only if each contain one element of the pair $2i-1,2i$ for some $1\leq i\leq n+m$. Using that $I\tau=\tau I$ we then notice that $\tau(2i-1)=\tau(I(2i))=I(\tau(2i))$, then each of $\Phi(C)$ and $\Phi(D)$ contain one element of the pair $\tau(2i)$ and $I(\tau(2i))$ that by definition of $I$ is of the form $2j-1,2j$ for some $1\leq j\leq n+m$. Thus $\Phi$ is the desired bijection.
Notice in particular that if  $\pi = \pi' \sqcup \pi''$ is a bipartition of $\pi$, then $T_\pi = \Phi(\pi') \sqcup \Phi(\pi'')$ is a bipartition of $T_\pi$, so $\pi$ and $T_\pi$ satisfy \ref{condition.1}.

Finally, to prove that the two permutations have opposite signs, we notice that the permutation $\tau$ preserves the parity for all the elements except for $2i-1$ and $2i$, where the parity changes. Thus, if we let $C$ be the cycle of $\pi$ containing $2i$, and $D$ the cycle containing $2i-1$, then $C$ and $D$ are connected by an edge and thus in different parts of the bipartition $\pi = \pi' \sqcup \pi''$. Assume that 
$\pi':=\{C,C_1,\dots,C_{r}\}$ and $\pi'':=\{D,D_{1},\dots,D_{s}\}$. By comparing $\pi'$ with the corresponding part $\Phi(\pi')=\{\Phi(C),\Phi(C_1),\dots,\Phi(C_{r})\}$ of $T_\pi$, we notice that
$$|\{ k\in C_s:\, k \text{ is even}\}|=|\{ k\in \Phi(C_s):\, k \text{ is even}\}| \qquad \text{ for }s=1,\dots,r,$$
whereas 
$$|\{ k\in C:\, k \text{ is even}\}|=|\{ k\in \Phi(C):\, k \text{ is even}\}|+1.$$
We conclude that the number of even elements in 
$$A(\varepsilon_\pi):=\cup\{C,C_1,\dots,C_{r}\}\qquad \text{and} \qquad A(\varepsilon_{T_\pi})=\cup\{\Phi(C),\Phi(C_1),\dots,\Phi(C_{r})\}$$
differs in 1. So we conclude that
$$s(T_\pi):=(-1)^{|\{ k\in A(\varepsilon_{T_\pi}):\, k \text{ is even}\}|}=(-1)^{|\{ k\in A(\varepsilon_\pi):\, k \text{ is even}\}|+1}=-s(\pi),$$
which is \ref{claim7}.
\end{proof}

\begin{remark}
The map $T$ is actually a bijection from $\FF^E_{2n,2m}$ to $\FF^O_{2n,2m}$. One way to check this is by looking at its inverse $U$. Assume $\delta\in \FF^O_{2n,2m}$ is an odd cancelable permutation with $f_{\delta}=2j-1 <2n$, and let $2i$ be the first even number after $2j-1$ in the list $L_\delta$, then the list is of the from
$$L_\delta= \dots, 2i-2,\dots,2j-1,\dots, 2i,\dots, 2j, \dots.$$
Define $\tau'\in S_{2n+2m}$ by letting
\begin{align*}
\tau(2j-1)&=2i, \\ 
\tau(2j)&=2i-1,\\
\tau(t)&=t+2, \qquad \text{whenever }t\in \{2i-1,\dots,2j-2\} \\
\tau(t)&=t, \qquad \qquad  \text{whenever }t\in \{1,2,\dots, 2n+2m\} \setminus \{2i-1,\dots,2j\}.
\end{align*}
Then, define the map 
$$U_\delta:= \tau \delta \tau^{-1}$$
Following the proof of Proposition \ref{prop:properties.T}, one check that $U$ maps $\FF^O_{2n,2m}$ to $\FF^E_{2n,2m}$.

Moreover, $U$ is the inverse of $T$. This follows from noticing that if $T_\pi=\delta$ then $\tau_{\delta}= \tau_{\pi}^{-1}$ or equivalently, if $U_\delta=\pi$ then $\tau_{\delta}= \tau_{\pi}^{-1}$. 
\end{remark}

\subsection{Cancellation in the second sum.}
\label{ssec:cancel.sum2}

Let us briefly recall the results from \cite[Section 3]{nica1998commutators} used to commutator formula in the first order. The key fact is that they prove the existence of an involution $\Psi$ in the set of non-crossing partitions of $\{1,\dots,n\}$ which have at least one block with an odd number of elements. Such involution $\Psi$ is explicitly assigning to each block $V$ of a partition $\cV$ a corresponding block of the same size $\Psi(V)$ in $\Psi(\cV)$. More over, $\Psi$ changes the sign associated to the partition $\cV$, where the sign is defined analogously to ours sign $s(\pi)$ and keep tracks of the signs of the terms in the sum of the commutator. By pairing $\cV$ with $\Psi(\cV)$ Nica and Speicher managed to canceled all the terms in a sum indexed by this type of partitions (having at least one block of odd size). Thus greatly simplifying the sum. 

By considering the canonical permutation $\pi\in\NC(n)$ in the disk associated to the partition one can easily extrapolate such involution to permutations. Specifically, denoting by $\NC\OO(2n)$ the set of non-crossing permutations in the disk which have at least one cycle with an odd number of elements, there exists an involution $\Psi:\NC\OO(2n)\to\NC\OO(2n)$ such that
\begin{itemize}
    \item $\Psi(\pi)$ has the same cycle type as $\pi$.
    \item $s(\pi)=-s(\Psi(\pi))$.
\end{itemize}
Moreover, to every cycle $C$ of $\pi$ we can associate a unique cycle $\Psi(C)$ of $\Psi(C)$ with the same size. 

We can further extend $\Psi$ to an involution $\Phi$ in the set of $\NCac\OO_{2n,2m}$ of permutations $\pi\in\NCac_{2n,2m}$ such that $\pi$ has at least one cycle of odd size. The new involution $\Phi$ will allow us to cancel several terms in the second sum of Theorem \ref{Thm.commutator.main}.

First notice that the set of partitioned permutations $(\cU,\pi) \in \NCac_{2n,2m} \subset S_{\NC}^\prime(2n,2m)$ can be alternatively described by quadruples $(\pi_1,\pi_2, C_1,C_2)$ where $\pi_1\in \NC(2n)$, $\pi_2\in \NC(2m)$, $C_1$ is a block of $\pi_1$ and $C_2$ is a block of $\pi_2$. Given such quadruple, then $\pi$ is retrieved as $\pi=\pi_1\sqcup \pi_2$ and $\cU$ is the partition corresponding to $\pi$ after merging $C_1$ and $C_2$. 

We now define the map $\Phi:\NCac\OO_{2n,2m} \to \NCac\OO_{2n,2m}$ as follows:
\begin{itemize}
\item If $\pi_1$ has a cycle of odd length, then $(\pi_1,\pi_2, C_1,C_2)\mapsto (\Psi(\pi_1),\pi_2, \Psi(C_1),C_2)$.
\item Otherwise, $\pi_2$ must have a cycle of odd length and then $(\pi_1,\pi_2, C_1,C_2)\mapsto(\pi_1,\Psi(\pi_2), C_1,\Psi(C_2))$. 
\end{itemize}

The fact that $\Phi$ is an involution follows easily from the fact that $\Psi$ is an involution. Following the same procedure as in the first sum, we can check that $\Phi$ preserves the graph structure while changes the sign. Thus, if $\pi$ has the natural decomposition $\pi =\pi'_1\sqcup \pi''_1 \sqcup A_1\sqcup B_1$ and $\Phi(\pi)$ has the natural decomposition $\Phi(\pi) =\pi'_2\sqcup \pi''_2 \sqcup A_2\sqcup B_2$, then $|A_1|=|A_2|$, $|B_1|=|B_2|$, $\freec{\pi'_1}{a}=\freec{\pi'_2}{a}$, $\freec{\pi''_1}{b}=\freec{\pi''_2}{b}$ and $s(\pi)=-s(\Phi(\pi) )$. So we conclude that the two terms in the sum cancel each other.
\begin{equation}
\label{eq:sum2.cancellation}
s(\pi) \freec{|A_1|,|B_1|}{a}\, \freec{\pi'_1}{a} \, \freec{\pi''_1}{b}+s(\Phi(\pi) ) \freec{|A_2|,|B_2|}{a}\, \freec{\pi'_2}{a} \, \freec{\pi''_2}{b}=0. 
\end{equation}

By canceling all this terms we end up with a sum indexed only by the set $\NCac\EE_{2n,2m}$ of permutations $\pi\in\NCac_{2n,2m}$ such that every cycle of $\pi$ has even size.

\subsection{Proof of main formula for the commutator.}

We are now ready to proof the main result of this section, Theorem \ref{Thm.commutator.main.2.intro} that gives a simplified formula to compute the second order cumulants of the commutator $ab-ba$ in terms of the cumulants of $a$ and $b$.

\begin{proof}[Proof of Theorem \ref{Thm.commutator.main.2.intro}]
By Theorem \ref{Thm.commutator.main}, we know that 
\begin{align*}
\freec{n,m}{ab-ba} =& \sum_{\substack{\pi \in \JJ_{2n,2m}\\ \pi=\pi' \sqcup \pi''}}s(\pi) \left(\freec{\pi'}{a} \, \freec{\pi''}{b}  + (-1)^{m+n} \freec{\pi'}{b} \, \freec{\pi''}{a}  \right) \\ & + \sum_{\substack{(\cU,\pi)\in \NCac_{2n,2m}\\ \pi =\pi'\sqcup \pi'' \sqcup A\sqcup B }} s(\pi) \left(\freec{|A|,|B|}{a}\, \freec{\pi'}{a} \, \freec{\pi''}{b} + (-1)^{m+n} \freec{|A|,|B|}{b}\, \freec{\pi'}{b} \, \freec{\pi''}{a}  \right).
\end{align*}

In the first sum, using the bijection $T$ from Definition \ref{def:bijection}, we can use 

Proposition \ref{prop:properties.T} to pair each even cancelable permutation $\pi\in\FF^E_{2n,2m}$ with an odd cancelable permutations $T_\pi\in \FF^O_{2n,2m}$, such that $\pi$ and $T_\pi$ satisfy \ref{condition.1} and \ref{condition.2}. By Equation \eqref{eq:first.sum.cancel}, this implies that the corresponding terms in the sum will cancel. Thus, instead of using $\JJ_{2n,2m}$ we can index the first sum by the smaller set $\JJ_{2n,2m} \setminus \FF_{2n,2m}=\AA_{2n,2m}$ of admissible partitions introduced in Definition \ref{def:admissible}:
$$
\sum_{\substack{\pi \in \JJ_{2n,2m}\\ \pi=\pi' \sqcup \pi''}}s(\pi) \left(\freec{\pi'}{a} \, \freec{\pi''}{b}  + (-1)^{m+n} \freec{\pi'}{b} \, \freec{\pi''}{a}  \right)=\sum_{\substack{\pi \in \AA_{2n,2m}\\ \pi=\pi' \sqcup \pi''}}s(\pi) \left(\freec{\pi'}{a} \, \freec{\pi''}{b}  + (-1)^{m+n} \freec{\pi'}{b} \, \freec{\pi''}{a}  \right).$$

Similarly for the second sum, we can use the involution $\Phi$ from Section \ref{ssec:cancel.sum2} to cancel permutations in $\NCac\OO_{2n,2m}$, which have at least one cycle of odd size. Thus, instead of using $\NCac_{2n,2m}$ we can index the first sum by the smaller set $\NCac_{2n,2m} \setminus \NCac\OO_{2n,2m}=\NCac\EE_{2n,2m}$ of permutations $\pi$ such that every cycle of $\pi$ has even size. Finally, one can check that $\NCac\EE_{2n,2m}$ is non-empty only when both $n$ and $m$ are even. Furthermore, the permutations in this set are of a very structured form, in the sense that given $\pi\in\NCac\EE_{2n,2m}$, the two tuples that contribute (from Notation \ref{Notation: Deiniftion of two epsilons}) are precisely $\varepsilon_\pi=(1,*,1,*,\dots ,1,*)$ and $\varepsilon'_\pi=(*,1,*,1,\dots ,*,1)$. In particular, $s(\pi)=1$. Thus the second sum has only positive sign and is cancellation free.
\end{proof}

\subsection{A concrete formula for small \texorpdfstring{$n$}{n} and \texorpdfstring{$m$}{m}.}
\label{ssec:examples.commutator}

In order to exemplify how our formula works, we will compute the second order cumulants for small values of $m$ and $n$. Since the sums are similar to those of the anti-commutator, but now with some possible signs, we will make use of our computations from Section \ref{ssec:concrete}.\\

For the case $m=n=1$, one can check that $\JJ_{2,2}=\AA_{2,2}$, as the unique permutation in the set, $\pi=(1,3)(2,4)$, is admissible. On the other hand, since $n$ is odd, the second sum vanishes. Thus, we conclude that given second order free random variables $a$ and $b$, the $(1,1)$-cumulant of their commutator is
$$\freec{1,1}{ab-ba} = 2\freec{2}{a}\freec{2}{b}.$$

For the case $n=2$, $m=1$, one can check that out of the 14 permutations in $\JJ_{4,2}$, 12 have flexible edges. Thus, there are only two admissible partitions:
$$(1,6,4)(2,3,5) \qquad \text{and}\qquad (1,5,4)(2,3,6). $$
By looking at the signs we obtain that the first sum is
$$2\freec{3}{a} \, \freec{3}{b} -2\freec{3}{a} \, \freec{3}{b}=0,$$
Since $m$ is odd, the second sum also vanishes, so we conclude that
$$\freec{2,1}{ab-ba} = \freec{1,2}{ab-ba} =0.$$
Notice, that this is already an instance where our formula is not cancellation free.\\

Lett us briefly mention the case $n=m=2$. Using the very structured form of $I\pi$ one can find that there are 20 admissible permutations in $\AA_{4,4}$. On the other $\NCac\EE_{4,4}$ has only four partitioned permutations $(\cV,\pi)$. After checking that all of the terms in the sums have a positive signs, we conclude that  
$$\freec{2,2}{ab-ba} = 4\freec{4}{a}\freec{4}{b}+12(\freec{2}{a}\freec{2}{b})^2 +12\freec{4}{a}(\freec{2}{b})^2+12\freec{4}{b}(\freec{2}{a})^2+4\freec{2,2}{a}(\freec{2}{b})^2+4\freec{2,2}{b}(\freec{2}{a})^2.$$

\begin{remark}
\label{rem:odd.cumulants.appear}
Finally, we would like to highlight and interesting phenomenon appearing in the second order commutator formula. There might be terms in the first sum that are indexed by the permutations containing cycles of odd size. In particular, the second order free commutator does depend on the moments of odd size. This is in contrast to the first order case, in which the sum is indexed only by partitions with all blocks of even size, thus having the important fact that the commutator in the first order only depends on the even moments. For an example of this behavior, one need to check the formula for the (2,4) and (3,3) cumulants. Computing the exact formula might require some machine help, but one can focus on some particular terms that are sure to survive.
For instance, when computing the first sum in the formula for $\freec{3,3}{ab-ba}$ one needs to find admissible partitions in $\AA_{6,6}$. If we only focus on all the admissible permutations that yield a term of the form 
$$(\freec{3}{a})^2(\freec{3}{b})^2.$$
Then we are only interested in permutations $\pi$ with 4 cycles, each of size 3. An example of an admissible permutation of that type is
$$\pi_0=(1,8,6)(2,12,7)(3,10,11)(4,5,9).$$
It is not hard to check that there are in total 9 permutations of this type. To compute them, one may conjugate $\pi_0$ by $\tau=(135)(246)(7)(8)(9)(10)(11)(12)$ or by $\delta=(1)(2)(3)(4)(5)(6)(7,9,11)(8,10,12)$.
One can also notice that the sign assigned to each of these permutations is always negative. Thus we conclude that
$$\freec{3,3}{ab-ba}=-18(\freec{3}{a})^2(\freec{3}{b})^2+c,$$
where $c$ does not have any term of the form $(\freec{3}{a})^2(\freec{3}{b})^2$. 

Similarly, when looking at the same term in the expansion for 
$\freec{4,2}{ab-ba}$, one can find that there 4 admissible permutations of the required type, each obtained by conjugating
$$\pi_0=(1,9,8)(2,3,12)(4,5,11)(6,7,10),$$
by $(1)(2)\dots (8)(9,10,11,12)$. In this case all the signs are positive and we obtain
$$\freec{4,2}{ab-ba}=8(\freec{3}{a})^2(\freec{3}{b})^2+c,$$
where $c$ does not have any term of the form $(\freec{3}{a})^2(\freec{3}{b})^2$. 
\end{remark}
\section{Product}\label{secc:product}

Let us show how our graph methods help us to write the second order free cumulants of the product $ab$ of two free second order variables $a,b$.

\begin{notation}
Recall that we denote $\mathcal{O}:=\{1,\dots,2n+2m-1\}$ and $\mathcal{E}:=\{2,\dots,2n+2m\}$. Given a permutation $\pi\in S_{\NC}(2n,2m)$ we say that
\begin{enumerate}
    \item $\pi$ is parity preserving if for every cycle $V\in\pi$ either $V\subset \mathcal{O}$ or $V\subset \mathcal{E}$. We denote the set of parity preserving permutations $S_{\NC}^{par}(2n,2m)$.
    \item $\pi$ is a non-crossing pairing if every cycle has size $2$. We let $\NC_2(2n,2m)$ to be the set of non-crossing pairings.
    \item We denote by $\NC_2^{non-par}(2n,2m)$ to be the subset of $\NC_2(n,m)$ with no parity preserving cycles.
    \item We let $\NC_2^{non-par}(n)$ the set of all non-crossing pairings $\pi\in \NC_2(2n)$ with no parity preserving cycles.
\end{enumerate}
\end{notation}

\begin{proposition}\label{Proposition: Non paritity preserving pairing implies kreweras is parity preserving}
Let $n,m\geq 1$, and $\pi\in \NC_2^{non-par}(2n,2m)\cup \NC_2^{non-par}(2n)\times \NC_2^{non-par}(2m)$. Then $\gamma\pi$ is parity preserving.
\end{proposition}
\begin{proof}
The proof follows immediately from the fact that both $\pi$ and $\gamma$ map even numbers into odd numbers and viceversa.
\end{proof}

\begin{lemma}\label{Lemma: Second order cumulants of free product.v1}
Let $a,b$ be two second order free random variables, then for every $n,m\geq 1$
\begin{eqnarray*}
\kappa_{n,m}(ab)&=&\sum_{\pi\in \NC_2^{non-par}(2n,2m)}\left(\prod_{\substack{W\in \gamma\pi \\ W\subset \mathcal{O}}}\freec{|W|}{a} \right)\left(\prod_{\substack{W\in \gamma\pi \\ W\subset\mathcal{E}}}\freec{|W|}{b} \right) \\
& + & \sum_{\pi}\sum_{\substack{U\in \gamma\pi \cap [2n] \\ V\in \gamma\pi \cap [2m] \\ U,V\subset\mathcal{O}}}\kappa_{|U|,|V|}(a)\left(\prod_{\substack{W\in \gamma\pi \\ W\subset \mathcal{O} \\ W\neq U,V}}\freec{|W|}{a} \right)\left(\prod_{\substack{W\in \gamma\pi \\ W\subset\mathcal{E}}}\freec{|W|}{b} \right) \\
& + & \sum_{\pi}\sum_{\substack{U\in \gamma\pi \cap [2n] \\ V\in \gamma\pi \cap [2m] \\ U,V\subset\mathcal{E}}}\kappa_{|U|,|V|}(b)\left(\prod_{\substack{W\in \gamma\pi \\ W\subset \mathcal{O}}}\freec{|W|}{a} \right)\left(\prod_{\substack{W\in \gamma\pi \\ W\subset\mathcal{E}\\ W\neq U,V}}\freec{|W|}{b} \right)
\end{eqnarray*}
where the second and third sums are over $\pi=\pi_1\times \pi_2\in \NC_2^{non-par}(2n)\times \NC_2^{non-par}(2m)$. Observe that the blocks of $\gamma\sigma$ are completely contained in either $\mathcal{O}$ or $\mathcal{E}$, thanks to Proposition \ref{Proposition: Non paritity preserving pairing implies kreweras is parity preserving}.
\end{lemma}

We can rephrase Lemma \ref{Lemma: Second order cumulants of free product.v1} in terms of graphs.

\begin{lemma}\label{Lemma: Second order cumulants of free product.v2}
Let $a,b$ be two second order free random variables, then for every $n,m\geq 1$
\begin{eqnarray*}
\kappa_{n,m}(ab)&=&\sum_{\substack{\pi\in S_{\NC}^{par}(2n,2m)\\ \GG_\pi\text{ unicyclic}}}\left(\prod_{\substack{W\in \pi \\ W\subset \mathcal{O}}}\freec{|W|}{a} \right)\left(\prod_{\substack{W\in \pi \\ W\subset\mathcal{E}}}\freec{|W|}{b} \right) \\
& + & \sum_{\substack{\pi=\pi_1\times\pi_2 \\ \GG_{\pi_1}\text{ tree} \\ \GG_{\pi_2}\text{ tree}}}\sum_{\substack{U\in \pi_1 \\ V\in \pi_2 \\ U,V\subset\mathcal{O}}}\kappa_{|U|,|V|}(a)\left(\prod_{\substack{W\in \pi \\ W\subset \mathcal{O} \\ W\neq U,V}}\freec{|W|}{a} \right)\left(\prod_{\substack{W\in\pi \\ W\subset\mathcal{E}}}\freec{|W|}{b} \right) \\
& + & \sum_{\substack{\pi=\pi_1\times\pi_2 \\ \GG_{\pi_1}\text{ tree} \\ \GG_{\pi_2}\text{ tree}}}\sum_{\substack{U\in \pi_1 \\ V\in \pi_2 \\ U,V\subset\mathcal{E}}}\kappa_{|U|,|V|}(b)\left(\prod_{\substack{W\in\pi \\ W\subset \mathcal{O}}}\freec{|W|}{a} \right)\left(\prod_{\substack{W\in\pi \\ W\subset\mathcal{E}\\ W\neq U,V}}\freec{|W|}{b} \right)
\end{eqnarray*}
where the second and third sums are over $\pi=\pi_1\times \pi_2\in \NC^{par}(2n)\times \NC^{par}(2m)$.
\end{lemma}

\begin{proof}[Proof of Lemma \ref{Lemma: Second order cumulants of free product.v1}]
Notice that we can make use of Theorem \ref{Thm.anticommutator.main.2} with the only difference that there is only one $\varepsilon\in \{1,*\}^{n+m}$ given by $\varepsilon=(1,\dots,1)$ because we only have $\kappa_{n,m}(ab,ab,\dots,ab)$. Thus, in the first sum in Equation \eqref{formula.anticommutator.main.2} the permutations $\pi\in S_{\NC}(2n,2m)$ that we consider must be $\pi\in S_{\NC}^{par}(2n,2m)$ so that $\{A(\varepsilon),B(\varepsilon)\}=\{\mathcal{O},\mathcal{E}\}\geq \pi$. Hence, any cycle of $\gamma^{-1}\pi$ alternates even and odd numbers. Since we require that $\pi^{-1}\gamma$ separates even numbers then every cycle of $\pi^{-1}\gamma$ has size $2$ and it consist of an even and an odd number, this is $\rho=\pi^{-1}\gamma\in \NC_2^{non-par}(2n,2m)$. Conversely if $\rho=\pi^{-1}\gamma\in \NC_2^{non-par}(2n,2m)$ then $\pi=\gamma\rho\in S_{\NC}^{par}(2n,2m)$ and $\pi^{-1}\gamma$ separates even. This mean we can re-index the sum over $\rho\in \NC_2^{non-par}(2n,2m)$ and the permutation that corresponds to each $\rho$ is $\pi=\gamma\rho$. This gives the first term in the formula. For the second term, we have $\pi\in \NC(2n)\times \NC(2m)$, again it must be $\pi=\pi_1\times\pi_2\in \NC^{par}(2n)\times \NC^{par}(2m)$. The same argument as before shows $\pi^{-1}\gamma$ has only cycles of size $2$ consisting of one odd and one even number. Thus $\rho=\pi^{-1}\gamma\in \NC_2^{non-par}(2n)\times \NC_2^{non-par}(2m)$. From here the same argument shows that we can index our sum in terms of $\rho$. Finally we choose one cycle from each $\pi_1$ and $\pi_2$ to make a block of $\cU$. Since we require $\{A(\varepsilon),B(\varepsilon)\}=\{\mathcal{O},\mathcal{E}\}\geq \cU$ then we must choose two cycles with the same parity. There are two possible options, each corresponding to the second and third terms in the right hand side of the formula. 
\end{proof}

\begin{proof}[Proof of Lemma \ref{Lemma: Second order cumulants of free product.v2}]
This proof is similar to the previous. The first sum corresponding to $\pi\in S_{\NC}(2n,2m)$ must be such that $\pi\in S_{\NC}^{par}(2n,2m)$. The same argument as before shows $\#(\pi^{-1}\gamma)=n+m$ and therefore $\#(\pi)=n+m$, thus $\GG_\pi$ is a connected graph with the same number of edges and vertices, namely $\GG_\pi$ consist of one simple cycle. Conversely if $\GG_\pi$ is a cycle, then $\#(\pi)=n+m$ from where $\#(\pi^{-1}\gamma)=n+m$. Since $\pi\in S_{\NC}^{par}(2n,2m)$, then any cycle of $\pi^{-1}\gamma$ alternates even and odd numbers, which means that every cycle has at least size $2$ and therefore $\pi^{-1}\gamma$ has $n+m$ cycles each of size $2$, thus $\pi^{-1}\gamma\in \NC_2^{non-par}(2n,2m)$. So we can rewrite the first sum in Lemma \ref{Lemma: Second order cumulants of free product.v1} in terms of the graphs $\GG_\pi$ that have one cycle. For the second and third sums we proceed similarly. Now $\pi\in \NC^{par}(2n)\times \NC^{par}(2m)$ so $\#(\pi)=n+m+2$, further $\#(\pi_1)=n+1$ because $\#(\pi_1^{-1}\gamma|[2n])=n$ from where we conclude $\GG_{\pi_1}$ must be a tree as it has $n+1$ vertices, $n$ edges, and it is connected (because $\pi_1^{-1}\gamma|[2n]$ separates even). Conversely if $\GG_{\pi_1}$ is a tree then $\#(\pi_1)=n+1$, hence $\#(\pi_1^{-1}\gamma|[2n])=n$, as before this implies $\pi_1^{-1}\gamma\in \NC_2^{non-par}(2n)$. The same applies to $\GG_{\pi_2}$. Finally we can choose one cycle from each $\pi_1$ and $\pi_2$ that has the same parity, yielding the second and third terms in the formula.
\end{proof}

One of the advantages this approach with graphs is that it greatly simplifies the computations in some cases. For instance when we consider centered variables.

\begin{corollary}\label{Corollary: Product of centered second order free variables}
Let $a,b$ be two centered second order free variables, that is $\kappa_1(a)=\kappa_1(b)=0$. Then, if $n,m\geq 1$ with $n\neq m$ we get $\kappa_{n,m}(ab)=0$. Otherwise,
\begin{equation}
\kappa_{n,n}(ab)=n[\kappa_{2}(a)\kappa_2(b)]^n\qquad \text{for } n\geq 1.
\end{equation}
\end{corollary}
\begin{proof}
We first apply Lemma \ref{Lemma: Second order cumulants of free product.v2}, and then look at which terms vanish. Notice that if a graph $\GG_\pi$ has a leave (meaning a vertex of degree $1$) then $\pi$ has a block of size $1$, say $U$, and since the variables are centered, then $\kappa_{|U|}(a)=\kappa_{|U|}(b)=0$ and the whole product  corresponding to $\pi$ vanishes. Thus we only keep the terms indexed by $\pi$ such that $\GG_\pi$ has no leaves.  Since the second and third sums are indexed by trees, that must have a leave, then all the terms vanish and we get 0. 
On the other hand, in the first sum we must have a simple cycle. So every cycle of $\pi$ must have size $2$. Further $\pi\in S_{\NC}^{par}(2n,2m)$ and then as proved in the proof of Lemma \ref{Lemma: Second order cumulants of free product.v2} it must be $\pi^{-1}\gammanm \in \NC_2^{non-par}(2n,2m)$, particularly $\pi^{-1}\gammanm$ separates even. Let $(2k,2s)$ be a cycle of $\pi$ then $\pi^{-1}\gamma$ has a cycle of the form
$$(\gamma^{-1}(2k),2s,\pi^{-1}(\gamma(2s)),\cdots).$$
If $\pi^{-1}(\gamma(2s)) \neq \gamma^{-1}(2k)$ then $\pi^{-1}(\gamma(\pi^{-1}(\gamma(2s))))$ is an even number distinct to $2s$ which is not possible. So, it must be $\pi^{-1}(\gamma(2s))=\gamma^{-1}(2k)$, thus $(\gamma(2s),\gamma^{-1}(2k))$ is a cycle of $\pi$. This means that given $(2k,2s)$ a cycle of $\pi$ all its other cycles are determined and hence the permutation $\pi$ is completely determined. Moreover, it must be $n=m$ and there are exactly $n$ possible permutations $\pi$ (such a permutations will appear latter on and will be called spoke diagrams). Finally, for each $\pi$ we get $\kappa_{\pi}(a,b,\dots,a,b)=[\kappa_{2}(a)\kappa_2(b)]^n$, yielding the desired result.
\end{proof}

To conclude, we prove Theorem \ref{Thm: Product of second order free variables} advertised in the Introduction. The result is just an alternative version of Lemmas \ref{Lemma: Second order cumulants of free product.v1} and \ref{Lemma: Second order cumulants of free product.v2}, but in this case it preserves the essence of the formula in the first order \eqref{Equation: free cumulants of product of free variables.intro}, due to Nica and Speicher \cite{nica1996multiplication}. As pointed out in Remark \ref{Remark: Similar result to cumulants of free variables}, our result generalizes that of \cite[Theorem 7.3]{arizmendi2023second}.

\begin{proof}[Proof of Theorem \ref{Thm: Product of second order free variables}]
First, consider $(\cU,\pi)\in S_{\NC}^\prime(2n,2m)$ such that $\pi^{-1}\gammanm$ separates even. In this case, as mentioned in the proof of Lemma \ref{Lemma: Second order cumulants of free product.v2}, we have $\pi=\pi_1\times \pi_2 \in \NC^{par}(2n)\times \NC^{par}(2m)$. Further, $\pi_1^{-1}\gamma_{2n}$ separates even which is equivalent to $\pi_1 \lor I_{2n}=1_{2n}$. Note that $\pi_1$ can be written as $\pi_{\mathcal{O}}\cup \pi_{\mathcal{E}}$ where $\pi_{\mathcal{O}}$ and $\pi_{\mathcal{E}}$ have only cycles contained in the set of odd $\mathcal{O}$ and even $\mathcal{E}$ numbers, respectively. Thus, we can replicate the proof of the formula \eqref{Equation: free cumulants of product of free variables.intro} (see \cite[Theorem 14.4]{NS}) to conclude that $\pi_{\mathcal{O}}$ is the Kreweras complement of $\pi_{\mathcal{E}}$. This means $\pi_1$ can be seen as the union of a partition in $\NC(n)$, which we also denote by $\pi_1$, and its Kreweras complement. The blocks of the partition $\pi_1$ correspond to the blocks contained in $\mathcal{O}$ in the original partition while the blocks of $Kr_n(\pi_1)$ correspond to the blocks of the original partition contained in $\mathcal{E}$. The same arguments apply to $\pi_2$. Finally $\cU$ is the union of two cycles of $\pi$ one from each $\pi_1\cup Kr_n(\pi_1)$ and $\pi_2\cup Kr_m(\pi_2)$. Since $a$ and $b$ are second order free then these cycles must be taken either from the partition or its Kreweras complement, this corresponds to the second and third sums in \eqref{Equation: Second order free cumulants of product of free variables}. 

Now let us consider $\pi\in S_{\NC}(2n,2m)$. Again $\pi= \pi_{\mathcal{O}}\cup \pi_{\mathcal{E}}$. Since $\pi\lor \gammanm=1_{2n+2m}$ there exist a cycle of $\pi$ (which must be a cycle of $\pi_{\mathcal{E}}$) that contains $2u\in [2n]$ and $2v\in [2m]$ and $\pi(2u)=2v$. A standard argument of non-crossing permutations shows that $\pi(2u,2v)\in \NC(\gammanm(2u,2v))$, where by $\NC(\gammanm(2u,2v))$ we mean non-crossing with respect to the permutation $\gammanm(2u,2v)$ which makes sense as $\gammanm(2u,2v)$ has a single cycle. Equivalently
$$\#(\bar{\pi})+\#(\bar{\pi}^{-1}\bar{\gamma})=2n+2m+1,$$
where $\bar{\pi}=\pi(2u,2v)$ and $\bar{\gamma}=\gammanm(2u,2v)$. Further, $\bar{\pi}^{-1}\bar{\gamma}$ separates even because $\pi^{-1}\gamma$ does. Note that $\bar{\pi}=\pi_{\mathcal{O}}\cup \bar{\pi_{\mathcal{E}}}$ where $\bar{\pi_{\mathcal{E}}}=\pi_{\mathcal{E}}(2u,2v)$. Again, we can use the results from the first order case to assert that $\bar{\pi_{\mathcal{E}}}$ is the Kreweras complement of $\pi_{\mathcal{O}}$. Equivalently $\bar{\pi_{\mathcal{E}}}=\pi_{\mathcal{O}}^{-1}\bar{\gamma}|_{\mathcal{O}}$, with the convention that the values of $\bar{\pi_{\mathcal{E}}}$ are relabeled from $\{2,4,\dots,2n+2m\}$ to $\{1,3,\dots,2n+2m-1\}$. Hence, we get that $\bar{\pi_{\mathcal{E}}}(2i)=\pi_{\mathcal{O}}^{-1}\bar{\gamma}|_{\mathcal{O}}(2i-1)+1$ for $i=1,\dots,n+m$, and then
$$(\pi_{\mathcal{E}}(2u,2v))(2i)=\pi_{\mathcal{O}}^{-1}\gamma(2u,2v)|_{\mathcal{O}}(2i-1)+1, \qquad \text{for } i=1,\dots,n+m$$
Notice that $\gamma(2u,2v)|_{\mathcal{O}}=\gamma|_{\mathcal{O}}(2u-1,2v-1)$, so
\begin{equation}\label{Aux: 1}
(\pi_{\mathcal{E}}(2u,2v))(2i)=(\pi_{\mathcal{O}}^{-1}\gamma|_{\mathcal{O}}(2u-1,2v-1))(2i-1)+1, \qquad \text{for } i=1,\dots,n+m.
\end{equation}
If $i=u,v$, this proves
$$\pi_{\mathcal{E}}(2v)=\pi_{\mathcal{O}}^{-1}\gamma|_{\mathcal{O}}(2v-1)+1,\qquad \text{and} \qquad \pi_{\mathcal{E}}(2u)=\pi_{\mathcal{O}}^{-1}\gamma|_{\mathcal{O}}(2u-1)+1.$$
If $i\neq u,v$, then $(\pi_{\mathcal{E}}(2u,2v))(2i)=\pi_{\mathcal{E}}(2i)$ and $(\pi_{\mathcal{O}}^{-1}\gamma|_{\mathcal{O}}(2u-1,2v-1))(2i-1)=\pi_{\mathcal{O}}^{-1}\gamma|_{\mathcal{O}}(2i-1)$. Using Equation \eqref{Aux: 1} we conclude,
$$\pi_{\mathcal{E}}(2i)=\pi_{\mathcal{O}}^{-1}\gamma|_{\mathcal{O}}(2i-1)+1.$$
In any case we obtain that
$$\pi_{\mathcal{E}}(2i)=\pi_{\mathcal{O}}^{-1}\gamma|_{\mathcal{O}}(2i-1)+1,\qquad \text{for } i=1,\dots,n+m,$$
which means $\pi_{\mathcal{E}}=Kr_{n,m}(\pi_{\mathcal{O}})$. The blocks of $\pi_{\mathcal{O}}$ correspond to cumulants of $a$ while the blocks of $\pi_{\mathcal{E}}$ correspond to cumulants of $b$, this concludes our proof.
\end{proof}

\subsection{A concrete formula for small \texorpdfstring{$n$}{n} and \texorpdfstring{$m$}{m}.}
\label{ssec:examples.product}

In order to exemplify how our formula works, we compute the formulas for $\freec{1,1}{ab},\freec{1,2}{ab}$ and $\freec{2,2}{ab}$ of two second order free variables $a,b$.

\subsubsection{Case \texorpdfstring{$n=m=1$}{n=m=1}}

Note that $S_{\NC}(1,1)$ and $\NC(1)$ have the unique permutations $\pi=(1,2)$ and $\pi=(1)$ respectively. In the former case $Kr_{1,1}(\pi)=(1,2)$, while in the latter $Kr(\pi)=(1)(2)$. Thus
$$\freec{1,1}{ab}=\freec{2}{a}\freec{2}{b}+\freec{1,1}{a}\freec{1}{b}\freec{1}{b}+\freec{1,1}{b}\freec{1}{a}\freec{1}{a}.$$

\subsubsection{Case \texorpdfstring{$n=2$, $m=1$}{n=2, m=1}}

First, $S_{\NC}(2,1)$ has the four permutations $(1,2,3),(1,3,2),(1,3)(2),(2,3)(1)$ with Kreweras complements $(2,3)(1),(1,3)(2),(1,2,3),(1,3,2)$ respectively. The set $\NC(2)$ has two permutations $(1,2)$ and $(1)(2)$. Thus

$$\freec{2,1}{ab}=\freec{1,2}{ab}=2\freec{3}{a}\freec{1}{b}\freec{2}{b}+2\freec{1}{a}\freec{2}{a}\freec{3}{b}+\freec{2,1}{a}\freec{1}{b}\freec{1}{b}\freec{1}{b}+2\freec{1}{a}\freec{2}{a}\freec{1}{b}\freec{1,1}{b}+\freec{2,1}{b}\freec{1}{a}\freec{1}{a}\freec{1}{a}+2\freec{1}{b}\freec{2}{b}\freec{1}{a}\freec{1,1}{a}.$$

\subsubsection{Case \texorpdfstring{$n=m=2$}{n=m=2}}

From the $24$ permutations in $S_{4}$ one can check that $18$ are in $S_{\NC}(2,2)$ except for the six permutations $(1)(2)(3)(4),(1)(2)(3,4),(1,2)(3)(4),(1,2)(3,4),(1,3,2,4),(1,4,2,3)$. With some effort, one can compute their Kreweras complement. Further, the set $\NC(2)$ has the two permutations $(1)(2),(1,2)$. Thus

\begin{eqnarray*}
\freec{2,2}{ab}&=&8\freec{1}{a}\freec{3}{a}\freec{1}{b}\freec{3}{b}+2(\freec{2}{a})^2(\freec{2}{b})^2+4(\freec{1}{a})^2\freec{2}{a}\freec{4}{b}+4(\freec{1}{b})^2\freec{2}{b}\freec{4}{a}+4\freec{1,1}{a}(\freec{1}{a})^2(\freec{2}{b})^2+\freec{2,2}{b}(\freec{1}{a})^4 \\
&+& 4\freec{1,1}{b}(\freec{1}{b})^2(\freec{2}{a})^2+\freec{2,2}{a}(\freec{1}{b})^4+4\freec{1,2}{a}\freec{1}{a}\freec{2}{b}(\freec{1}{b})^2+4\freec{1,2}{b}\freec{1}{b}\freec{2}{a}(\freec{1}{a})^2.
\end{eqnarray*}

\section{Application to second order free semicircular variables}\label{secc:examples}

As an application of our results, we will use them with one of the most important distributions in free probability, the semicircular distribution. Let us recall that in a non-commutative probability space, an element $a\in \AA$ in our algebra has \textit{semicircular distribution} if its first order cumulants are all $0$ except $\kappa_2(a)=1$. Further, we say that $a$ has \textit{second order semicircular distribution} if all its first and second order cumulants are $0$ except $\kappa_2=\kappa_{2,2}=1$. As a motivation, these variables appear naturally as the large $N$-limit of Wigner $N\times N$ random matrices, see for instance \cite{male2022joint,mingo2024asymptotic}.

\subsection{Anti-commutator}
\label{ssec:app.anticommutator}

Using two different approaches we will prove that if $a,b$ are two second order free semicircular elements, the second order cumulants of its anti-commutator have a simple expression. First a direct argument using the Kreweras complement and the permutation $I\pi$ studied in Proposition \ref{prop:I.pi}. In the second approach we identify the set of non-crossing annular pairings that contribute to the sum. These pairing actually have a very simple description.

\begin{proof}[Proof of Proposition \ref{Proposition: example of semicircles.intro}]

Applying Theorem \ref{Thm.anticommutator.main.2}, and using that all the cumulants of $a$ and $b$ vanish except for $\freec{2}{a}=\freec{2,2}{a}=\freec{2}{b}=\freec{2,2}{b}=1$, then we know a term in \eqref{formula.anticommutator.main.2.intro} 
will vanish unless all the elements of the product correspond to one of those four cumulants. In other words, the partition permutations must be a pairing, namely all blocks have size $2$. 

 Let us focus in the first sum in \eqref{formula.anticommutator.main.2.intro}.
The only terms surviving are when $\pi$ is a pairing. Moreover, each surviving term is just a product of ones, so contributes 1 to the sum. Thus, for the first sum we just need to count the number of pairings $\pi$ in $\JJ_{2n,2m}$. Since $\pi$ is a non-crossing annular pairing, then $|\pi|=n+m$ and $|Kr_{2n,2m}(\pi)|=2n+2m-n-m=n+m$. Since $\pi\in\JJ_{2n,2m}$ separates even, this means that each of the $m+n$ cycles of $Kr_{2n,2m}(\pi)$ must contain an even element. In particular, no cycle of $Kr_{2n,2m}(\pi)$ has only odd elements. By Proposition  
\ref{prop:I.pi}, this means that $I\pi$ 
$I\pi=C^{\text{out}}C^{\text{inn}}$ consists only of two cycles. Since both $I$ and $\pi$ are pairings, one can easily check that the two cycles determine each other in the following sense:
\begin{equation}
\label{eq:alternative.proof.aux}
C^{\text{out}}=(c_1c_2\dots c_{j}) \qquad \Rightarrow \qquad  C^{\text{inn}}=(I(c_{j})I(c_{j-1})\dots I(c_1)).
\end{equation}

Notice that since these are the only two cycles, this in particular means that both cycles of $I\pi$ have the same size $j=n+m$. Using part 1 of Proposition  
\ref{prop:I.pi}, we conclude that $n+m$ must be even. In other words, if $n+m$ is odd then there are no pairings $\pi$ in $\JJ_{2n,2m}$ so the first sum in \eqref{formula.anticommutator.main.2.intro}. Assuming that $n+m$ is even, we notice that parts 2 and 3 of Proposition  
\ref{prop:I.pi} provide a specific description of the two cycles. $C^{\text{out}}$ contains the elements $2,4,\dots, 2n$ in that order, so by \eqref{eq:alternative.proof.aux} we obtain that $C^{\text{inn}}$ must contain $2n-1,2n-3,\dots,1$ in that order. Similarly $C^{\text{inn}}$ contains $2n+2,2n+4,\dots, 2n+2m$ in order, so $C^{\text{out}}$ must contain $2n+2m-1,2n+2m-3,\dots,2n+1$ in that order.

Thus, to construct $C^{\text{out}}$ we have to interpolate the numbers $2,4,\dots, 2n$ (in that order) with the elements $2n+2m-1,2n+2m-3,\dots,2n+1$ (in that order). To count in how many ways this is possible, let us fix $2$ as the first number in the cycle.  For the remaining $n+m-1$ positions in the cycle we have $\binom{n+m-1}{m}$ ways to choose which $m$ positions are occupied by the odd numbers and which $n-1$ positions are occupied by the even numbers. The actual position of each even number is now fixed, as we started with 2 and they must be in increasing order. Finally we have $m$ ways to choose which is the first odd number appearing in the cycle, once this is settled, the remaining odd numbers should be in decreasing order. This means that there are $m\binom{n+m-1}{m}=\frac{(n+m-1)!}{(n-1)!(m-1)!}$ ways to construct $C^{\text{out}}$. By \eqref{eq:alternative.proof.aux}, once we construct $C^{\text{out}}$, then $C^{\text{inn}}$ is determined.

It is not hard to check that all permutations $ \sigma:=C^{\text{out}}C^{\text{inn}}$ created with the previous method, satisfy that the permutation $\pi:= I\sigma$ actually belongs to $\JJ_{2n,2m}$. The fact that $\pi^{-1} \gammanm$ separates even follows from construction. Also one can check that $\pi$ is a pairing, so all the vertices of $\GG_\pi$ have order 2. Since $\GG_\pi$ is also connected (because $\pi^{-1} \gammanm$ separates even) then $\GG_\pi$ must be a simple cycle of size $n+m$. Thus, it is bipartite when $n+m$ is even.

We conclude that the first sum is equal to 
\begin{equation}
\label{eq.auxi.pairing.1}
\sum_{\substack{\pi \in \JJ_{2n,2m}\\ \pi=\pi' \sqcup \pi''}}\left( \freec{\pi'}{a} \, \freec{\pi''}{b}  + \freec{\pi'}{b} \, \freec{\pi''}{a}  \right)=2\, \frac{(n+m-1)!}{(n-1)!(m-1)!}
\end{equation}
when $n+m$ is even and 0 otherwise.

To conclude the proof we analyze the second sum. For a term to contribute to the sum we require that $\pi =\pi'\sqcup \pi'' \sqcup A\sqcup B $ satisfies that $|A|=|B|=2$ and $\pi'$ and $\pi''$ are both pairings. So $\pi$ is a pairing. Moreover, in order for $\GG_\cU$ to be connected $\pi= \pi_1\times \pi_2\in \NC(2n)\times \NC(2m)$ must be of the form:
$\pi=(1,2n)(2,3)\cdots(2n-2,2n-1)(2n+1,2m)(2n+2,2n+3)\cdots (2m-2,2m-1).$
So $\GG_\pi$ are two simple cycles (one of size $n$ and one of size $m$), and $\GG_\cU$ is obtained by gluing together two vertices, one in each cycle. Thus, for the graph to be bipartite we require that $n$ and $m$ are even. If this is the case, then there are $nm$ ways to choose the two cycles $(A,B)\in \pi_1\times \pi_2=\pi$ that are glued together to form $\cU$. And again, each choice contribution is $2$. So the second sum is equal to
\begin{equation}
\label{eq.auxi.pairing.2}
\sum_{\substack{(\cU,\pi)\in \NCac_{2n,2m}\\ \pi =\pi'\sqcup \pi'' \sqcup A\sqcup B }} \left(\freec{|A|,|B|}{a}\, \freec{\pi'}{a} \, \freec{\pi''}{b} + \freec{|A|,|B|}{b}\, \freec{\pi'}{b} \, \freec{\pi''}{a}  \right)= 2nm  
\end{equation}
when $n$ and $m$ are even, and 0 otherwise.

Putting \eqref{eq.auxi.pairing.1} and \eqref{eq.auxi.pairing.2} together we obtain the desired formula.
\end{proof}

\begin{remark}
\label{rem.anticommutator.semicircular.general}
One can readily generalize the previous proof to allow the four (non-vanishing) cumulants be different from 1. In this case we get that if $a,b\in \AA$  are second order free semicircular variables, then the second order cumulants of their anti-commutator are
$$
\freec{n,m}{ab+ba}= \left\{ \begin{array}{lc} 2\frac{(n+m-1)!}{(n-1)!(m-1)!}(\freec{2}{a}\freec{2}{b})^{\frac{n+m}{2}}  & \text{if } n\text{ and }m\text{ are odd, }\\
2\frac{(n+m-1)!}{(n-1)!(m-1)!}(\freec{2}{a}\freec{2}{b})^{\frac{n+m}{2}} +nm(\freec{2}{a}\freec{2}{b})^{\frac{n+m-4}{2}} (\freec{2,2}{a}\left(\freec{2}{b})^2+\freec{2,2}{b}(\freec{2}{a})^2\right) &  \text{if }n\text{ and }m\text{ are even,}\\
0 &  \text{otherwise.} \end{array} \right.
$$
\end{remark}

We now provide an alternative approach where we identify precisely which non-crossing pairings are counted by the term $\frac{(n+m-1)!}{(n-1)!(m-1)!}$. In simple words, these are pairings with interval cycles of the form $(2s,2s+1)$ and through cycles of the form $(2s,2u)$ or $(2s+1,2u+1)$ such that given a through string the rest of through strings are determined. The non-crossing pairing obtained by removing the interval cycles is usually called a \textit{spoke diagram} and it emerges in other calculations such as the cumulants of the square of a second order semicircular variable \cite[Example 8.6]{arizmendi2023second} and the concept of second order real freeness \cite{redelmeier2014}. 

\begin{notation}
Let $\pi\in S_{\NC}(2n,2m)$ be a non-crossing pairing and recall that we use the notation $\gamma:=(1,\dots,2n)(2n+1,\dots,2n+2m)$. We say that $(u,v)\in \pi$ is parity preserving if $u \equiv v \text{ (mod }2)$. 
\begin{enumerate}
\item We let $\Ssemicircle(2n)$ to be the set of all non-crossing pairings $\pi\in S_{\NC}(2n,2n)$ with only parity preserving cycles, only through strings and such that if $(u,v)\in \pi$ then $(\gamma(u),\gamma^{-1}(v))\in\pi$. 
\item We let $\Ssemicircle(2n,2m)$ to be the set of all non-crossing pairings $\pi\in S_{\NC}(2n,2m)$ whose parity preserving cycles are the through strings, if $(u,v)\in\pi$ is not parity preserving then $u$ is even and $v=\gamma(u)$ and such that if $B$ is the union of all parity preserving cycles of $\pi$ then $\pi|_B \in \Ssemicircle(|B|/2)$.
\end{enumerate}
\end{notation}

\begin{remark}
It is clear $\Ssemicircle(2n)$ has $n$ distinct permutations. For instance, the two elements of $\Ssemicircle(4)$ are depicted in Figure \ref{Figure:SN permutations a}. On the other hand, the elements from $\Ssemicircle(2n,2m)$ can be easily obtained from elements of $\Ssemicircle(2k)$ for some $k\leq \min\{n,m\}$ by inserting pairings of the form $(2s,\gamma(2s))$ to a permutation in $\Ssemicircle(2k)$, see Figure \ref{Figure:SN permutations b} for an example. The permutations $\Ssemicircle(2n)$ are also referred as \textit{spoke diagrams} and appear naturally when computing the second order cumulants of $a^2$ with $a$ having semicircular distribution, e.g see \cite[Example 8.6]{arizmendi2023second}. 
\end{remark}

\begin{figure}[h!]
\begin{center}
\includegraphics[width=0.9\linewidth]{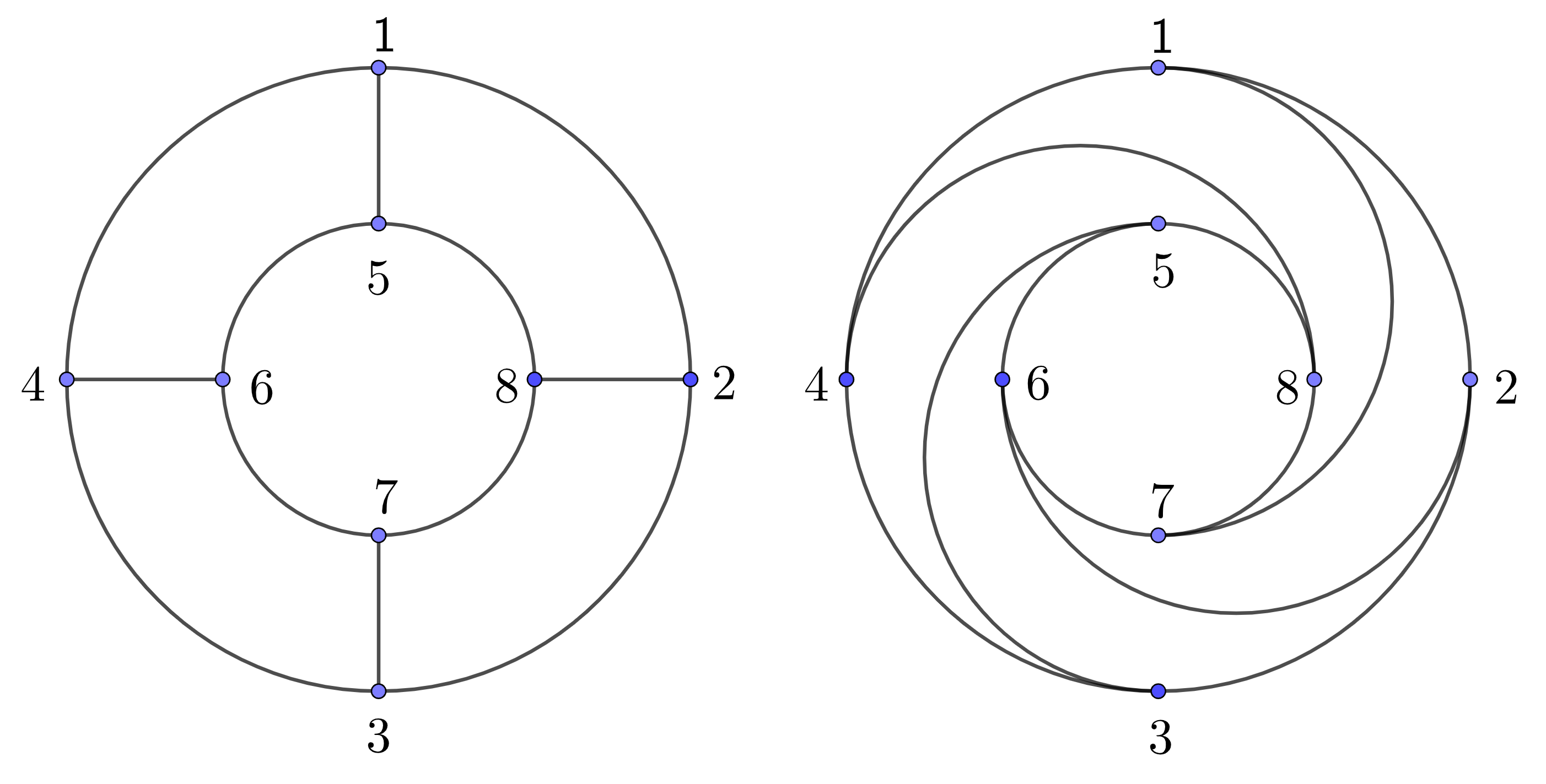} 
\caption{The only two permutations in $\Ssemicircle(4)$.}
    \label{Figure:SN permutations a}   
\end{center}
\end{figure}

\begin{figure}[h!]
\begin{center}
\includegraphics[width=0.9\linewidth]{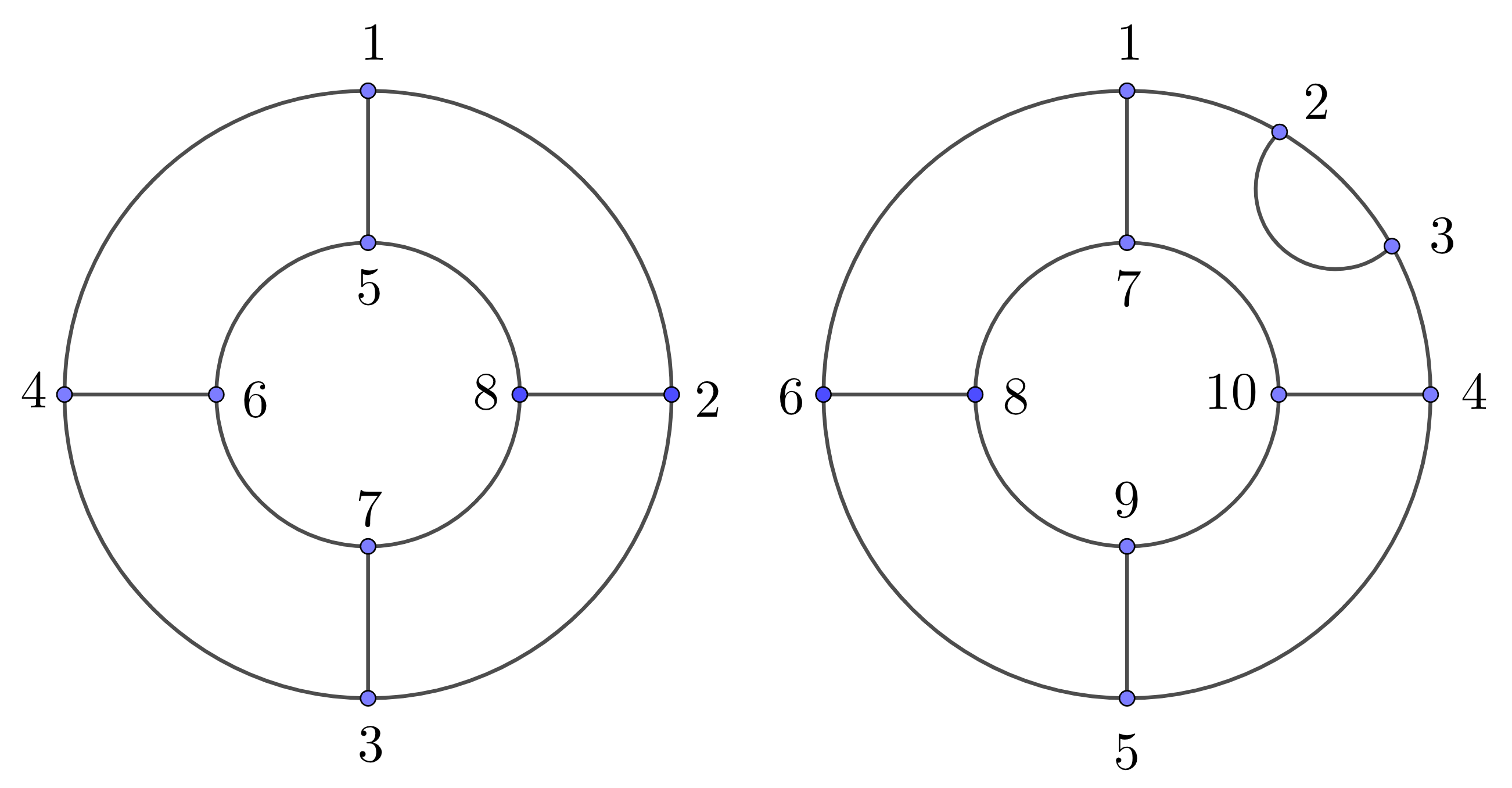}
\caption{Starting from a partition in $\Ssemicircle(4)$ we can construct a partition in $\Ssemicircle(6,4)$ by inserting the interval pairing  $(2,3)$.}
    \label{Figure:SN permutations b}   
\end{center}
\end{figure}

\begin{proposition}\label{Proposition: example of semicircles}
Let $(\AA,\varphi,\varphi^{2})$ be a second order non-commutative probability space and let $a,b\in \AA$ be two second order free semicircular variables. Then the second order cumulants of their anti-commutator are given by
$$
\freec{n,m}{ab+ba}= \left\{ \begin{array}{lc} 2|\Ssemicircle(2n,2m)| & \text{if } n\text{ and }m\text{ are odd }\\
2|\Ssemicircle(2n,2m)|+2nm &  \text{if }n\text{ and }m\text{ are even}\\
0 &  \text{otherwise} \end{array} \right.
$$
\end{proposition}
\begin{proof}
First of all note that all cycles of $\pi$ must be of size $2$. Let us start by assuming all cycles of $\pi$ are parity preserving, this will illustrate the methodology of our proof. It is clear there is no $\pi\in \NC(2n)\times \NC(2m)$ with only blocks of size $2$ and whose elements in the same block have the same parity. Thus the second sum in Equation (\ref{formula.anticommutator.main.2}) is empty so we are reduce to finding $\pi\in S_{\NC}(2n,2m)$ with $\pi\in \JJ_{2n,2m}$. Notice that our graph has $n+m$ edges and vertices, further, any block ${u,v}\in \pi$ that correspond to a vertex of $\GG_\pi$ has exactly two adjacent edges. Hence the graph $\GG_\pi$ is a simple cycle, i.e. a cycle where any vertex has two adjacent edges. Let us recall that our edges are $(1,2),(3,4),\dots, (2n+2m-1,2n+2m)$. In order to construct the graph $\GG_\pi$ we just need to connect the edges along its vertices, moreover the vertices must have the same parity. Since we require our graph to be bipartite we also require $n+m$ to be even. The previous construction guarantees $\GG_\pi$ is connected and bipartite, but we also need to verify $\GG_\pi \in \JJ_{2n,2m}$, i.e. $\pi^{-1}\gammanm$ separates even. Let $(2k,2s)$ be a cycle of $\pi$ then $\pi^{-1}\gamma$ has a cycle of the form
$$(\gamma^{-1}(2k),2s,\pi^{-1}(\gamma(2s)),\cdots).$$
If $\pi^{-1}(\gamma(2s)) \neq \gamma^{-1}(2k)$ then $\pi^{-1}(\gamma(\pi^{-1}(\gamma(2s))))$ is an even number distinct to $2s$ which means $\GG_\pi \notin \JJ_{2n,2m}$.
So it must be $\pi^{-1}(\gamma(2s))=\gamma^{-1}(2k)$, thus $(\gamma(2s),\gamma^{-1}(2k))$ is a cycle of $\pi$. This means that given $(2k,2s)$ a cycle of $\pi$ all its other cycles are determined and hence the graph $\GG_\pi$ is completely determined. Moreover $n=m$ and $\pi\in \Ssemicircle(2n)$. This counts all possible $\pi$ with only parity preserving cycles, but it might be possible that some of the cycles are not parity preserving. We will see that this case is reduced to the parity preserving case. Let us first assume $\pi\in S_{\NC}(2n,2m)$ which correspond to the first summand of Equation (\ref{formula.anticommutator.main.2}). The second summand will then be easily obtained from the same arguments. Suppose $\pi$ has a cycle of the form $(2k,2l+1)$, then $\pi^{-1}\gamma$ has a cycle of the form 
$$(\gamma^{-1}(2l+1),2k,\cdots).$$
Since we require $\GG_\pi\in \JJ_{2n,2m}$, it must be $2k=\gamma^{-1}(2l+1)$, otherwise $\pi^{-1}\gammanm$ does not separate even. This means the cycle of $\pi$ is of the form $(2k,\gamma(2k))$. Thus, all no parity preserving cycles of $\pi$ are of the form $(2k,\gamma(2k))$. Now let us consider a parity preserving cycle. Let $(2k,2s)\in\pi$, we proceed similarly as before so that $\pi^{-1}\gamma$ has a cycle of the form
$$(\gamma^{-1}(2k),2s,\pi^{-1}(\gamma(2s)),\cdots).$$
Since $\GG_\pi\in \mathcal{S}$ the only even number of this cycle must be $2s$, so $\delta=\pi^{-1}(\gamma(2s))$ must be odd and $\pi^{-1}(\gamma(\delta))$ must be odd as well. Note that $\gamma(\delta)$ is even so $(\pi(\gamma(\delta)),\gamma(\delta))$ is a non parity preserving cycle of $\pi$ which means $\pi(\gamma(\delta))=\gamma(\gamma(\delta))=\gamma^2(\delta)$. This means the cycle of $\pi^{-1}\gamma$ that contains $2s$ has the form
$$(\gamma^{-1}(2k),2s,\delta,\gamma^2(\delta),\gamma^4(\delta),\cdots).$$
This proves $\gamma^{-1}(2k)=\gamma^{2p}(\delta)$ for some $p$ and $\pi$ has cycles: $(\gamma(\delta),\gamma^2(\delta)),(\gamma^3(\delta),\gamma^4(\delta)),\dots, \ab (\gamma^{2p-1}(\delta),\ab \gamma^{2p}(\delta))=(\gamma^{-2}(2k),\gamma^{-1}(2k))$ and $(\delta,\gamma(2s))$. Further, it is clear $(2k,2s)$ and $(\delta,\gammanm(2s))$ must be through strings otherwise $\pi$ has no through strings, see for instance figure \ref{Figure:Counting semicircle}. This proves $\pi\in \Ssemicircle(2n,2m)$. 

\begin{figure}[h!]
\begin{center}
\includegraphics[width=0.8\linewidth]{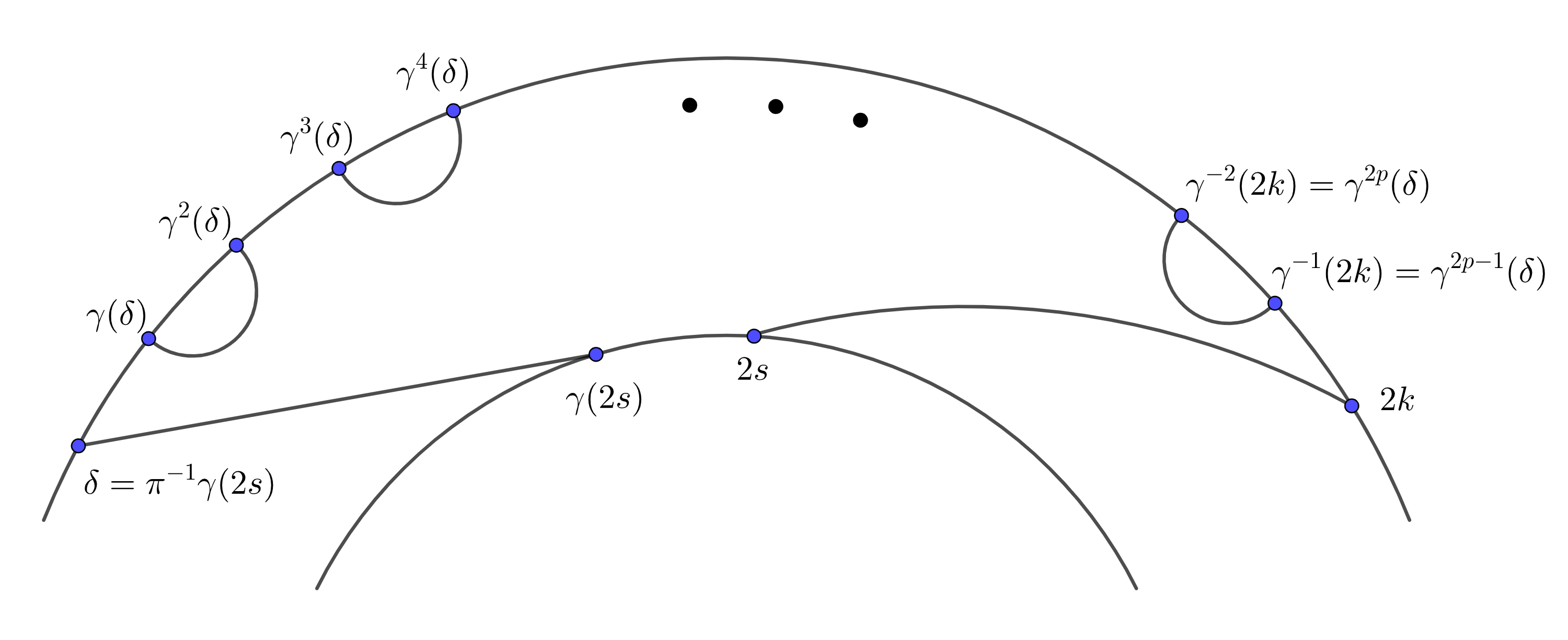}
\caption{}
    \label{Figure:Counting semicircle}   
\end{center}
\end{figure}

It remains to check the second term of Equation (\ref{formula.anticommutator.main.2}). That is $\pi=\pi_1\times \pi_2\in \NC(2n)\times \NC(2m)$. First of all, notice that $n$ and $m$ must be even, otherwise there is no bipartite graphs $\GG_{\pi_1},\GG_{\pi_2}$. Assume $n$ is even and it might have non parity preserving cycles. The same argument shows these cycles are of the form $(2k,\gamma(2k))$. Moreover the same argument as before shows that if $(2k,2s)$ is a cycle of $\pi$ then the cycle of $\pi^{-1}\gamma$ that contains $2s$ has the form
$$(\gamma^{-1}(2k),2s,\delta,\gamma^2(\delta),\gamma^4(\delta),\cdots).$$
Again, this proves $\gamma^{-1}(2k)=\gamma^{2p}(\delta)$ for some $p$ and $(\gamma(\delta),\gamma^2(\delta)),(\gamma^3(\delta),\gamma^4(\delta)),\dots, \ab (\gamma^{2p-1}(\delta),\gamma^{2p}(\delta))=(\gamma^{-2}(2k),\gamma^{-1}(2k))$ are all cycles of $\pi$. Note that there is an even number of elements in the set $\{2k+1,\dots,2s-1\}$, so $\pi$ has necessarily a crossing which means $\pi\notin \NC(2n)\times \NC(2m)$. The latest proves $\pi$ cannot have parity preserving cycles, so all its cycles must be non-parity-preserving which yields a unique choice of $\pi$, namely,
\begin{multline*}
\pi=(2,3)(4,5)\cdots (2n-2,2n-1)(1,2n) \\(2n+2,2n+3)\cdots (2n+2m-2,2n+2m-1)(2n+1,2m).
\end{multline*}
To conclude it is easy to observe that if both $n$ and $m$ are odd then the first term in Equation \eqref{formula.anticommutator.main.2} runs over all set of non-crossing pairings $\Ssemicircle(2n,2m)$ and for each pairing the contribution is $2$ as $\kappa_2(a)=\kappa_2(b)=1$. If both $n$ and $m$ are even then the first term remains as before but the second term is not empty and there is exactly one permutation $\pi=\pi_1\times\pi_2$. The choice for the cycles $(C_1,C_2)\in \pi_1\times \pi_2$ is $nm$ and for each choice the contribution is $2$ as $\kappa_{2,2}(a)=\kappa_{2,2}(b)=\kappa_2(a)=\kappa_2(b)=1$. Finally, if $n+m$ is not even then both sums are empty.
\end{proof}

To finish with this example let us show that the set $\Ssemicircle(2n,2m)$ can explicitly be counted.

\begin{proposition}\label{Proposition: Counting semicircle set}
For any $n,m\geq 1$,
$$|\Ssemicircle(2n,2m)|=\frac{(n+m-1)!}{(n-1)!(m-1)!}.$$
\end{proposition}
\begin{proof}
Set $\gamma=\gammanm$. Let $\pi$ be the permutation with cycles $(2s,\gamma(2s))$ for any $1\leq s\leq n+m$. Any permutation $\sigma$ in $\Ssemicircle(2n,2m)$ can be obtained from $\pi$ by turning some of its cycles into through strings. Note that if $(2s,2u)$ is a through string of $\sigma$ then $\gamma(2s)$ and $\gamma(2u)$ must lie in a through string. This means that we are allowed to turn any cycle of $\pi$ into through strings to get $\sigma$. In the outer circle we have as many as $n$ cycles of $\pi$ while in the inner circle we have $m$. So if $\sigma$ has $2k$ through strings, the number of ways to choose the cycles of $\pi$ is to make them into through strings of $\sigma$ is $\binom{n}{k}\binom{m}{k}$. Once we have a choice we just need to pair any two elements (in distinct circles and with the same parity) to get a through string. By definition of $\Ssemicircle(2n,2m)$ the rest of through strings are determined. This can be done in $k\binom{n}{k}\binom{m}{k}$ distinct ways. Let us assume $n\leq m$, hence
\begin{eqnarray*}
|\Ssemicircle(2n,2m)| &=& \sum_{k=1}^{n}k\binom{n}{k}\binom{m}{k}=\sum_{k=0}^{n-1} (k+1)\binom{n}{k+1}\binom{m}{k+1} \\
&=& m\sum_{k=0}^{n-1} \binom{n}{k+1}\binom{m-1}{k}= m\sum_{k=0}^{n-1} \binom{n}{n-1-k}\binom{m-1}{k}
\end{eqnarray*}
The sum $\sum_{k=0}^{n-1} \binom{n}{n-1-k}\binom{m-1}{k}$ counts the number of ways of choosing $n-1$ objects from a total of $n+m-1$ objects composed of two groups having $n$ and $m-1$ objects of the first and second group respectively. Therefore the sum in the last equality is $\binom{n+m-1}{n-1}$, hence $|\Ssemicircle(2n,2m)|=m\binom{n+m-1}{n-1}=\frac{(n+m-1)!}{(n-1)!(m-1)!}$ as desired. The case $m\leq n$ follows exactly the same.
\end{proof}

Proposition \ref{Proposition: example of semicircles.intro} follows from combining Propositions \ref{Proposition: example of semicircles} and \ref{Proposition: Counting semicircle set}.

\begin{remark}
Unlike the first order case in which any connected graph $\GG_\pi$ contributes, in the second order case Proposition \ref{Proposition: example of semicircles} provides an explicit example in which we observe that a massive set of connected graphs $\GG_\pi$ are not in the sum. For instance, it is very simple to observe that the set of connected and bipartite graphs is counted by $2^{n+m-1}(n+m-1)!$ if $n+m$ is even. This follows directly from noticing that our graphs are the result of gluing the edges through its vertices such that the resulting graph is a simple cycle. However, only a few $\pi$ of these graphs $\GG_\pi$, are further in the set $\Ssemicircle(2n,2m)$. It is evident that the cardinality of this set is much smaller than $2^{n+m-1}(n+m-1)!$.

\end{remark}

\subsection{Commutator}

We now want to compute the cumulants of the commutator of two second order semicircular elements, $a,b$. We will use the permutation $I\pi$ and follow the proof of Proposition \ref{Proposition: example of semicircles.intro} from Section \ref{ssec:app.anticommutator}.

\begin{proof}[Proof of Proposition \ref{prop.commutator.semicircular.intro}]
Recall that formulas for the commutator and anti-commutator have the same terms, what changes is the sign. Thus, we can follow the proof of Proposition \ref{Proposition: example of semicircles.intro} to check that $ I\pi=C^{\text{out}}C^{\text{inn}}$ and $C^{\text{out}}$ must interpolate the numbers $2,4,\dots, 2n$ (in that order) with the elements $2n+2m-1,2n+2m-3,\dots,2n+1$ (in that order). Moreover, there are $\frac{(n+m-1)!}{(n-1)!(m-1)!}$ ways to construct $C^{\text{out}}$, and given that cycle,  $C^{\text{inn}}$ is determined. Also, in order for $\GG_\pi$ to be bipartate we must have that $n+m$ is even. Now, we need to check the sign $s(\pi)$ associated to each permutation constructed in this way. Turns out that the sign is always the same. Indeed, recall from the prof of Proposition \ref{prop:I.pi}, that if the natural bipartite decomposition is $\pi = \pi ' \sqcup \pi ''$, and we denote $A'=\cup\pi'$ and $A''=\cup\pi''$, then
\begin{equation}
I \pi(A') \subset A'', \qquad \text{and} \qquad I \pi(A'')\subset A'.
\end{equation}
In particular, since $2\in A''$ by assumption. The elements of $C^{\text{out}}$ at an odd position in the cycle (starting with 2) are in $A''$ while the elements at an even position are in $A'$. Similarly, the elements of $C^{\text{inn}}$ at an odd position in the cycle (starting with 1) are in $A''$ while the elements at an even position are in $A''$. Let $t:=|\{ k\in A'\cap C^{\text{out}} :\, k \text{ is even}\}|$, since $C^{\text{out}}$ is determined by $C^{\text{inn}}$ it is not hard to check that 
$$|\{ k\in A'\cap C^{\text{inn}} :\, k \text{ is even}\}|=|\{ k\in A''\cap C^{\text{out}} :\, k \text{ is odd}\}|$$
$$=m+n-|\{ k\in A''\cap C^{\text{out}} :\, k \text{ is even}\}|=m+n-(n-t)=m+t.$$
By \eqref{not:sign.pi}, we conclude that the sign
$$
s(\pi)=(-1)^{|\{ k\in A':\, k \text{ is even}\}|}=(-1)^{t+m+t}=(-1)^m.
$$
Notice that it does not depend of $\pi$, moreover, since $m+n$ is even, then $(-1)^{m+n}=1$. Using Theorem \ref{Thm.commutator.main} and the previous analysis we conclude that the first sum is equal to 
\begin{equation}
\label{eq.auxi.pairing.1.comm}
\sum_{\substack{\pi \in \JJ_{2n,2m}\\ \pi=\pi' \sqcup \pi''}} s(\pi)\left( \freec{\pi'}{a} \, \freec{\pi''}{b}  +(-1)^{m+n} \freec{\pi'}{b} \, \freec{\pi''}{a}  \right)=2(-1)^m \, \frac{(n+m-1)!}{(n-1)!(m-1)!}
\end{equation}
when $n+m$ is even and 0 otherwise.

To analyze the second sum, we can also follow the proof of Proposition \ref{Proposition: example of semicircles.intro} to check that to be non-vanishing we require that both $n$ and $m$ are even and $\pi$ must be of the form $\pi=(1,2n)(2,3)\cdots(2n-2,2n-1)(2n+1,2m)(2n+2,2n+3)\cdots (2m-2,2m-1).$ Then one can easily check that $s(\pi)=(-1)^{\frac{m+n}{2}}=1$. Since there are $nm$ permutations in the sum, we conclude that the second sum is equal to
\begin{equation}
\label{eq.auxi.pairing.2.comm}
\sum_{\substack{(\cU,\pi)\in \NCac_{2n,2m}\\ \pi =\pi'\sqcup \pi'' \sqcup A\sqcup B }} s(\pi) \left(\freec{|A|,|B|}{a}\, \freec{\pi'}{a} \, \freec{\pi''}{b} + (-1)^{m+n}\freec{|A|,|B|}{b}\, \freec{\pi'}{b} \, \freec{\pi''}{a}  \right)= 2nm(-1)^{\frac{m+n}{2}}  
\end{equation}
when $n$ and $m$ are even, and 0 otherwise.

Putting \eqref{eq.auxi.pairing.1.comm} and \eqref{eq.auxi.pairing.2.comm} together we obtain the desired formula.
\end{proof}

\begin{remark}
Same as with the anti-commutator, one can readily generalize the previous proof to allow the four (non-vanishing) cumulants be different from 1. We obtain
$$
\freec{n,m}{ab-ba}= \left\{ \begin{array}{lc} -2\frac{(n+m-1)!}{(n-1)!(m-1)!}(\freec{2}{a}\freec{2}{b})^{\frac{n+m}{2}}  & \text{if } n\text{ and }m\text{ are odd, }\\
2\frac{(n+m-1)!}{(n-1)!(m-1)!}(\freec{2}{a}\freec{2}{b})^{\frac{n+m}{2}} & \\
\hspace{.5cm}+(-1)^{\frac{n+m}{2}}nm(\freec{2}{a}\freec{2}{b})^{\frac{n+m-4}{2}} (\freec{2,2}{a}\left(\freec{2}{b})^2+\freec{2,2}{b}(\freec{2}{a})^2\right) &  \text{if }n\text{ and }m\text{ are even,}\\
0 &  \text{otherwise.} \end{array} \right.
$$
\end{remark}

\subsection{Product} 

Let us recall that we are considering free semicircle variables $a,b$, which means free variables with cumulants all $0$ except for $\freec{2}{a},\freec{2}{b},\freec{2,2}{a}$, and $\freec{2,2}{b}$.

\begin{proof}[Proof of Proposition \ref{Proposition: example of semicircles.intro.v2.product}]
First of all note that second and third sums of Theorem \ref{Thm: Product of second order free variables} vanish as any non-crossing pairing $\pi\in \NC(n)\times \NC(m)$ must be such that $Kr(\pi)$ has a block of size $1$. So we are reduce to consider the first sum of Theorem \ref{Thm: Product of second order free variables}. Let $\pi\in S_{\NC}(n,m)$ and suppose $\pi$ has a non-through string $(u,v)$ with $u<v$ and such that both $u,v\in [2n]$. Observe that if $\pi$ has two through strings $(a_1,b_1)$ and $(a_2,b_2)$ with $b_1,b_2\in [2m]$ and $a_1\in \{u+1,\dots,v-1\}$ and $a_2\in [2n]\setminus\{u+1,\dots,v-1\}$ then $\pi$ satisfies the crossing condition $AC-3$ as defined in \cite[Definition 3.5]{mingo2004annular}. The latest means that $\pi$ has no through strings in either $\{u+1,\dots,v-1\}$ or $[2n]\setminus \{u+1,\dots,v-1\}$. Assume without loss of generality it has no through strings in $\{u+1,\dots,v-1\}$. Then $\pi|_{\{u+1,\dots,v-1\}}\in \NC_2(\{u+1,\dots,v-1\})$ and therefore $Kr_{n,m}(\pi)$ has a block of size $1$ in $\{u+1,\dots,v-1\}$ which leads to vanishing in the first sum of Theorem \ref{Thm: Product of second order free variables}. We conclude $\pi$ has only through strings. Since $a,b$ are both semicircular it means both $\pi$ and $Kr_{n,m}(\pi)$ has only cycles of size $2$. From the topological interpretation of $Kr_{n,m}(\pi)$ it is clear the unique permutations that satisfy such conditions are $\NC_2^{spoke}(n)$. Hence $n=m$ and there are exactly $n$ permutations $\pi$ for which each contribution is $\freec{2}{a}\freec{2}{b}$. 
\end{proof}

\section*{Acknowledgments}

The first author was supported by Hong Kong GRF 16304724 and NSFC 12222121. Most of the work was carried out while the second author was affiliated with Texas A\&M University. We thank the organizers of IWOTA 2024, where this project originated. Specially, the second author expresses its gratitude to the NSF travel grant that made it possible to attend IWOTA. We also thank Octavio Arizmendi and James Mingo for their valuable comments.


\begin{thebibliography}{MMPS22}

\bibitem[AM23]{arizmendi2023second}
Octavio Arizmendi and James~A Mingo.
\newblock Second order cumulants: Second order even elements and $ r $-diagonal elements.
\newblock {\em Annales de l’Institut Henri Poincar{\'e} D}, 2023.

\bibitem[Bia97]{biane1997some}
Philippe Biane.
\newblock Some properties of crossings and partitions.
\newblock {\em Discrete Mathematics}, 175(1-3):41--53, 1997.

\bibitem[CM{\'S}S07]{collins2007second}
Beno{\^\i}t Collins, James~A Mingo, Piotr {\'S}niady, and Roland Speicher.
\newblock Second order freeness and fluctuations of random matrices. {III}: Higher order freeness and free cumulants.
\newblock {\em Documenta Mathematica}, 12:1--70, 2007.

\bibitem[KS00]{krawczyk2000combinatorics}
Bernadette Krawczyk and Roland Speicher.
\newblock Combinatorics of free cumulants.
\newblock {\em Journal of Combinatorial Theory, Series A}, 90(2):267--292, 2000.

\bibitem[MG24]{mingo2024asymptotic}
James~A Mingo and Daniel~Munoz George.
\newblock Asymptotic limit of cumulants and higher order free cumulants of complex wigner matrices.
\newblock {\em Preprint arXiv:2407.17608}, 2024.

\bibitem[MM13]{mingopopa2013}
James~A Mingo and Popa Mihai.
\newblock Real second order freeness and haar orthogonal matrices.
\newblock {\em Journal of Mathematical Physics}, 54(5), 2013.

\bibitem[MMPS22]{male2022joint}
Camile Male, James~A Mingo, Sandrine P{\'e}ch{\'e}, and Roland Speicher.
\newblock Joint global fluctuations of complex wigner and deterministic matrices.
\newblock {\em Random Matrices: Theory and Applications}, 11(02):2250015, 2022.

\bibitem[MN04]{mingo2004annular}
James~A Mingo and Alexandru Nica.
\newblock Annular noncrossing permutations and partitions, and second-order asymptotics for random matrices.
\newblock {\em International Mathematics Research Notices}, 2004(28):1413--1460, 2004.

\bibitem[MS06]{mingo2007secondPart1}
James~A Mingo and Roland Speicher.
\newblock Second order freeness and fluctuations of random matrices: I. gaussian and wishart matrices and cyclic fock spaces.
\newblock {\em Journal of Functional Analysis}, 235(1):226--270, 2006.

\bibitem[MS17]{mingo2017free}
James~A Mingo and Roland Speicher.
\newblock {\em Free probability and random matrices}, volume~35.
\newblock Springer, 2017.

\bibitem[M{\'S}S07]{mingo2007second}
James~A Mingo, Piotr {\'S}niady, and Roland Speicher.
\newblock Second order freeness and fluctuations of random matrices: {II}. unitary random matrices.
\newblock {\em Advances in Mathematics}, 209(1):212--240, 2007.

\bibitem[MST09]{mingo2009second}
James Mingo, Roland Speicher, and Edward Tan.
\newblock Second order cumulants of products.
\newblock {\em Transactions of the American Mathematical Society}, 361(9):4751--4781, 2009.

\bibitem[NS96]{nica1996multiplication}
Alexandru Nica and Roland Speicher.
\newblock On the multiplication of free n-tuples of noncommutative random variables.
\newblock {\em American Journal of Mathematics}, pages 799--837, 1996.

\bibitem[NS98]{nica1998commutators}
Alexandru Nica and Roland Speicher.
\newblock Commutators of free random variables.
\newblock {\em Duke Mathematical Journal}, 92(3):553--592, 1998.

\bibitem[NS06]{NS}
Alexandru Nica and Roland Speicher.
\newblock {\em Lectures on the combinatorics of free probability}, volume~13.
\newblock Cambridge University Press, 2006.

\bibitem[Per23]{perales2021anti}
Daniel Perales.
\newblock On the anti-commutator of two free random variables.
\newblock {\em Indiana University Mathematics Journal}, 72:1867--1908, 2023.

\bibitem[Red14]{redelmeier2014}
C.~Emily~I. Redelmeier.
\newblock Real second-order freeness and the asymptotic real second-order freeness of several real matrix models.
\newblock {\em International Mathematics Research Notices}, 2014(12):3353--3395, 2014.

\bibitem[Spe94]{speicher1994multiplicative}
Roland Speicher.
\newblock Multiplicative functions on the lattice of non-crossing partitions and free convolution.
\newblock {\em Mathematische Annalen}, 298(1):611--628, 1994.

\bibitem[VDN92]{voiculescu1992free}
Dan~V Voiculescu, Ken~J Dykema, and Alexandru Nica.
\newblock {\em Free random variables}.
\newblock CRM Monograph Series 1, American Mathematical Society, 1992.

\bibitem[Voi91]{voiculescu1991limit}
Dan Voiculescu.
\newblock Limit laws for random matrices and free products.
\newblock {\em Inventiones mathematicae}, 104(1):201--220, 1991.

\bibitem[Voi98]{voiculescu1998strengthened}
Dan Voiculescu.
\newblock A strengthened asymptotic freeness result for random matrices with applications to free entropy.
\newblock {\em IMRN: International Mathematics Research Notices}, 1998(1), 1998.

\end{thebibliography}
\end{document}